\documentclass[11pt]{amsart}
\usepackage{amsmath}
\usepackage{hyperref} 
\usepackage{amsfonts}
\usepackage{amssymb}
\usepackage{amsthm}
\usepackage[arrow, matrix, curve]{xy} 
\usepackage{hyperref}
\usepackage{amssymb,amscd,url,tikz,stmaryrd,mathtools,fixltx2e}
\usepackage{amssymb,amscd,url}

\usepackage{color}

\let\Sec\S

\newcommand{\A}{{\mathbb A}}
\newcommand{\Z}{{\mathbb Z}}
\renewcommand{\S}{{\mathbb S}}
\newcommand{\R}{{\mathbb R}}
\newcommand{\C}{{\mathbb C}}
\newcommand{\D}{{\mathbb D}}
\newcommand{\N}{{\mathbb N}}

\def\Res{{\,\rm Res}}
\def\Re{{\rm Re}}

\def\Id{{\rm Id}}
\def\ii{{\rm i}}
\def\xx{{\bf x}}
\def\yy{{\bf y}}
\def\sl{\mathfrak{sl}}
\def\su{\mathfrak{su}}

\def\tr{{\rm trace}}
\def\Area{{\rm Area}}
\def\SL{{\rm SL}}
\def\SU{{\rm SU}}
\def\z{\overline{z}}
\renewcommand{\matrix}[1]{\left(\begin{array}{cc} #1\end{array}\right)}
\newcommand{\minimatrix}[1]{\left(\begin{smallmatrix}#1\end{smallmatrix}\right)}
\newcommand{\wt}[1]{\widetilde{#1}}
\newcommand{\wh}[1]{\widehat{#1}}
\newcommand{\cal}[1]{{\mathcal #1}}
\newcommand{\other}[1]{\widefrown{#1}}
\theoremstyle{plain}
\newtheorem{theorem}{Theorem}
\newtheorem{lemma}{Lemma}
\newtheorem{proposition}[lemma]{Proposition}
\newtheorem{remark}[lemma]{Remark}
\newtheorem{corollary}[lemma]{Corollary}

\newtheorem{definition}[lemma]{Definition}

\DeclareFontFamily{U}{mathx}{\hyphenchar\font45}
\DeclareFontShape{U}{mathx}{m}{n}{
      <5> <6> <7> <8> <9> <10>
      <10.95> <12> <14.4> <17.28> <20.74> <24.88>
      mathx10
      }{}
\DeclareSymbolFont{mathx}{U}{mathx}{m}{n}
\DeclareFontSubstitution{U}{mathx}{m}{n}
\DeclareMathAccent{\widecheck}{0}{mathx}{"71}
\DeclareMathAccent{\widetilde}{0}{mathx}{"72}
\DeclareMathAccent{\widebar}{0}{mathx}{"73}
\DeclareMathAccent{\widevec}{0}{mathx}{"74}
\DeclareMathAccent{\widehat}{0}{mathx}{"70}
\DeclareMathAccent{\widefrown}{0}{mathx}{"75}
\DeclareMathAccent{\chinesehat}{0}{mathx}{"69}

\def\Res{{\,\rm Res}}
\def\Re{{\rm Re}}

\def\Id{{\rm Id}}
\def\ii{{\rm i}}
\def\xx{{\bf x}}
\def\sl{\mathfrak{sl}}
\def\su{\mathfrak{su}}

\def\tr{{\rm trace}}
\def\Area{{\rm Area}}
\def\SL{{\rm SL}}
\def\SU{{\rm SU}}
\def\z{\overline{z}}
\renewcommand{\matrix}[1]{\left(\begin{array}{cc} #1\end{array}\right)}

\usepackage{shuffle}
\usepackage{bbm}
\newcommand{\Li}{\operatorname{Li}}
\newcommand{\sgn}{\operatorname{sgn}}
\newcommand{\abs}[1]{|#1|}

\newcommand{\wc}[1]{\widecheck{#1}}

\newcommand{\cv}[1]{\underline{#1}}

\makeatletter\let\@wraptoccontribs\wraptoccontribs\makeatother
\newcommand{\appauthor}{}

\begin{document}

\title[Complete families of embedded high genus CMC surfaces in the 3-sphere ]{Complete families of embedded high genus CMC surfaces in the 3-sphere 
}

 \author{Lynn Heller}
 \author{Sebastian Heller}
 \author{Martin Traizet}

 \contrib[With an appendix by]{Steven Charlton}
 
\noindent
\address{Institut f\"ur Differentialgeometrie\\
Welfengarten 1\\
30167 Hannover\\
Germany} 
\email{lynn.heller@math.uni-hannover.de}
 \noindent 
\address{Institut f\"ur Differentialgeometrie\\
Welfengarten 1\\
30167 Hannover\\
Germany} 
\email{seb.heller@gmail.com}
\noindent
\address{Institut Denis Poisson, CNRS UMR 7350 \\
Facult\'e des Sciences et Techniques \\
Universit\'e de Tours\\France }
\email{martin.traizet@univ-tours.fr }

\address{{\normalfont \bf Appendix A \newline} Fachbereich Mathematik (AZ), Universit\"at Hamburg, Bundesstra\textup{\ss}e 55, 20146 Hamburg, Germany}
\email{steven.charlton@uni-hamburg.de}

\begin{abstract}
For every $g \gg 1$, we show the existence of a complete and smooth family  of closed constant mean curvature surfaces $f_\varphi^g,$ $ \varphi \in [0, \tfrac{\pi}{2}],$ in the round $3$-sphere deforming the Lawson surface $\xi_{1, g}$ to a doubly covered geodesic 2-sphere with monotonically increasing Willmore energy. 
To construct these we use an implicit function theorem argument in the parameter $t= \tfrac{1}{2(g+1)}$.  
This allows us to give an iterative algorithm to compute the power series expansion of the DPW potential and area of $f_\varphi^g$  at $t= 0$ explicitly.  In particular,  we obtain for large genus Lawson surfaces $\xi_{1,g}$ 
a scheme to explicitly compute the coefficients of the power series in $t$ in terms of multiple polylogarithms. Remarkably, the third order coefficient of the area expansion is  identified with $\tfrac{9}{4}\zeta(3),$ where $\zeta$ is the Riemann $\zeta$ function (while the first and second order term were shown to be $\log(2)$  and $0$ respectively in \cite{HHT}).
\end{abstract}

\thanks{The first and second author are supported by the  {\em Deutsche Forschungsgemeinschaft} within the priority program {\em Geometry at Infinity}.}
\thanks{MT is supported by the French ANR project Min-Max (ANR-19-CE40-0014).}
\thanks{SC was supported by DFG Eigene Stelle grant CH 2561/1-1, for Projektnummer 442093436.}
\maketitle

\begin{figure}[h]
\centering
\includegraphics[width=0.3\textwidth]{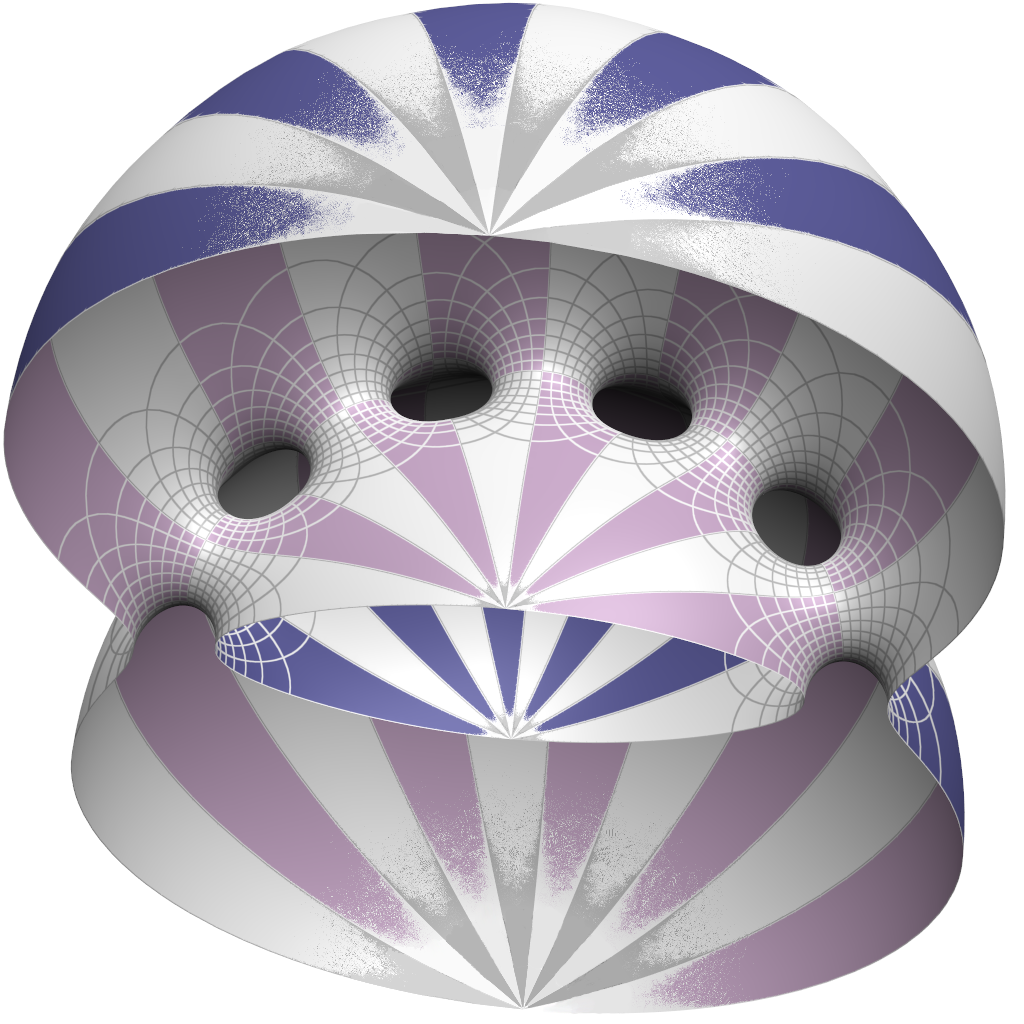}
\caption{
\footnotesize{Cutaway view of the Lawson surface of genus 9, which allude to the convergence to two perpendicularly intersecting 2-spheres for $g \rightarrow \infty$. Image by Nick Schmitt.
}}
\label{l9}
\end{figure}

\setcounter{tocdepth}{1}
\tableofcontents

\section*{Introduction}\label{sec:intro}
 
Minimal surfaces, and more generally constant mean curvature (CMC) surfaces, in three-dimensional space forms have been the object of intensive research for centuries. Local properties of these surfaces are completely understood via (generalizations of) the Weierstra\ss $\;$ representation, which gives explicit parametrizations of the surface in terms of holomorphic data. Determining global properties, such as the topology, area or embeddedness of minimal or CMC surfaces remains very challenging. \\

Due to the maximum principle, the behaviour of solutions and their moduli space depends  crucially on the (constant) curvature of the ambient space. The most difficult case is the round $3$-sphere, where very few examples have been constructed. The first compact embedded minimal surfaces in the $3$-sphere of all genera were found by Lawson \cite{Lawson} using the solution of the Plateau problem with respect to a polygonal boundary curve which are then reflected and rotated along geodesics. Similarly, Karcher-Pinkall-Sterling \cite{KPS} have constructed compact minimal surfaces with platonic symmetries. The other class of closed minimal surfaces in $\S^3$ were constructed by Kapouleas via glueing equatorial 2-spheres with catenoidal handles. Though topology and embeddedness of these examples are known, other geometric properties, like area or stability index are difficult to determine. While the index of the Lawson surfaces $\xi_{1,g}$ were recently computed by Kapouleas and Wiygul \cite{KapWiy}, no area of any minimal surface of genus $g\geq2$ has been explicitly computed. \\

An alternate  approach to constructing minimal and CMC surfaces in space forms uses the integrable systems structure of harmonic maps. This can be interpreted as a global version of the Weierstra\ss $\;$ representation, which is often referred to as the DPW method \cite{DPW} in this context. For tori the integrable systems approach was pioneered by Hitchin \cite{Hitchin} and Pinkall-Sterling \cite{PS} around 1990, and Bobenko \cite{Bobenko} gave an explicit parametrization of all CMC tori in 3-dimensional space forms.
 \\

Consider a conformally parametrized minimal immersion $f$ from a compact Riemann surface $M_g$ of genus $g$ into the round $3$-sphere. Then $f$ is harmonic gives rise to a symmetry of the Gauss-Codazzi equations  inducing an associated $\S^1$-family of (isometric) minimal surfaces on the universal covering of $M_g$ with rotated Hopf differential. The gauge theoretic counterpart of this symmetry is manifested in an associated $\C^*$-family 
of flat $\text{SL}(2,\mathbb C)$-connections $\nabla^\lambda$  \cite{Hitchin} on the trivial $\C^2$-bundle over $M_g$ satisfying 
\begin{enumerate}\label{closingconditions}
\item[(i)] conformality:  $\nabla^\lambda=\lambda^{-1}\Phi+\nabla+\lambda \Psi$ for a nilpotent $\Phi\in \Omega^{1,0}(M_g,\mathfrak{sl}(2,\mathbb C)) ;$
\item[(ii)]  intrinsic closing: $\nabla^\lambda$ is unitary for all $\lambda\in\S^1,$ i.e., $\nabla$ is unitary and $\Psi=-\Phi^*$ with respect to the standard hermitian metric on $\underline{\C}^2$;
\item[(iii)] extrinsic closing: $\nabla^\lambda$ is trivial for $\lambda=\pm1.$
\end{enumerate}
The minimal surface can be reconstructed from the associated family of connections as the gauge between
$\nabla^{-1}$ and $\nabla^1.$ Constructing minimal surfaces is thus equivalent to writing down appropriate families of flat connections. CMC surfaces into space forms can be obtained from the associated families of the corresponding minimal surfaces (via Lawson correspondence) by requiring $\nabla^\lambda$ to be trivial at two other points $\lambda_1$ and $\lambda_2$ on the unit circle. Again the gauge between the connections $\nabla^{\lambda_1}$ and $\nabla^{\lambda_2}$ gives the immersion.  \\

The DPW method \cite{DPW} is a way to generate families of flat connections on $M_g$ from so-called {\em DPW potentials}, denoted by  $\eta = \eta^\lambda$ 
using loop group factorisation. In fact, $\eta^\lambda$ fixes the gauge class of the connections $\nabla^\lambda$ as
$$d+ \eta^\lambda \in [\nabla^\lambda].$$ On simply connected domains $\D$, all DPW potentials give rise to minimal surfaces from $\D$. Whenever the domain has non-trivial topology, finding DPW potentials satisfying  
conditions equivalent to (i)-(iii) becomes difficult.  The problem of finding DPW potentials that fulfill these conditions is referred to as Monodromy Problem.

Though successful in the case of tori, the first embedded and closed minimal surfaces of genus $g>1$ using DPW were only recently constructed in \cite{HHT}. This is due to the fact that in contrast to tori the fundamental group of a higher genus surface is non-abelian. A global version of DPW has been developed in \cite{He1, He2} under certain symmetry assumptions. The main challenge to actually construct higher genus minimal and CMC surfaces is to determine infinitely many parameters in the holomorphic ``Weierstra\ss-data'' -- referred to as spectral data. The idea to determine these missing parameters is to start at a well-understood surface, relax some closing conditions, and deform the known spectral data in a direction that changes the genus of the surface and such that the Monodromy Problem is solved for rational parameters.
These ideas were first implemented in \cite{HHS} to deform homogenous and 2-lobed Delaunay tori in direction of higher genus CMC surfaces giving rise to families of closed but branched CMC surfaces in the 3-sphere. A similar philosophy was independently pursued in \cite{nnoids} to find CMC spheres with Delaunay ends  in Euclidean space, and more generally, CMC surfaces close to a chain of spheres have been constructed in \cite{nodes}. Combining both approaches embedded minimal surfaces using the DPW approach were constructed in \cite{HHT}. In particular, we gave an alternate existence proof of the Lawson surfaces $\xi_{1,g,}$ by constructing a family $f^t$ of minimal surfaces for $t\sim0$ starting at two geodesic spheres intersecting at right angle and deform its DPW potential in direction of a Scherk surface such that $f^t= \xi_{1,g}$ at $t= \tfrac{1}{2(g+1)}$.  The main advantage of the DPW approach is that geometric properties of the surfaces can be explicitly derived from the spectral data. For example, the Taylor expansion of the Area of $\xi_{1,g}$ at $g=\infty$ is computed in \cite{HHT} to be 
\begin{equation}\label{Lawsonestimates}
\Area(\xi_{1,g})=8\pi \left(1-\frac{\log (2)}{2(g+1)}+ O\left(\frac{1}{(g+1)^3}\right)\right).
\end{equation}
The major improvement of the paper at hand is that we construct Fuchsian DPW potentials that incorporate all symmetries of the immersion.  In stark contrast to \cite{HHT} this new approach not only allows 
to obtain a complete family of CMC surfaces $f^g_\varphi$ deforming the Lawson surface $\xi_{1,g}$ but also allows for an iterative algorithm to compute the Taylor expansions of the DPW potential as well as the area of $\xi_{1,g}$ explicitly. 
Moreover, by an estimate of Li and Yau \cite{LiYau} surfaces $f:M \longrightarrow\S^3$ with Willmore energy 
\begin{equation*}
\mathcal W (f) = \int_{M_g} (H^2+1) dA,
\end{equation*}
below $8 \pi$ are automatically embedded. This is the property used to obtain embeddedness of the CMC surfaces $f^g_\varphi$ for $g\gg1$. 

 The outline of the paper is as follows.
Let  $M_g= M^g_\varphi$ be a Riemann surface admitting a $\Z_{g+1}$-symmetry, i.e., it has the structure of a totally branched $(g+1)$-fold covering of $\C P^1$ with four branch points $p_1, ..., p_4$ given by
$$p_1 = e^{-i \varphi}; \quad p_2 = e^{i \varphi}; \quad p_3 = -p_1; \quad p_4 = -p_2. $$
The conformal structure of $M^g_\varphi$ is determined by $\varphi \in (0, \tfrac{\pi}{2})$ which degenerates for $\varphi\to 0,\tfrac{\pi}{2}$ and satisfies $$M^g_{\varphi} \cong M^g_{\tfrac{\pi}{2}- \varphi}.$$
For $g\gg1$ we show the existence of a complete family of closed CMC surfaces connecting the Lawson surface $\xi_{1,g}$ and  the doubly covered geodesic sphere (for $\varphi\to 0,\tfrac{\pi}{2}$)  with $2g+2$ branch points. Combining the results by Kusner-Mazzeo-Pollack \cite{KMP} and Kapouleas and Wiygul \cite{KapWiy} the moduli space of CMC surfaces is $1$-dimensional at the Lawson surface $\xi_{1,g}$. Therefore, the families we construct here are in fact the only ones deforming $\xi_{1,g}$  and give the first global result on the structure of the moduli space of CMC surfaces of genus $g>1.$

Our two main theorems are the following:

\begin{theorem}(Existence)\label{MT}
For every $g \in \N$ sufficiently large,  there exists  a smooth family of conformal CMC embeddings $f^g_\varphi \colon M^g_\varphi \longrightarrow \S^3$  with parameter $\varphi \in (0, \tfrac{\pi}{2})$ satisfying
\begin{itemize}
\item for $\varphi \rightarrow 0, \tfrac{\pi}{2}$ the immersion $f^g_\varphi$ smoothly converges to a doubly covered geodesic 2-sphere with $2g+2$ branch points, i.e., the family $f^g_\varphi$ cannot be extended in the parameter $\varphi$ in the space of immersions;
\item $f^g_\varphi = f^g_{\tfrac{\pi}{2} - \varphi}$ up to  reparametrization and (orientation reversing) isometries of $\S^3$;
\item $f^g_{\tfrac{\pi}{4}}$ is the Lawson surface $\xi_{1,g}$ of genus $g$;
\item the (constant) mean curvature $H^g_\varphi$ of $f^g_\varphi$ is zero if and only if $\varphi =\tfrac{\pi}{4}$.
\end{itemize}
\end{theorem}

\begin{theorem}(Energy computation)\label{MT2}
Let  $f^g_\varphi \colon M^g_\varphi \longrightarrow \S^3$ be the  smooth family of conformal CMC embeddings constructed in Theorem \ref{MT}.
There exist an iterative algorithm to compute the DPW potential as well as the area and Willmore energy of $f^g_\varphi$ in term of multi-polylogarithms.

In particular we have
\begin{itemize}
\item the Willmore energy of $f^g_\varphi$ is strictly monotonically decreasing in $\varphi$ for $\varphi \in (0, \tfrac{\pi}{4})$ from $8 \pi$ to Area$(\xi_{1,g})= \mathcal W(\xi_{1,g})$  with Taylor expansion  at $g=\infty$ given by  

\begin{eqnarray}
\label{estimates}
\mathcal W(f_\varphi^g) &=& 8 \pi \left(1+  \left[ \cos(\varphi)^2 \log(\cos(\varphi)) + \sin(\varphi)^2\log(\sin(\varphi)) \right] \frac{1}{g+1}\right.\\
&+&
\left.\sin(2\varphi)^2\log(\tan(\varphi))^2\frac{1}{(g+1)^2}
+ O\left(\frac{1}{(g+1)^3}\right)\right)
\nonumber
\end{eqnarray}

\item furthermore, for $\varphi = \tfrac{\pi}{4}$

\begin{equation}\label{Area}\text{Area}(f_\varphi^g) =8 \pi \left(1 - \frac{\log(2)}{2 g+2} - \tfrac{9}{4}\zeta (3) \frac{1}{(2g+2)^3} + O\left(\frac{1}{(g+1)^5}\right)\right)
\end{equation}
where $\zeta$ is the Riemann $\zeta$-function.
\end{itemize}
\end{theorem}

\begin{remark}
The identification of the third order coefficient of the area expansion via $\zeta(3)$ is conducted by Steven Charlton in Appendix \ref{appendixC}.

\end{remark} 
An immediate corollary from the energy estimate is that the infimum Willmore energy in the conformal class of $f^g_\varphi$ is below $8 \pi$ for $g\gg1$. Therefore, by \cite{KuwertSchatzle, KuwertLi} as well as \cite{Riviere1, Riviere2} this infimum is attained at an embedding.
\begin{corollary}
For $g\gg1$ the infimum 
\begin{equation*}
\text{Inf}\;\{\mathcal W(f) \; | \; f \colon M^g_\varphi \longrightarrow \S^3, \text{conformal immersion}\}.
\end{equation*}
is attained at a smooth and conformal embedding. 
\end{corollary}

The strategy to prove these two main theorems is as follows. Since the considered Riemann surface $M^g_\varphi$ is $\Z_{g+1}$-symmetric, i.e, admits a covering map $\pi\colon M^g_\varphi \longrightarrow \C P^1$ totally branched over four points $p_1, ..., p_4$, it suffices to show the existence of a DPW potential solving a Monodromy Problem on the 4-punctured sphere $\C P^1\setminus\{p_1, ..., p_4\}.$
We introduce the real parameter $t=\frac{1}{2g+2}$ and define a family of Fuchsian DPW potential
$\eta_{t,\varphi}^{\xx}$ on the 4-punctured sphere depending on $t$ and some auxiliary parameters $\xx$.
The Monodromy Problem is solved using the Implicit Function Theorem at $t=0$ (which corresponds to $g=\infty$) yielding
 a fundamental piece whose boundary lies on a wedge of angle 
$4\pi t$.
When $t=\frac{1}{2g+2}$, we can complete the surface by symmetry to obtain a closed surface of genus $g$: our conformal CMC immersion $f^g_{\varphi}:M^g_{\varphi}\to\S^3$.
Its area is $g+1$ times the area of the fundamental piece.
\begin{remark}
Let $\mathcal A^t_{\varphi}$ be the area of the fundamental piece divided by
$2t$, which agrees with the area of $f^g_{\varphi}$ when $t=\frac{1}{2g+2}$.
Then $\mathcal A^t_{\varphi}$ is a real analytic function of $t\sim 0$ and $\varphi \in (0, \tfrac{\pi}{2})$, as a consequence of the real analytic Implicit Function Theorem. Therefore, the Taylor expansions of the DPW potential and the area are in fact power series of real analytic functions at $t=0$. 
For $\varphi = \tfrac{\pi}{4}$, which corresponds to the Lawson surfaces $\xi_{1,g},$  existence and real analyticity of the corresponding DPW potentials $\eta_t$ is shown in \cite{HH3} for all $t \in [0, \tfrac{1}{4}]$ (till $g=1$). In particular, it is shown \cite[Theorem 1]{HH3} that $\mathcal A^t_{\varphi}$ depends real analytically on $t$ using the strict stability and uniqueness of its fundamental piece established in \cite{KapWiy}.
In fact, the numerics seems to suggest that the convergence radius $R$ of the area power series satifies $R\geq \tfrac{1}{4}$.
\end{remark}

In \cite{HHT} we have constructed DPW potentials of high genus Lawson surfaces $\xi_{1,g}$ by starting at two orthogonally intersecting geodesic spheres and deform it in direction of a Scherk surface with wing angle $\tfrac{\pi}{2}.$ Since the potential used in  \cite{HHT}  cannot be generalized to general wing angles (in a way that an implicit function theorem argument can be applied to obtain compact surfaces), we give in Section \ref{sec:tech} an alternate Ansatz using the same philosophy which allows the intersection angle and the wing angle to be $2\varphi$ with $\varphi \in (0, \tfrac{\pi}{2})$.

The first major advantage of these Fuchsian potentials $\eta_{t, \varphi}^{\xx(t)}$ is that the limiting behaviour for $\varphi\rightarrow 0$ (and $\varphi \rightarrow \tfrac{\pi}{2}$) can be well understood to obtain a uniform existence interval in $t$ for all $\varphi$ in Section \ref{limitvarphi}. This gives rise to complete families of CMC surfaces of genus $g \gg1$.  First order area and Willmore energy expansions are then computed in Section \ref{areaestimates}.

The second advantage of the new Fuchsian potential is  a stark simplification of these computations. 
This enables us
in the last Section \ref{HOD} to obtain an iterative algorithm to compute the coefficients of the DPW potential. Though at each step of the algorithm the to-be-computed derivatives are polynomials, i.e., only an invertible finite dimensional system  needs to be solved, carrying out the algorithm explicitly grows enormously in complexity with the order of the expansion. The main step is to compute certain iterated integrals which can be expressed in terms of multiple polylogarithms (MPL's) to identify the desired coefficients. The Appendix \ref{appendixC} by Steven Charlton relates the third order coefficient of the area expansion for the Lawson surfaces  with Ap\'ery's constant $\zeta(3)$ using properties of MPL's. That computing the Taylor coefficients becomes highly non-trivial with increasing order seems to foreshadow deeper interconnections between the solutions of our monodromy problems and some MPL identities.

Other directions of future research and conjectures includes:
\begin{itemize}
\item We conjecture the family of CMC surfaces $f^g_\varphi,$ for $ \varphi\in]0,\tfrac{\pi}{2}[$, to exist for every $g\geq1$. 
\item For $g$ fixed, we conjecture $f^g_\varphi$  to be minimizers of the Willmore energy in their respective conformal classes, or at least to be constrained Willmore stable, analogously to the family of 2-lobed Delaunay tori \cite{HHNd}.
\item Finding a generating function for the coefficients of Taylor expansions for DPW potential and area of $f^g_\varphi$ and estimate their convergence radius.
\item Finding an integrable systems structure or Lax pair representation for the deformation of the DPW potential in $t.$
\item Similar constructions should work for the general Lawson surfaces $\xi_{k,l}$ as well as the Karcher-Pinkall-Sterling surfaces. New minimal and CMC surfaces could be constructed using more complicated initial conditions.
\end{itemize}
\section{Preliminaries}\label{sec:prelimi}

The DPW method \cite{DPW} is a technique to parametrize minimal and CMC surfaces in space forms via holomorphic data using loop groups. We first set the basic definitions and the necessary notations here, for more details adapted to our approach see \cite{HHT}. \\
\subsection{Loop groups}
\label{loop-groups}
Let $G$ be a Lie group and $\mathfrak g$ its Lie algebra. Then the associated loop group is defined to be
 $$\Lambda G:= \{\text{ real analytic maps (loops) }\Phi\colon \S^1\longrightarrow G,\quad \lambda \longmapsto \Phi^\lambda\},$$
which is an infinite dimensional Frechet Lie group via pointwise multiplication. Its Lie algebra is given by
$$\Lambda \mathfrak{g}:=\{\text{ real analytic maps (loops) }\eta\colon \S^1\longrightarrow\mathfrak{g},\quad  \lambda \longmapsto \eta^\lambda\}.$$ 
Denote $\overline\D=\{\lambda\in\C\mid |\lambda|\leq1\}.$
 For a complex Lie group $G^\C$ we denote by
\[\Lambda_+G^\C=\{\Phi\in\Lambda G^\C\mid \Phi \text{ extends holomorphically to } \overline\D\}\]
 
the positive part of the loop group and similarly 
 \[\Lambda_+\mathfrak g^\C=\{\eta\in\Lambda \mathfrak g^\C\mid \eta \text{ extends holomorphically to } \overline\D\}\]
 is the positive part of its Lie algebra.\\

As in  \cite{nnoids, HHT} we consider the following functional spaces: 
For  $h\in L^2(\S^1,\mathbb C)$ consider its Fourier series
\[h=\sum_k h_k \lambda^k.\] Fix $\rho>1$ and define

\[\parallel h\parallel_{\rho}=\sum |h_k| \rho^{|k|}\leq\infty.\]
Let
 \[\mathcal W_{\rho}:=\{h\in  L^2(\S^1,\mathbb C)\mid \,\parallel h\parallel_{\rho}<\infty\}\] 
be the space of Fourier series being absolutely convergent  on the annulus 
$$\A_{\rho}=\{\lambda\in\C\mid \tfrac{1}{\rho} < |\lambda| <{\rho}\}.$$ 
 \begin{remark}
For arbitrary loop spaces $\mathcal H$,
$\mathcal H_{\rho}$ denotes the subspace of $\mathcal H$ of loops whose entries
are in $\mathcal W_{\rho}$. Then $\Lambda SL(2,\C)_{\rho}$, $\Lambda SU(2)_{\rho}$ and $\Lambda_+^{\R} SL(2,\C)_{\rho}$
are Banach Lie groups.
We will actually omit the subscript $\rho$ most of the time.
 \end{remark}

Moreover, let
\[\cal{W}^{\geq 0}_{\rho}:=\{h=\sum_k h_k \lambda^k\in \mathcal W_{\rho}\mid
h_k=0\; \;\;\forall  \;k<0\}\]
denote the space of those  loops $f\in L^2(\S^1,\mathbb C)$ that can be extended to a holomorphic function on the disk $\D_{\rho}=D(0,\rho)$.  Similarly, let
\[\cal{W}^{>0}_{\rho}:=\{h=\sum_k h_k \lambda^k\in \mathcal W_{\rho}\mid h_k=0\; \;\;\forall\; k\leq0\,\}\]
\[\cal{W}^{<0}_{\rho}:=\{h=\sum_k h_k \lambda^k\in \mathcal W_{\rho}\mid h_k=0\;\;\;\forall \;k\geq0\,\}\]
denote the positive and negative space, respectively.
Therefore we can decompose every $h\in \mathcal W_{\rho}$
$$h=h^+ + h^0+ h^-$$
into its positive and negative component $h^\pm \in W^{\gtrless 0}_{\rho}$, and a constant component $h^0 = h_0$.  \\

We define the star and conjugation involutions on $\mathcal W_{\rho}$ by
$$h^{*}(\lambda)=\overline{h(1/\overline{\lambda})}
\quad\mbox{ and }\quad \overline{h}(\lambda)=\overline{h(\overline{\lambda})}$$
and we denote by $\mathcal W_{\R,\rho}$ the space of functions $h\in\mathcal W_{\rho}$
with $h=\overline{h}$.\\

The following proposition will allow us to apply a variant of the implicit function theorem when the differential is not surjective by adding the parameters $(a_0, ..., a_{n-1}).$
\begin{proposition}
\label{Pro:decomposition}
Let $\mu_1,\cdots,\mu_n\in\D_{\rho}$. The operator
\[(a_0,\cdots,a_{n-1},g)\mapsto
\sum_{k=0}^{n-1}a_k\lambda^k + g(\lambda)\prod_{k=1}^n(\lambda-\mu_i)\]
is an isomorphism from $\C^n\times \mathcal W_{\rho}^{\geq 0}$ to $\mathcal W_{\rho}^{\geq 0}$.
If in addition $\mu_1,\cdots,\mu_n$ are real, the operator restricts to an isomorphism
from $\R^n\times \mathcal W^{\geq 0}_{\R,\rho}$ to $\mathcal W^{\geq 0}_{\R,\rho}$.
\end{proposition}
\begin{proof} We first prove the case $n=1$ and write $\mu=\mu_1$. Let $h\in\mathcal W^{\geq 0}_{\rho}$ and define
$$g(\lambda)=\frac{h(\lambda)-h(\mu)}{\lambda-\mu}.$$
We have
$$g(\lambda)=\sum_{k=0}^{\infty} h_k\frac{\lambda^k-\mu^k}{\lambda-\mu}=
\sum_{k=0}^{\infty}\sum_{j=0}^{k-1}h_k\lambda^j\mu^{k-1-j}$$
$$\|g\|_{\rho}\leq \sum_{k=0}^{\infty}\sum_{j=0}^{k-1} |h_k|\,\rho^j\,|\mu|^{k-1-j}
=\sum_{k=0}^{\infty}|h_k|\frac{\rho^k-|\mu|^k}{\rho-|\mu|}\leq
\frac{\|h\|_{\rho}}{\rho-|\mu|}.$$
Hence $g\in\mathcal W^{\geq 0}_{\rho}$ and $(h(\mu),g)$ is the (unique) pre-image of $h$.
This proves the proposition in the case $n=1$. The general case follows by induction and composition.
\end{proof}

\subsection{DPW approach}
A DPW potential on a Riemann surface $M$ is a closed (i.e. holomorphic) complex linear 1-form

\[\eta\in\Omega^{1,0}(M,\Lambda \mathfrak{sl}(2,\C))\]

with  
\[\lambda \eta\in \Omega^{1,0}(M,\Lambda_+\mathfrak{sl}(2,\C))\]
such that  its residue at $\lambda=0$ 
\[\eta_{-1}:=\text{Res}_{\lambda=0} (\eta)\]
is a nowhere vanishing and nilpotent 1-form.  For a given DPW potential $\eta$ the extended frame $\Phi$ is a solution of 

$$d_{M} \Phi =  \Phi \eta$$

with some initial value $\Phi (z_0)= \Phi_0\in\Lambda\SL(2,\C)$ for some $z_0 \in M$ fixed.
Consider an element of the first fundamental group $\gamma\in\pi_1(M,z_0)$ and let $\cal{M}(\Phi,\gamma)$ denote the monodromy of $\Phi$ with respect to $\gamma.$ The conditions for the DPW potential to give a well-defined minimal or CMC surface from $M$ into the round 3-sphere in terms of $\Phi$ are

\begin{equation} \label{monodromy-problem}
\left\{\begin{array}{l}
\cal{M}(\Phi,\gamma)\in \Lambda SU(2)\\
\cal{M}(\Phi,\gamma)|_{\lambda=\lambda_1}=\cal{M}(\Phi,\gamma)_{\lambda=\lambda_2}=\pm\Id_2, \quad \text{for all }\gamma\in\pi_1(M,z_0)\end{array}\right.\end{equation}

where $\lambda_1\neq\lambda_2\in\mathbb S^1.$ 
We refer to these conditions in \eqref{monodromy-problem} as the {\em Monodromy Problem}. 
The Iwasawa decomposition is the splitting of $\Phi$ into a unitary and a positive factor, i.e.,

$$\Phi = F B$$

with $F \in \Lambda \SU (2)$ and $B \in \Lambda_+\SL(2,\C)$. This splitting becomes unique if we require 

$$B^{\lambda=0} \in \mathcal B$$ with
\[\mathcal B=\{B\in\SL(2,\C)\mid B \text{ is upper triangular with positive diagonal entries}\}.\]

Applying the Iwasawa decomposition pointwise to the extended frame on the universal covering $\widetilde M$ of the Riemann surface $M$
yiels a smooth unitary factor $F$.
Then conformal immersion of constant mean curvature $H = i  \tfrac{\lambda_1 + \lambda_2}{\lambda_1- \lambda_2}$ can be reconstructed from the unitary factor $F$ by the Sym-Bobenko formula

$$f = F^{\lambda = \lambda_1} \left(F^{\lambda = \lambda_2} \right)^{-1}\colon\widetilde M\to\mathbb S^3.$$ 
Provided the Monodromy Problem \eqref{monodromy-problem} is satisfied, $f$ is well-defined on $M$.
 
\begin{remark} 
The Iwasawa decomposition is a smooth diffeomorphism between the Lie groups
$\Lambda SL(2,\C)_{\rho}$ and $\Lambda SU(2)_{\rho}\times\Lambda_+^{\R} SL(2,\C)_{\rho}$
(see  Theorem 5 in \cite{minoids}).
\end{remark}
 
In this paper we restrict ourselves to genus $g$ Riemann surfaces $M^g_\varphi$ with a $\Z_{g+1}$-symmetry. To be more concrete, $M^g_\varphi$ admits a $(g+1)$-fold covering
$$\pi \colon M^g_\varphi \longrightarrow \C P^1$$
totally branched over four points
\begin{equation}\label{realvarphipj}p_1=e^{i\varphi},\quad p_2=-e^{-i\varphi},\quad p_3=-e^{i\varphi},\quad p_4=e^{-i\varphi}.\end{equation}

The Riemann surface structure of $M^g_\varphi$ is determined by the algebraic equation
$$y^{g+1} = \frac{(z-p_1)(z-p_3)}{(z-p_2)(z-p_4)}.$$

Rather than showing the existence of DPW potentials on $M^g_\varphi$, we consider DPW potentials $\eta_{t, \varphi} $ on
$$\Sigma = \Sigma_\varphi:= \C P^1  \setminus \{p_1, ..., p_4\}.$$

\begin{remark}
On a compact Riemann surface
there exists no  DPW potential solving the Monodromy Problem without singularities (e.g. poles).
To obtain compact CMC surfaces we require the necessary singularities to be apparent, i.e., removable by suitable local gauge transformation.
\end{remark}
\begin{remark}\label{biholo}
For $\varphi$ fixed let $\wt \varphi := \tfrac{\pi}{2} -\varphi.$ Then the resulting Riemann surfaces $\Sigma_\varphi$ and $\Sigma_{\wt \varphi}$ are biholomorphic to each other by $z \mapsto i z.$ The order of the singular points points are hereby permuted:
$$(p_1(\varphi), p_2(\varphi), p_3(\varphi), p_4(\varphi)) \longmapsto (p_2(\wt \varphi), p_3(\wt \varphi), p_4(\wt \varphi), p_1(\wt \varphi)).$$\\
\end{remark}

On the $4$-punctured sphere there exist a particular choice of DPW potentials such that $\eta^\lambda$ gives rise to a Fuchsian system for all $\lambda \in \overline\D^*_{\rho}\subset\C^*$. We will refer to these potentials as Fuchsian potentials in the following.
\subsection{Fuchsian Systems}$\;$\\
A \SL$(2, \C)$ Fuchsian system on the 4-punctured sphere $\Sigma$ is a holomorphic connection on the trivial $\C^2$-bundle over $\Sigma$ of the form $\nabla= d+ \xi$ 
$$\xi = \sum_{k=1}^4 A_k \frac{dz}{z-p_k},$$
where $A_k \in \mathfrak {sl}(2, \C)$ and ${\displaystyle \sum_{k=1}^4 A_k= 0}$ to avoid a further singularity at $z= \infty.$  Two Fuchsian systems are equivalent (when fixing the punctures $p_k$), if there exist an invertible matrix $G$ such that $\tilde A_k = G^{-1}A_k G.$
Due to its form, the curvature of $\nabla$ is automatically zero, and we can consider the associated monodromy 
 representation of the first fundamental group $\pi_1(\Sigma)$. 
 This depends by construction on the initial value of the fundamental solution at the base point and is therefore only well defined up to an overall conjugation. Via the monodromy representation the space of these (irreducible) Fuchsian systems
is biholomorphic to an open dense subset of the space of (irreducible) representations of the first fundamental group of $\Sigma$  to $\mathrm{SL}(2,\C)$ (both modulo conjugation). Of particular interest for the construction of CMC surfaces
are  Fuchsian systems admitting a unitary monodromy representation. 
 
 \begin{definition}
 A \SL$(2, \C)$ Fuchsian system is called unitarizable if there exist a hermitian metric $h$  on $\underline{\C}^2\to\Sigma$ such that the connection $d+\xi$ is unitary with respect to $h.$
 \end{definition}
  
The analogous definition on the representation side is
\begin{definition}
A monodromy representation is unitarizable if it lies in the conjugacy class of a unitary representation.
\end{definition}

The space of representations of $\pi_1(\Sigma,z_0)$ modulo conjugation
 can be parametrized using so-called Fricke coordinates. Let $M_1,\dots,M_4$ denote the monodromies of $\Phi$ along $\gamma_k$ (a simply closed loop around one singularity $p_k$) and define
\[s_k:=\tr (M_k); \quad \quad s_{k,l}=\tr (M_kM_l).\]
Since $M_k \in \SL(2, \C), $ the trace $s_k$ determines the eigenvalues of $M_k, $ i.e., the conjugacy class of the local monodromy $M_k.$ For the symmetric case considered in this paper we have $$s_1= s_2=s_3 = s_4.$$
\begin{proposition}\label{Pro:Friecke}
Consider a $\SL(2,\C)$-representation on the $4$-punctured sphere $\Sigma$. Let  $s=s_1=\dots =s_4$ and  
let $u=s_{1,2}$, $v=s_{2,3}$, $w=s_{1,3}.$
Then the following algebraic equation holds
\begin{equation}\label{goldpol}u^2+v^2+w^2+u\,v\,w-2 s^2(u+v+w)+4(s^2-1)+s^4=0.\end{equation}
When satisfying \eqref{goldpol} the parameters $s$ and $u,v,w$ together determine a monodromy representation $\rho \colon \pi_1(\Sigma) \rightarrow$ SL$(2, \C)$ from the first fundamental group of $\Sigma$ into SL$(2, \C)$. For generic parameters, this representation is unique up to conjugation.
\end{proposition}
\begin{proof}
This is a classical result of Vogt and Fricke, for a modern presentation see \cite{Gold}.
\end{proof}

We will use the following observation:
\begin{lemma}
\label{Lem:discrim}
The polynomial in \eqref{goldpol} is quadratic in $w$ and its discriminant is given by
\begin{equation}\label{discriminant}
(u-2)(v-2)((u+2)(v+2)-4 s^2).
\end{equation}
\end{lemma}
 
Via the 
monodromy representation we obtain new coordinates on the space of irreducible Fuchsian systems modulo conjugation.
The advantage of these coordinates is that the condition whether the Fuchsian system $\eta$ is unitarizable  can be explicitly expressed in terms of the traces $s_k$ and $s_{k,l}$.
A standard result is that an irreducible representation of $\pi_1(M)$
into $\mathrm{SL}(2,\C)$ is unitarizable if and only if 
$s_k \in(-2,2)$ for $k=1,\dots4$
and
$s_{k,l}\in[-2,2]$ for $(k,l)\in\{(1,2),(2,3),(1,3)\}.$
In the same vein, we obtain under the previous symmetry assumptions the following Lemma.
\begin{lemma}\label{Lem:irreducibility-traces}
Let $\rho\colon \pi_1(\Sigma) \rightarrow $SL$(2,\C)$ be a representation of the first fundamental group of the 4-punctured sphere
into $\mathrm{SL}(2,\C)$ determined  by
$M_1,\dots,M_4\in\mathrm{SL}(2,\C)$ satisfying  
$$M_1M_2M_3M_4=\Id; \quad M_3=D^{-1}M_1D; \quad \text{and} \quad M_4=D^{-1}M_2D, $$ with $D=\text{diag}(\ii,-\ii)$.
If moreover 
\begin{equation}\begin{split}
&\tr(M_1)=\tr(M_2)=s\in(-2,2)\subset \R,\\
&s_{1,2},\,s_{2,3}\in(-2,2) \subset \R,\\
&s_{1,3},\,s_{2,4}\in\R\end{split}\end{equation}
then $\rho$ is unitarizable by a diagonal matrix.
\end{lemma}
\begin{proof} First we claim that there exists two matrices
\[X=\begin{pmatrix} a& b\\c & d\end{pmatrix}\in\mathrm{SL}(2,\C)\quad Y=\begin{pmatrix} x&y\\ z& w\end{pmatrix}\in\mathrm{SL}(2,\C)\]
such that 
\begin{equation}
\label{eq:XY}
M_1=X Y^{-1},\quad M_2 =Y D^{-1}X^{-1}D.
\end{equation}
Indeed, we have
\[\Id=M_1M_2M_3M_4=-(M_1M_2D^{-1})^2,\]
hence $M_1M_2D^{-1}$ has eigenvalues $\pm\ii$ and is diagonal up to conjugation  by some $X \in \SL(2, \C),$ i.e.,
$M_1M_2D^{-1}=XD^{-1}X^{-1}$. Then choosing $Y=M_1^{-1}X$ shows that \eqref{eq:XY} holds.
\begin{remark}
\label{Rem:XY}
We will apply the Lemma in the proof of Proposition \ref{Pro:monoPQ}. With the notations used there, we will have $X=\Phi(+\infty)$ and $Y=\Phi(\ii\infty)$.
\end{remark}
A computation gives
$$s_{1,2}=2(ad+bc)\quad\text{ and }\quad s_{2,3}=2(xw+yz).$$
If $ad=0$ then $s_{1,2}=-2$. If $bc=0$ then $s_{1,2}=2$.
Since $s_{1,2}\in(-2,2)$, $a,b,c,d$ are all non-zero.
Multiplying $X$ and $Y$ with a diagonal matrix $D_1$ from the left and another diagonal matrix $D_2$ from the right, we may assume without loss of generality that
\[a=d\quad \text{and}\quad c=-b.\]
(This normalisation does not change the symmetry, and the monodromies are conjugate to the initial ones by $D_1,$ so their traces do not change).
Then
\begin{equation}\label{detX}
\det(X)=a^2+b^2=1\end{equation}
and
\begin{equation}\label{s12}
s_{1,2}=2a^2-2 b^2=4a^2-2.
\end{equation}
By assumption, $s_{1,2}\in(-2,2)$ from which we get $a\in(-1,1)\subset\R$. Thus, \eqref{detX} implies $b\in(-1,1)\subset\R$.

A computation gives
\[s_1=a(x+w)+b(y-z)\quad\text{ and }\quad s_2=a(x+w)-b(y-z).\]
Since $s_1=s_2$ and $b\neq 0$,
we have $y=z$. Then
\begin{equation}\label{detY}
\det(Y)=xw-y^2=1\end{equation}
and
\begin{equation}\label{s23}
s_{2,3}=2xw+2y^2=2+4y^2.
\end{equation}
From $s_{2,3}\in(-2,2)$ we get $y\in i\R$ with $|y|<1.$
%
We compute
\[\begin{split}
s_{1,3}&=4 ab (x-w)y+a^2(w^2+x^2-2y^2)+2b^2(xw+y^2)\\
s_{2,4}&=-4 ab (x-w)y+a^2(w^2+x^2-2y^2)+2b^2(xw+y^2),\\
\end{split}\]
so since $s_{1,3}-s_{2,4}\in\R$,
\[4 ab (x-w)y\in\R\]
and because $ab\in\R^*,$ $y\in i\R^*$
we have
\[x-w\in i\R.\]
From $s_1=a(x+w)\in\R$ we obtain $x+w\in\R$.
This gives
\[w=\bar x.\]
Therefore
$X,Y\in\mathrm{SU}(2)$
and 
\[M_k\in\mathrm{SU}(2)\quad\forall k=1,\dots,4\]
(keeping in mind that we have normalized  $X,Y$ by diagonal conjugation).
\end{proof}


\section{DPW potentials for CMC surfaces of high genus}\label{sec:tech}
 
In \cite{HHT} we constructed  Lawson surfaces of high genus using an ansatz for the DPW potential which had an apparent singularity at $z= \infty$. This 
ansatz cannot be generalized to arbitrary $\varphi, $ as the proof 
 of $z= \infty$ being an apparent singularity relies on an additional symmetry of the surfaces for $\varphi = \tfrac{\pi}{4}.$ Therefore, we adjust the ansatz used in \cite{HHT} to Fuchsian DPW potentials $\eta_{t, \varphi}$ on the $4$-punctured sphere $\Sigma$. Geometrically speaking we start for
$\varphi \in (0, \tfrac{\pi}{2})$ fixed and $t=0$ with the trivial DPW potential $\eta_{0, \varphi}=0$ which we interpret as the DPW potential of a point representing the dual surface of two geodesic spheres intersecting at angle $\varphi.$ This potential is then deformed in direction of a Scherk surface with wing angle $2\varphi.$ Note that for $\varphi \rightarrow  0$ or $\varphi \rightarrow \tfrac{\pi}{2}$, the Scherk surface converges to a Catenoid after a suitable blow-up. We will first state our ansatz and the involved  parameters in detail and then shortly explain how to obtain the geometric interpretation.\\ 
\subsection{The potential}\label{ssec:thepotential}$\;$\\
For real $t\sim 0$ and $\varphi \in (0, \tfrac{\pi}{2})$ consider on the Riemann surface $\Sigma$ the DPW potentials of the form
\begin{equation}\label{thegeneraleta}\eta_{t, \varphi}=r \cdot t{\sum_{k=1}^4 }A_k(\lambda)\frac{dz}{z-p_k},\end{equation}
where $r\sim 1$ is a real parameter and the matrices $A_k \in \Lambda\mathfrak{sl}(2, \C)_{\rho}$ are defined below, so that $\eta_{t,\varphi}$ is a DPW potential with the right symmetries. At $t=0$ the potential $\eta_{0, \varphi} = 0$ is trivial and the resulting surface degenerates to a point. 
By the choice of $p_k$ in \eqref{realvarphipj}, the Riemann surface $\Sigma$ has three symmetries given by
$$\delta(z)=-z,  \quad \tau(z)= \tfrac{1}{z}, \quad  \text{ and }\quad  \sigma(z) = \bar z.$$
We require the CMC surface $f$ corresponding to a potential of the form $\eta_{t, \varphi}$ to be compatible with these symmetries, i.e, $f$ should be equivariant with respect to $\delta,$ $\tau,$ $\sigma$ and the isometries of the round 3-sphere.
A  useful condition (compare with the proof of Theorem \ref{thm:construction-fixed-varphi}) on our potential to obtain an equivariant immersion is to require
\begin{itemize}
\item $\eta_{t,\varphi}$ is $\delta$-symmetric:
$$\delta^*\eta_{t, \varphi}=D^{-1}\eta_{t, \varphi}D \quad \text{ with }\quad D=\matrix{i&0\\0&{-i}};$$

\item $\eta_{t,\varphi}$ is $\tau$-symmetric:
$$\tau^*\eta_{t,\varphi} = C^{-1}\eta_{t, \varphi}C\quad \text{ with }\quad C=\matrix{0&i\\i& 0 };$$
\item $\eta_{t,\varphi}$ is $\sigma$-symmetric:
$$\sigma^*\overline{\eta_{t,\varphi}}:= \sigma^*\overline{ \eta_{t,\varphi}(., \overline\lambda)}=  \eta_{t,\varphi}(. , \lambda).$$
\end{itemize}

This is equivalent to
\begin{equation}\label{symmetries}
\begin{split}
&A_{j+2}(\lambda)=D^{-1}A_j({\lambda})D\quad\mbox{ for $1\leq j\leq 2$,  }\\ 
&A_4(\lambda)=C^{-1}A_1(\lambda) C\quad \text{and}\quad A_3(\lambda)=C^{-1}A_2(\lambda)C\\
&A_4(\lambda)=\overline{A_1(\bar\lambda)} \quad\text{and}\quad
A_3(\lambda)=\overline{A_2(\bar\lambda)}.
\end{split}
\end{equation}

In particular, all residues $A_k$ are determined by $A_1$. We write
$$A_1=\begin{pmatrix}\ii a & b+\ii c\\ b-\ii c& -\ii a\end{pmatrix}$$
where $a,b,c\in \mathcal W_{\R,\rho}$. The other residues are then given by

\begin{equation*}
\begin{split}
A_2&=\begin{pmatrix}-\ii a & -b+\ii c\\ -b-\ii c& \ii a\end{pmatrix}\\
A_3&=\begin{pmatrix}\ii a & -b-\ii c\\ -b+\ii c& -\ii a\end{pmatrix}\\
A_4&=\begin{pmatrix}-\ii a & b-\ii c\\ b+\ii c& \ii a\end{pmatrix}.
\end{split}
\end{equation*}

In particular,
\[\sum_{k=0}^4A_k=0,\] i.e.,
the potential is regular at $z= \infty$.  While the symmetries $\delta$ and $\tau$ give rise to symmetries of the Fuchsian system $\eta_{t,\varphi}(\lambda)$
 for every individual  $\lambda\in\C^*,$ the $\sigma$ symmetry relates two different Fuchsian systems in the family $\lambda\mapsto\eta_{t,\varphi}(\lambda).$ The symmetry $\sigma$ indicates that the Sym-points, which we will denote by $\lambda_1,\lambda_2 \in \S^1$ in the following, should be complex conjugate. Hence, our ansatz for the Sym-points is
\[\lambda_1(\theta) :=e^{ i\theta}\quad \text{ and }\quad  \lambda_2(\theta):=e^{-i\theta}\]
for some parameter $\theta\in\R.$

\subsection{The Monodromy Problem}$\;$\\
 Let $\Phi_{t,\varphi}$ be the solution of the Cauchy Problem

\begin{equation}\label{eqn-sol}
d_\Sigma \Phi_{t,\varphi} = \Phi_{t,\varphi} \eta_{t,\varphi}, \quad \text{ with initial condition } \quad \Phi_{t,\varphi}(z=0)=\Id.
\end{equation}
Let $\gamma_1,\cdots,\gamma_4$ be generators of the fundamental group
$\pi_1(\Sigma,0)$, with
$\gamma_k$ enclosing the singularity $p_k$ with
\[\gamma_1\circ\gamma_2\circ\gamma_3\circ\gamma_4=1.\] Let
$M_k(t)=\cal{M}(\Phi_{t,\varphi},\gamma_k)$ be the monodromy of $\Phi_{t,\varphi}$ along $\gamma_k$,  where $\Phi_{t,\varphi}$ solves \eqref{eqn-sol}. 
Following \cite{HHT}, the goal is to solve the following Monodromy Problem

\begin{equation}
\label{monodromy-problem2}
\left\{\begin{array}{l}
(i) \quad\;\forall k\;M_k(t)\in\Lambda SU(2)\\ 
(ii) \quad\;\forall k,\;M_k(t) \mbox{ has eigenvalues $e^{\pm 2\pi\ii t}$}\\
(iii) \quad\exists \theta\in\R: \;\forall k,\;M_k(t)|_{\lambda=\lambda_1(\theta), \lambda_2(\theta)}\mbox{ is diagonal}
\end{array}\right..
\end{equation}

If the problem \eqref{monodromy-problem2} is solved,
the potential pulls back for $t=\frac{1}{2(g+1)}$ to $M^g_\varphi$ and gives rise to a potential with apparent singularities at $\pi^{-1}\{p_1,\dots,p_4\}$
solving the Monodromy Problem \eqref{monodromy-problem}. This yields the desired closed CMC surface $f \colon M^g_\varphi \rightarrow \S^3$. 
Since we impose all symmetries $\sigma,\tau,\delta$ on the Fuchsian potentials, it is necessary to only require the monodromy  to be unitarizable rather than unitary along $\lambda \in \S^1$, i.e., the corresponding monodromy representation lies only in the conjugacy class of a unitary representation. 
In doing so, only the traces of certain monodromies need to be controlled by Lemma \ref{Lem:irreducibility-traces}.\\

For convenience, we define $M(t):=M_1(t)$. At $t=0$, as $\eta_{t,\varphi} = 0,$ the extended frame 
\begin{equation}\label{PhiC}
\Phi_{0}(z):= \Phi_{0,\varphi}(z) = \matrix{1&0\\0 &1}\end{equation}
is the identity, which has trivial monodromy. 
Following \cite[Proposition 8]{nnoids} we compute for $p_1=e^{i\varphi}$
\begin{equation}\label{mondiff}
M'(0)=\int_{\gamma_1}\Phi_0\frac{\partial \eta_{t, \varphi}}{\partial t}|_{t=0}\Phi_0^{-1}
=2\pi\ii \,r\Res_{p_1}\left [A_1\frac{dz}{z-p_1} \right]
=2\pi\ii \,r A_1.
\end{equation}

\subsubsection{Central value}$\;$\\
The Monodromy Problem will be solved by applying the implicit function theorem
at $t=0$ and at a central value of the parameters, denoted with an underscore.
We choose the following central value for the parameters $a,b,c,r$ and $\theta$:
\begin{equation}
\label{central1}
\begin{split}
\cv{a}(\lambda)&=\tfrac{1}{2}(\lambda^{-1}-\lambda)\\
\cv{b}(\lambda)&=-\tfrac{1}{2}\sin(\varphi)(\lambda^{-1}+\lambda)\\
\cv{c}(\lambda)&=-\tfrac{1}{2}\cos(\varphi)(\lambda^{-1}+\lambda)\\
\cv{r}&=1\\
\cv{\theta}&=\tfrac{\pi}{2}.
\end{split}
\end{equation}
These values are chosen so that the Monodromy Problem \ref{monodromy-problem2} is solved to the first order in $t$. To see this,
let
$$\eta'=r\sum_{j=1}^4 \cv{A}_j\frac{dz}{z-p_j}=\sum_{k=-1}^{\infty}\eta'_k\lambda^k$$
be the first order derivative of the potential with respect to $t$ at $t=0$.
Then we need to require $\eta'_{-1}$ to  be nilpotent, i.e., $\det(\eta'_{-1})\equiv 0,$ which gives
$$\det(\Res_{p_1}\eta'_{-1})=\cv{a}_{-1}^2-\cv{b}_{-1}^2-\cv{c}_{-1}^2=0.$$
Provided this holds true, we have
$$\det(\eta'_{-1})=
\frac{-16 \,\cv{r}^2\left(\cos(\varphi)^2\cv{b}_{-1}^2-\sin(\varphi)^2\cv{c}_{-1}^2\right)dz^2}{z^4-2\cos(2\varphi)z^2+1}.$$
So up to scale and signs we need
$$\cv{a}_{-1}=\tfrac{1}{2},\quad
\cv{b}_{-1}=-\tfrac{1}{2}\sin(\varphi)\quad\text{and}\quad
\cv{c}_{-1}=-\tfrac{1}{2}\cos(\varphi).$$
(The reason for the $\frac{1}{2}$ will be clear in a moment).
Using 
\eqref{mondiff} and remembering that all coefficients are real this yields
$$M'_1(0)\in\Lambda\su(2)\Leftrightarrow i \cv{A}_1\in\Lambda\su(2)
\Leftrightarrow \left\{\begin{array}{l}
\cv{a}=-\cv{a}^*\\
\cv{b}=\cv{b}^*\\
\cv{c}=\cv{c}^*\end{array}\right.
\Leftrightarrow
\left\{\begin{array}{l}
\cv{a}(\lambda)=\cv{a}_{-1}(\lambda^{-1}-\lambda)\\
\cv{b}(\lambda)=\cv{b}_{-1}(\lambda^{-1}+\lambda)+\cv{b}_0\\
\cv{c}(\lambda)=\cv{c}_{-1}(\lambda^{-1}+\lambda)+\cv{c}_0\end{array}\right.$$
$$M'_1(0)\mid_{\lambda=\pm i}\text{ diagonal }\Leftrightarrow
\cv{b}_0=\cv{c}_0=0.$$
Then $\det(\cv{A}_1)=-1$, so
for $M'_1(0)$ to have eigenvalues $\pm 2\pi$ we need $\cv{r}=1.$ To conclude, we have shown that the central value of the parameters is uniquely determined up to signs. Moreover, with this central value, the blowup $t^{-1}(f^{t}_\varphi-\Id)$ will
converge in the tangent space of $\S^3$ at $\Id$ to a minimal surface with Weierstra\ss $\;$ data
$$g=\frac{i}{z},\qquad \omega=\frac{-8 i \sin(2\varphi) z dz}{z^4-2\cos(2\varphi) z^2+1},$$
which corresponds to a Scherk surface of wing angle 2$\varphi$ and period $4\pi$.
To see this, first use the gauge $\minimatrix{1&-1/z\\z&0}$ to transform $\eta_{t,\varphi}$ into a potential whose $\lambda^{-1}$ part is strictly upper-triangular, and then use \cite[Theorem 2]{HHT}.

\subsubsection{Parameters}$\;$\\
We take the parameters $a,b,c$ of the form
$$a=\cv{a}+\wt a,\quad
b=\cv{b}+\wt b,\quad\text{and}\quad
c=\cv{b}+\wt c$$
where $\cv{a},\cv{b},\cv{c}$ are given by Equation \eqref{central1} and $\wt a,\wt b,\wt c$ are functions of $\lambda$ in the functional space
$\mathcal W_{\R,\rho}^{\geq 0}$ in a neighborhood of $0$.
Recall that the exponent $\geq 0$ means that these functions have only non-negative powers of $\lambda$ in their Fourier expansion, so all the $\lambda^{-1}$- terms of the potential
are contained in the central values $\cv{a}$, $\cv{b}$ and $\cv{c}$.

Let $\xx = (\wt{a}, \wt{b}, \wt{c}, r,\theta)$ denote the parameter vector. When emphasizing the dependence of  the potential $\eta_{t, \varphi}$ on the parameter vector we write $\eta^\xx_{t,\varphi}$. The central value of the parameter vector $\xx$ is
$\underline \xx=(0,0,0,1,\tfrac{\pi}{2})$.
\subsection{New Fricke-type coordinates}$\;$\\
Consider a DPW potential $\eta=\eta_{t, \varphi}$
of the form \eqref{thegeneraleta} which is compatible with the symmetries $\delta,$ $\tau,$ and $\sigma.$
Let $\Phi$ be the solution of $d\Phi=\Phi\eta$ with $\Phi(z=0)=\Id$
defined on a simply connected domain $U$ with $ \D_1\cup\{1,i\}\subset U.$
Denote
\begin{equation}\label{PQ}
{\mathcal P}:=\Phi(z=1)\quad\text{and}\quad {\mathcal Q}:=\Phi(z=\ii)
\end{equation}
which give the principal solution of the extended frame along the straight lines from $z=0$ to $z=1$ and $z=\ii$, respectively.

 \begin{figure}[h]
\centering
\includegraphics[width=0.75\textwidth]{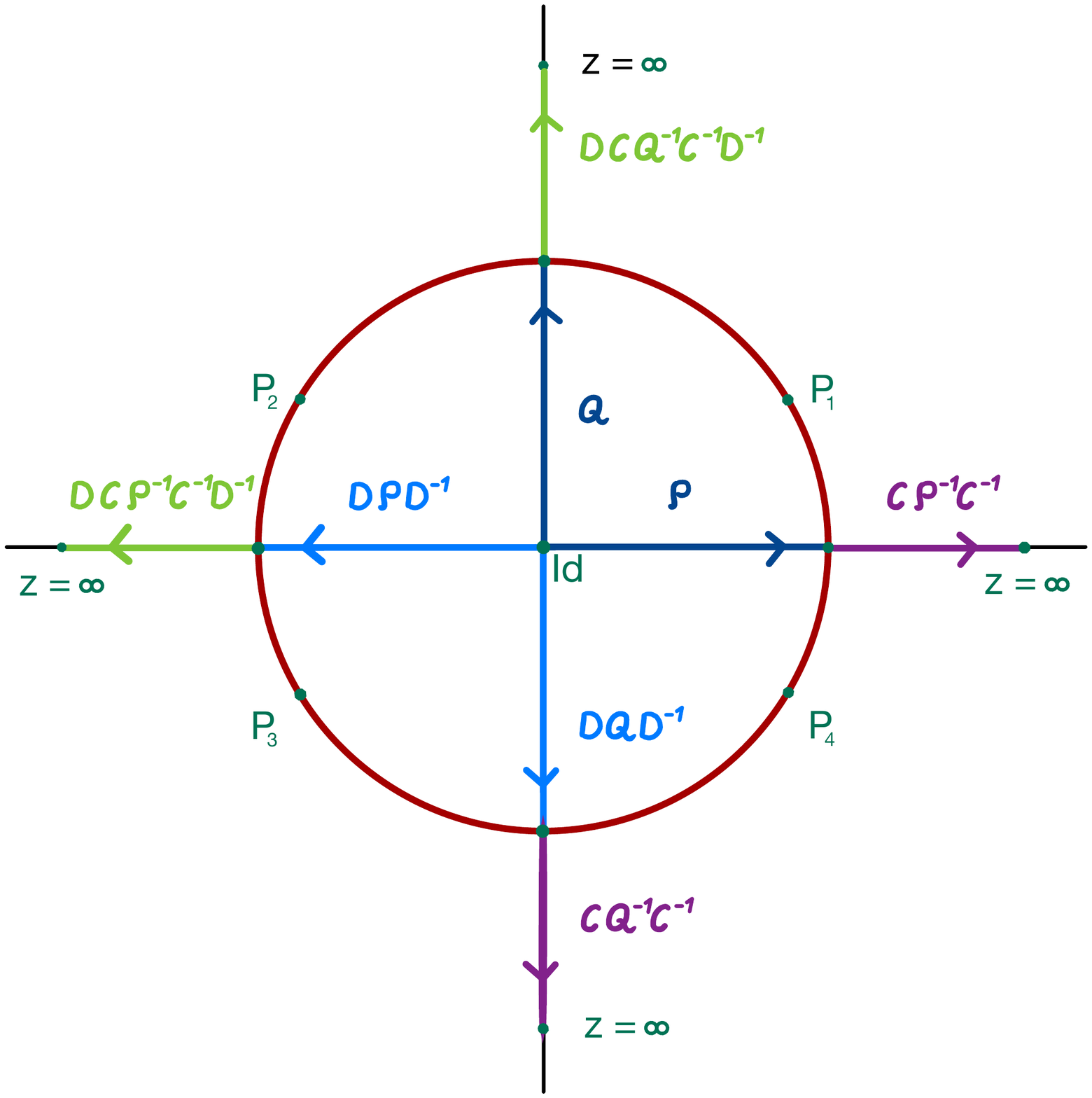}
\hspace{0.25cm}
\caption{
\footnotesize The figure shows a cartoon of $4$-punctured sphere $\Sigma$ with punctures $p_1, ..., p_4,$ where the points $z= \infty$ need to be identified.    The initial value of $\Phi$ at $z=0$ is $\Id$. The principal solution $\Phi$ along the dark blue curves are given by $\mathcal P$ and $\mathcal Q$ respectively. The principal solution along the light blue curves are given by the principal solution along the dark blue curves together with the symmetry $\delta.$ The $x$-axis and $y$-axis each consists of a dark blue, light blue, green and purple line segment, and on $\Sigma$ each axis is homotopic to a curve around two punctures $p_1, p_2$ and $p_2, p_3,$ respectively.
}
\label{fig:mon}
\end{figure}

\begin{proposition}\label{Pro:monoPQ}
The monodromy $M_1M_2$ of $\Phi$ along a simply closed curve around the two punctures $p_1$ and $p_2$ of $\Sigma$
(starting at $z=0$) depends only on $\mathcal P$ and is given by
\[M_1M_2=-(\mathcal PC \mathcal P^{-1}D)^2.\]
The monodromy $M_2M_3$ around $p_2$ and $p_3$ depends only on $\mathcal Q$ and is given by
\[M_2M_3= -(\mathcal QDC\mathcal Q^{-1}D)^2.\]
\end{proposition}

\begin{proof}
Denote $\Phi(+\infty)$, $\Phi(\ii\infty)$, $\Phi(-\infty)$ and $\Phi(-\ii\infty)$
the value of $\Phi$ at $\infty$ obtained by analytic continuation along
the positive real axis, the positive imaginary axis, the negative real axis and the negative imaginary axis, respectively. 
(see Figure \ref{fig:mon}).
Using the $\delta$ and $\tau$-symmetries, we obtain
\[\begin{split}
\Phi(+\infty)&=\mathcal PC\mathcal P^{-1}C^{-1}\\
\Phi(\ii\infty)&=\mathcal QDC\mathcal Q^{-1}C^{-1}D^{-1}\\
\Phi(-\infty)&=D\mathcal PC\mathcal P^{-1}C^{-1}D^{-1}\\
\Phi(-\ii\infty)&=DQD^{-1}CQ^{-1}C^{-1}.
\end{split}\]
This gives
\[\begin{split} M_1M_2&=\Phi(+\infty)\Phi(-\infty)^{-1}=\mathcal PC \mathcal P^{-1}C^{-1}DC\mathcal P C^{-1}\mathcal P^{-1}D^{-1}\\
M_2M_3&=\mathcal QDC\mathcal Q^{-1}C^{-1}D^{-1}C\mathcal QC^{-1}D\mathcal Q^{-1}D^{-1}
\end{split}\]
which simplify to the results of Proposition \ref{Pro:monoPQ} using
$DC=-CD$ and $C^2=D^2=-\Id$.
\end{proof}

\begin{proposition}\label{Pro:trace-square}
With the notations of Lemma \ref{Lem:irreducibility-traces}, we have
\[u= s_{1,2}=s_{3,4}=2-4\left(\mathcal P_{11}\,\mathcal P_{21}-\mathcal P_{12}\,\mathcal P_{22}\right)^2  \quad \text{and} \quad v= s_{1,4}=s_{2,3}=2+4\left(\mathcal Q_{11}\,\mathcal Q_{21}+\mathcal Q_{12}\,\mathcal Q_{22}\right)^2.\]
\end{proposition}

\begin{proof}
Write
\[\mathcal P=\begin{pmatrix}\mathcal P_{11}&\mathcal P_{12}\\ \mathcal P_{21}&\mathcal P_{22}\end{pmatrix}.\]
Then
\[\mathcal P C\mathcal P^{-1}D=\begin{pmatrix}
\mathcal P_{11}\mathcal P_{21}-\mathcal P_{12}\mathcal P_{22} &
\mathcal P_{11}^2-\mathcal P_{12}^2\\
\mathcal P_{21}^2-\mathcal P_{22}^2&
\mathcal P_{11}\mathcal P_{21}-\mathcal P_{12}\mathcal P_{22}
\end{pmatrix}\]
and
\begin{eqnarray*}
\tr\left((\mathcal PC\mathcal P^{-1}D)^2\right)&=&
4\mathcal P_{11}^2\mathcal P_{21}^2+4\mathcal P_{12}^2\mathcal P_{22}^2-2\mathcal P_{12}^2\mathcal P_{21}^2-2\mathcal P_{11}^2\mathcal P_{22}^2-4\mathcal P_{11}\mathcal P_{12}\mathcal P_{21}\mathcal P_{22}\\
&=&4(\mathcal P_{11}\mathcal P_{21}-\mathcal P_{12}\mathcal P_{22})^2-2(\mathcal P_{11}\mathcal P_{22}-\mathcal P_{12}\mathcal P_{21})^2.
\end{eqnarray*}
Hence by Proposition \ref{Pro:monoPQ}
$$u=\tr(M_1M_2)=2-4(\mathcal P_{11}\mathcal P_{21}-\mathcal P_{12}\mathcal P_{22})^2.$$
The proof for $v$ is analogous.
\end{proof}

Instead of using the Fricke-coordinates for the moduli space of representations of $\pi_1(\Sigma,0)$ given by $s$ and $u,v,w$, we switch to coordinates determined by $\mathcal P$ and $\mathcal Q.$ Up to  constants these new coordinates are given by the square root of the holomorphic coordinates  $u,v.$ Remarkably, these new coordinates remain holomorphic by Proposition \ref{Pro:trace-square}. We define
\begin{equation}\label{pq}
\mathfrak p=\mathcal P_{11}\,\mathcal P_{21}-\mathcal P_{12}\,\mathcal P_{22}
\quad\text{ and }\quad
\mathfrak q=\ii\left(\mathcal Q_{11}\,\mathcal Q_{21}+\mathcal Q_{12}\,\mathcal Q_{22}\right).\end{equation}
By Proposition \ref{Pro:trace-square} we have
\begin{equation}
u=2-4\mathfrak p^2\quad\text{and}\quad v=2-4\mathfrak q^2.
\end{equation}
 In view of Lemma \ref{Lem:irreducibility-traces},
we want that $\mathfrak p(\lambda)$ and $\mathfrak q(\lambda)$ are real for
$\lambda\in\S^1$, which can be rewritten using the $*$ operator defined in Section \ref{loop-groups} as $\mathfrak p=\mathfrak p^*$ and
$\mathfrak q=\mathfrak q^*$. These equations are solved in Section \ref{sec:implicit}.
\begin{lemma}
\label{Lem:realsymPQ}
We have $\bar{\mathfrak p}=\mathfrak p$ and $\bar{\mathfrak q}=\mathfrak q$, which means that
\[\overline{\mathfrak p(\bar\lambda)}=\mathfrak p(\lambda)\quad\text{and}\quad \overline{\mathfrak q(\bar\lambda)}=\mathfrak q(\lambda)\]
\end{lemma}
\begin{proof} The result follows from the $\sigma-$symmetry, as
\[\overline{\mathcal P(\bar\lambda)}=\mathcal P(\lambda)\quad\text{ and }\quad
\overline{\mathcal Q(\bar\lambda)}=D^{-1}\mathcal Q(\lambda)D.\]
\end{proof}

\begin{proposition}
\label{Pro:pq}
At $t=0$, the derivatives of $\mathfrak p$ and $\mathfrak q$ with respect to $t$ are given by
$$\mathfrak p'(0)=2\pi r c\quad\text{ and }\quad
\mathfrak q'(0)=2\pi r b.$$
\end{proposition}

\begin{proof} By differentiation of $d_{\Sigma}\Phi_t=\Phi_t\eta_t$, we obtain, since $\Phi_0=\Id,$
\begin{equation}\label{eq:dphiprime}d_{\Sigma}\Phi'\mid_{t=0}=\eta'\mid_{t=0}
=r \sum_{k=1}^4 A_k\frac{dz}{z-p_k}.\end{equation}
To compute the integral of $\eta'$, we introduce the $1$-forms 
\begin{equation}\label{omegaabc}
\begin{split}
\omega_a = \left (\frac{1}{z-p_1} - \frac{1}{z-p_2}+\frac{1}{z-p_3}- \frac{1}{z-p_4}\right)dz\\
\omega_b = \left (\frac{1}{z-p_1} - \frac{1}{z-p_2}-\frac{1}{z-p_3} + \frac{1}{z-p_4}\right)dz\\
\omega_c = \left (\frac{1}{z-p_1} +\frac{1}{z-p_2}-\frac{1}{z-p_3}- \frac{1}{z-p_4}\right)dz\\
 \end{split}
 \end{equation}
 and matrices
 \begin{equation}\label{mabc}
 \mathfrak m_a=\begin{pmatrix} i&0\\0&-i\end{pmatrix}\quad
\mathfrak m_b=\begin{pmatrix} 0&1\\1&0\end{pmatrix}\quad
\mathfrak m_c=\begin{pmatrix} 0&i\\-i&0\end{pmatrix}.
 \end{equation}
 
 Then the DPW potential $\eta_t$ has the form 
 $$\eta_t 
 =r \cdot t  (a\,\omega_a \mathfrak m_a +  b\,\omega_b \mathfrak m_b +  c\, \omega_c \mathfrak m_c).$$
Let
 \begin{equation}\label{Omegaabc}
\Omega_a(z)= \int_0^z \omega_a\qquad
\Omega_b(z) =\int_0^z \omega_b\qquad
\Omega_c(z) = \int_0^z \omega_c,
 \end{equation}
where the integrals are computed on the segment from $0$ to $z$.
Then
\begin{equation*}
\begin{split}
\mathcal P'(0)&=ra\,\Omega_a(1)\mathfrak m_a + rb\,\Omega_b(1)\mathfrak m_b + rc\,\Omega_c(1)\mathfrak m_c\\
\mathcal Q'(0)&=ra\,\Omega_a(\ii)\mathfrak m_a + rb\,\Omega_b(\ii)\mathfrak m_b + rc\,\Omega_c(\ii)\mathfrak m_c
\end{split}
\end{equation*}
Elementary computations give 
\begin{equation}\label{eq:Omega1}
\begin{split}
\Omega_a(1)=\ii(\pi-2\varphi), \quad 
\Omega_b(1)=\log\left(\frac{1-\cos(\varphi)}{1+\cos(\varphi)}\right),\quad
\Omega_c(1)=\ii \pi
\end{split}
\end{equation}
\begin{equation}\label{eq:OmegaI}
\begin{split}
\Omega_a(\ii)=-2\ii\varphi,\quad
\Omega_b(\ii)=-\ii\pi,\quad
\Omega_c(\ii)=\log\left(\frac{1-\sin(\varphi)}{1+\sin(\varphi)}\right).
\end{split}
\end{equation}
Hence
\begin{equation*}
\begin{split}
\mathfrak p'(0)&=\mathcal P_{21}'(0)-\mathcal P_{12}'(0)=-2\ii rc\,\Omega_c(1)=2\pi rc\\
\mathfrak q'(0)&=\ii(\mathcal Q_{21}'(0)+\mathcal Q_{12}'(0))=2\ii rb\,\Omega_b(\ii)=
2\pi rb.
\end{split}
\end{equation*}
\end{proof}
\begin{remark}
The use of the coordinates  $(\mathfrak p,\mathfrak q)$ should be compared with \cite[Section 6]{Gold97}.\end{remark}
\subsection{Solving the Monodromy Problem  for $t\neq 0$}\label{sec:implicit} $\;$\\
If $f$ is an analytic functions of $t$ in a neighborhood of $0$, we denote by $\wh{f}$ the analytic function
\[\wh{f}(t)=\left\{\begin{array}{ll}
\frac{f(t)-f(0)}{t}\quad&\text{ if } t\neq 0\\
$\;$\\
f'(0)&\text{ if } t=0.\end{array}\right.\]

Consider the analytic functions $\wh{\mathfrak p}$ and $\wh{\mathfrak q}$ and  let
\begin{equation}\label{calFG}
\begin{split}
\mathcal K&=\det(r A_1)\\
\mathcal F&=\wh{\mathfrak p}-\wh{\mathfrak p}^*\\
\mathcal G&=\wh{\mathfrak q}-\wh{\mathfrak q}^*\\
\mathcal H_1&=\wh{\mathfrak p}\mid_{\lambda=e^{i\theta}}\\
\mathcal H_2&=\wh{\mathfrak q}\mid_{\lambda=e^{i\theta}}.\\
\end{split}
\end{equation}
Thanks to Lemma \ref{Lem:irreducibility-traces}, the Monodromy Problem
\eqref{monodromy-problem2} can be reformulated as
\begin{equation}\label{monodromy-problem3}
\left\{\begin{array}{l}
(i)\quad\;\;\mathcal F=\mathcal G= 0\\ 
(ii)\quad\;\mathcal K=-1.\\
(iii)\quad\exists \theta\in(0,\pi):\;\mathcal H_1=\mathcal H_2=0.
\end{array}\right.
\end{equation}
\begin{lemma}\label{lem:monodromyproblemequiv}
For $t\sim 0$ small enough and $\xx$ close to $\cv{\xx}$, if $(t,\xx)$
solves Problem \eqref{monodromy-problem3}, then $\Phi_{t,\varphi}^{\xx}$ solves the Monodromy Problem \eqref{monodromy-problem2} for a suitable choice of the initial condition $\Phi_{t,\varphi}^{\xx}\mid_{z=0}=\Phi_0\in\Lambda\SL(2,\C)$.
 Moreover, the initial condition $\Phi_0$ is diagonal.
\end{lemma}

\begin{proof}
We use Lemma \ref{Lem:irreducibility-traces}.
Since $\mathcal K=-1$, $rA_1$ has eigenvalues $\pm 1$. Hence $M_1$ has eigenvalues $e^{\pm 2\pi\ii t}$, so $s=\tr(M_1)=2\cos(2\pi t)\in(-2,2)$.
The condition $\mathcal F=\mathcal G=0$ gives that $\wh{\mathfrak p}$ and
$\wh{\mathfrak q}$ are real on the unit circle.
At $(t,\xx)=(0,\cv{\xx})$, we have by Proposition \ref{Pro:pq}
\[\wh{\mathfrak p}\mid_{(t,\xx)=(0,\cv{\xx})}=2\pi\cv{c}=-\pi\cos(\varphi)(\lambda^{-1}+\lambda)\]
\[\wh{\mathfrak q}\mid_{(t,\xx)=(0,\cv{\xx})}=2\pi\cv{b}=-\pi\sin(\varphi)(\lambda^{-1}+\lambda)\]
with simple zeros at $\pm\ii$.
Hence for $(t,\xx)$ close to $(0,\xx)$, $\wh{\mathfrak p}$ and $\wh{\mathfrak q}$ both have two simple zeros in the annulus $\A_{\rho}$.
The condition $\mathcal H_1=\mathcal H_2=0$ gives that $\wh{\mathfrak p}$
and $\wh{\mathfrak q}$ both have two zeros at $e^{i\theta}$ and $e^{-i\theta}$
by symmetry. Hence they have no further zeros in $\A_{\rho}$.
By Proposition \ref{Pro:trace-square}, for $t\neq 0$ small enough, $s_{1,2}$ and $s_{2,3}$ lie in the interval $(-2,2)$ along the unit circle, except at the two Sym points where they are equal to $2$.
  
In order to obtain that the corresponding monodromy representation is unitarizable
it remains to show that the traces $s_{1,3}$ and $s_{2,4}$ are real
on the unit circle.
Recall that $w=s_{1,3}$ is a solution of the quadratic equation $E(u,v,w)=0$
given by \eqref{goldpol},
and $\tilde w=s_{2,4}$ is a solution of the same equation. Indeed, by cyclic permutation we have
$E(s_{2,3},s_{3,4},s_{2,4})=0$ and by symmetry we have $s_{1,2}=s_{3,4}.$ So since the polynomial $E$ is symmetric, $E(s_{1,2},s_{2,3},s_{2,4})=0$.

Let $\Delta$ be the discriminant of \eqref{goldpol}. Note that $\Delta$ is real on the unit
circle $\S^1$. Since $s_{1,3}$ is a well-defined holomorphic function of $\lambda$ in a neighborhood of $\S^1$, $\sqrt{\Delta}$ is a well-defined holomorphic function of $\lambda$.
Hence its zeros must have even order, so $\Delta$ does not change sign on the unit circle.
At $\lambda=e^{\ii\theta}$ we have $u=v=2$ so $\Delta=16-16\cos^2(2\pi t)>0$.
Hence $\Delta\geq 0$ on $\S^1$ and the solutions $s_{1,3}$ and $s_{2,4}$ are real.
By Lemma \ref{Lem:irreducibility-traces}, the monodromy representation is pointwise unitarizable on the unit circle except at a finite number of points (the Sym-points),
and the unitarizer is diagonal.
By  \cite[Theorem 9.1]{SKKR}, the monodromy is globally unitarizable on the unit circle
with a diagonal unitarizer $\Phi_0\in\Lambda SL(2,\C)$.

It remains to prove that the monodromy representation is diagonal at the Sym points
$\lambda=e^{\pm\ii \theta}$.
With the notations of Proposition \ref{Pro:monoPQ}, we compute
\[\Phi(+\infty)=\mathcal PC\mathcal P^{-1}C^{-1}=\begin{pmatrix}\mathcal P_{11}^2-\mathcal P_{12}^2&-\mathfrak p\\\mathfrak p&\mathcal P_{22}^2-\mathcal P_{21}^2\end{pmatrix}\]
\[\Phi(\ii\infty)=\mathcal Q DC\mathcal Q^{-1}C^{-1}D^{-1}=\begin{pmatrix}\mathcal Q_{11}^2+\mathcal Q_{12}^2&-\ii\mathfrak q\\ -\ii\mathfrak q& \mathcal Q_{22}^2+\mathcal Q_{21}^2\end{pmatrix}.\]
Hence both are diagonal at the Sym points.
This implies that $\Phi(-\infty)$ and $\Phi(-\ii\infty)$ are diagonal, hence all $M_k$
are diagonal at the Sym points.
\end{proof}

\begin{lemma}\label{Lem:realFGK}
The following identities hold
\[\mathcal F=-\mathcal F^* \quad \text{and}\quad \mathcal G=-\mathcal G^*,\]
\[\overline{\mathcal F}=\mathcal F \quad \text{and}\quad \overline{\mathcal G}=\mathcal G.\]
\[\overline{\mathcal K}=\mathcal K.\]
 In particular, the equations $\mathcal F^+=0$ and  $\mathcal F= 0$ (respectively $\mathcal G^+= 0$ and $\mathcal G= 0$) are equivalent. 

\end{lemma}
\begin{proof}
These properties directly  follow from the symmetries $\delta$, $\tau$ and $\sigma$ and the resulting equations for $\mathfrak p$ and  $\mathfrak q$ in Lemma \ref{Lem:realsymPQ}.
\end{proof}

\begin{proposition}\label{Fuchsianpotential} Let $\varphi\in(0,\tfrac{\pi}{2}).$
For $t\sim0$ there is a unique family $\eta_{t,\varphi}^{\xx(t)}$ of Fuchsian DPW potentials of the form \eqref{thegeneraleta} solving the Monodromy Problem \eqref{monodromy-problem3} with $\xx(0) = \underline{\xx}$.
\end{proposition}

\begin{proof}
By Proposition \ref{Pro:pq}, we have at $t=0$
\[\begin{split}
\wh{\mathfrak p}\mid_{t=0}&=2\pi rc=2\pi r(\cv{c}+\wt c)\\
\wh{\mathfrak q}\mid_{t=0}&=2\pi r b =2\pi r(\cv{b}+\wt{b})\\
\mathcal F\mid_{t=0}&=2\pi r(\wt c-\wt c^*)\\
\mathcal G\mid_{t=0}&=2\pi r(\wt b-\wt b^*)\\
\mathcal H_1\mid_{t=0}&=2\pi r\left(\cv{c}(e^{\ii\theta})+\wt{c}(e^{\ii\theta})\right)
=2\pi r\left(-\cos(\varphi)\cos(\theta)+\wt{c}(e^{\ii\theta})\right)\\
\mathcal H_2\mid_{t=0}&=2\pi r\left(\cv{b}(e^{\ii\theta})+\wt{b}(e^{\ii\theta})\right)
=2\pi r\left(-\sin(\varphi)\cos(\theta)+\wt{b}(e^{\ii\theta})\right).
\end{split}\]
So the differential of $(\mathcal F^+,\mathcal G^+,\mathcal H_1,\mathcal H_2)$ with respect to $\xx$ at $(t,\xx)=(0,\cv{\xx})$ is given by
\[\begin{split}
d\mathcal F^+&=2\pi d \wt c^+\\
d\mathcal G^+&=2\pi d \wt b^+\\
d\mathcal H_1&=2\pi\cos(\varphi)d\theta+2\pi d\wt c(\ii)\\
d\mathcal H_2&=2\pi\sin(\varphi)d\theta+2\pi d\wt b(\ii).
\end{split}\]
The partial derivative  with respect to $(\wt b,\wt c)$ is clearly an
isomorphism from $(\mathcal W^{\geq 0}_{\R})^2$ to $(\mathcal W^+_{\R})^2\times\R^2$.
Therefore, there exist by the implicit function theorem for $(t,\widetilde{a},r,\theta)$ in a neighborhood of $(0,0,1,\pi/2)$ unique values of $\wt b$ and $\wt c$
in $\mathcal W^{\geq 0}_{\R}$ solving $\mathcal F^+=\mathcal G^+=0$ and
$\mathcal H_1=\mathcal H_2=0$.
Moreover, the differential of $\wt b$ and $\wt c$ with respect to the remaining parameters
$(\wt a,\theta)$ is given by
$$d\wt b=d\wt b^0=-\sin(\varphi)d\theta\quad\text{ and }\quad
d\wt c=d\wt c^0 = -\cos(\varphi) d\theta.$$
Then the differential of $\mathcal K$ with respect to $(\wt a,r,\theta)$ at
$(t,\wt a,r,\theta)=(0,0,1,\pi/2)$ is
$$d\mathcal K=2\cv{a}\,d\wt{a}-2\cv{b}\,d\wt{b}-2\cv{c}\,d\wt{c}-2dr
=(\lambda^{-1}-\lambda)d\wt{a}-(\lambda^{-1}+\lambda)d\theta -2dr.$$
Observe that $\lambda\mathcal K\in\mathcal W^{\geq 0}_{\R}$.
Using Proposition \ref{Pro:decomposition} with $(\mu_1,\mu_2)=(1,-1)$, we decompose $\mathcal K$ as
$$\mathcal K=\lambda^{-1}\mathcal K_{-1}+
\lambda^0 \mathcal K_0
+(\lambda^{-1}-\lambda) \wc{\mathcal K}$$
with $\wc{K}\in\mathcal W^{\geq 0}_{\R}$.
Then rewriting $d\mathcal K$ as
$$d\mathcal K=-2\lambda^{-1}d\theta-2 dr +(\lambda^{-1}-\lambda)(d\wt{a}+d\theta)$$
we obtain
\[\begin{split}
d\mathcal K_{-1}&=-2d\theta\\
d\mathcal K_0&=-2dr\\
d\wc{\mathcal K}&=d\wt{a}+d\theta.\end{split}\]
The partial differential of $(\mathcal K_{-1},\mathcal K_0,\wc{\mathcal K})$
with respect to $(r,\theta,\wt{a})$ is clearly an automorphism
of $\R^2\times\mathcal W^{\geq 0}_{\R}$.
By the implicit function theorem, for $t\sim 0$, there exists a unique 
$(\wt{a},r,\theta)\in\mathcal W^{\geq 0}_{\R}\times\R^2$ in a neighhorhood of
$(0,1,\pi/2)$ such that $\mathcal K=-1$.
\end{proof}

\begin{remark}
Since the differential of Monodromy Problem \eqref{monodromy-problem3} remains invertible for  $\varphi\rightarrow0$ and $\varphi\rightarrow\pi/2$,  the proof of the Proposition works for all $\varphi \in [0, \tfrac{\pi}{2}]$ and gives existence to a unique family of DPW potentials $\eta_t$ solving the Monodromy Problem \eqref{monodromy-problem3},
provided the functions $\mathcal F,$ $\mathcal G$, $\mathcal K$ and $\mathcal H_1,$ $\mathcal H_2$ extend smoothly to $\varphi \rightarrow 0$ and $\varphi\rightarrow\pi/2$.
We will prove that this is the case in Section \ref{limitvarphi}.
\end{remark}

\begin{proposition}$\,$\label{propo:symmetries}\begin{enumerate}
\item The solution $\xx(t)$ given by Proposition \ref{Fuchsianpotential} has the following
parity with respect to $t$:
\begin{equation*}
\begin{split}
&-a(-t)(-\lambda)=a(t)(\lambda),\quad
-b(-t)(-\lambda)=b(t)(\lambda),\quad
-c(-t)(\lambda)=c(t)(\lambda)\\
&r(-t)=r(t)
\quad\text{and}\quad
\pi-\theta(-t)=\theta(t).
\end{split}
\end{equation*}
\item As a function of $\varphi$, $\xx(t,\varphi)$ has the following symmetries:
\begin{equation*}
\begin{split}
&a(-t,\tfrac{\pi}{2}-\varphi)=a(t,\varphi),\quad
b(-t,\tfrac{\pi}{2}-\varphi)=c(t,\varphi),\quad
c(-t,\tfrac{\pi}{2}-\varphi)=b(t,\varphi)\\
&r(-t,\tfrac{\pi}{2}-\varphi)=r(t,\varphi)
\quad\text{and}\quad
\theta(-t,\tfrac{\pi}{2}-\varphi)=\theta(t,\varphi).
\end{split}
\end{equation*}
\end{enumerate}
In particular, $\theta(t,\tfrac{\pi}{4})=\tfrac{\pi}{2}$ for all $t$, so $\varphi=\tfrac{\pi}{4}$ yields a family of minimal surfaces.
\end{proposition}
\begin{proof}
Given the parameter $\xx$, consider the parameter $\other{\xx}$ defined by
$$\other{a}(\lambda)=-a(-\lambda),\quad
\other{b}(\lambda)=-b(-\lambda),\quad
\other{c}(\lambda)=-c(-\lambda),\quad
\other{r}=r\quad\text{and}\quad
\other{\theta}=\pi-\theta.$$
Note that $\other{\cv{\xx}}=\cv{\xx}$.
We then have by inspection
\begin{align*}
&\eta_{-t}^{\other{\xx}}(-\lambda)=\eta_{t}^{\xx}(\lambda)\\
&\Phi_{-t}^{\other{\xx}}(-\lambda)=\Phi_{t}^{\xx}(\lambda)\\
&\widehat{\mathfrak p}(-t,\other{\xx})(-\lambda)=-\widehat{\mathfrak p}(t,\xx)(\lambda)\\
&\mathcal F(-t,\other{\xx})(-\lambda)=-\mathcal F(t,\xx)(\lambda)\\
&\mathcal H_1(-t,\other{\xx})
=\widehat{\mathfrak p}(-t,\other{\xx})(e^{\ii\other{\theta}})
=-\widehat{\mathfrak p}(t,\xx)(-e^{\ii(\pi-\theta)})
=-\widehat{\mathfrak p}(t,\xx)(e^{-\ii\theta})
=-\overline{\mathcal H_1(t,\xx)}\\
&\mathcal K(\other{\xx})(\lambda)=\mathcal K(\xx)(-\lambda)
\end{align*}
and similar statements hold for $\widehat{\mathfrak q}$, $\mathcal G$ and $\mathcal H_2$.
Therefore, if $(t,\xx)$ solves the Monodromy Problem \eqref{monodromy-problem3} then $(-t,\other{\xx})$ also solves the Monodromy Problem.
By uniqueness in the implicit function theorem, $\other{\xx}(-t)=\xx(t)$ and
point 1 follows.
\medskip

To prove point (2), consider this time the parameter $(\other{\xx},\other{\varphi})$ defined by
$$\other{a}=a,\quad
\other{b}=c,\quad
\other{c}=b,\quad
\other{r}=r',\quad
\other{\theta}=\theta
\quad\text{and}\quad
\other{\varphi}=\tfrac{\pi}{2}-\varphi.$$
Note that $\cv{\other{\xx}}(\other{\varphi})=\cv{\xx}(\varphi)$.
Let $\other{p}_1$, $\other{p}_2$, $\other{p}_3$, $\other{p}_4$ be the punctures corresponding to the angle
$\other{\varphi}$ and $\other{\omega}_a$, $\other{\omega}_b$,
$\other{\omega}_c$ be the corresponding 1-forms.
Define $\iota(z)=\ii z$. Then $\iota(p_j)=\other{p}_{j+1}$, where the indices are considered mod 4. This gives
$$\iota^*(\other{\omega}_a,\other{\omega}_b,\other{\omega}_c)
=(-\omega_a,-\omega_c,\omega_b).$$
Therefore, we obtain for the pull-back of the potential under $\iota$
\begin{eqnarray*}
\iota^*\left(\eta_{-t,\other{\varphi}}^{\other{\xx}}\right)
&=&-t\begin{pmatrix} \ii a (-\omega_a)&
c(-\omega_c)+\ii b \omega_b\\
c(-\omega_c)-\ii b\omega_b& -\ii a(-\omega_a)\end{pmatrix}
=S^{-1}\eta_{t,\varphi}^{\xx}S
\end{eqnarray*}
with

$$S=\begin{pmatrix}e^{\ii\pi/4}&0\\0&e^{-\ii\pi/4}\end{pmatrix}.$$

The same holds for the extended frame

$$\iota^*\Phi_{-t,\other{\varphi}}^{\other{\xx}}=S^{-1}\Phi_{t,\varphi}^{\xx}S$$
resulting in 
$$\mathcal Q(-t,\other{\varphi},\other{\xx})=S^{-1} \mathcal P(t,\varphi,\xx)S,\qquad
\mathcal P(-t,\other{\varphi},\other{\xx})=S\mathcal Q(t,\varphi,\xx)S^{-1},$$
which gives
$$\mathfrak p(-t,\other{\varphi},\other{\xx})=-\mathfrak q(t,\varphi,\xx),\qquad
\mathfrak q(-t,\other{\varphi},\other{\xx})=-\mathfrak p(t,\varphi,\xx).$$

Consequently, if $(t,\varphi,\xx)$ solves the Monodromy Problem \eqref{monodromy-problem3} then $(-t,\other{\varphi},\other{\xx})$ solves the Monodromy Problem as well and point (2) also follows by the uniqueness in the implicit function theorem.
\end{proof}

\subsection{Building the surface}
We consider a genus $g\in\N$ large enough and set 
\begin{equation}
\label{eq:t}
t=\frac{1}{2g+2}.
\end{equation}
In this section, we prove that the lift of the potential $\eta_{t,\varphi}^{\xx(t,\varphi)}$ to
the compact Riemann surface $M_{\varphi}^g$ has apparent singularities and defines a closed CMC immersion.
More precisely:
\begin{theorem}\label{thm:construction-fixed-varphi}
Let $\varphi\in(0,\tfrac{\pi}{2})$ be fixed.
For every $g \gg1$ there exist a conformal immersion $f^g_\varphi \colon M^g_\varphi \longrightarrow \S^3$ of constant mean curvature such that, with $t$ given
by \eqref{eq:t}:
\begin{enumerate}
\item $f^g_{\varphi}$ has mean curvature
$H=\operatorname{cotan}(\theta(t,\varphi)).$
\item The area of $f^g_{\varphi}$ is given by
$$\operatorname{Area}(f^g_{\varphi})=8\pi\left[1-r(t,\varphi)\left(\cos(\varphi)b_0(t,\varphi)-\sin(\varphi)c_0(t,\varphi)\right)\right].$$
\item  $f^g_\varphi$ is equivariant with respect to the symmetries $\delta$, $\tau$ and $\sigma.$
\item $f^g_\varphi$ and $f^g_{\tfrac{\pi}{2} - \varphi}\circ\iota$ are congruent  by an orientation reversing isometry of $\S^3$, where $\iota$ is determined by $\iota(z)=\ii z$.
\item The image of $f^g_{\tfrac{\pi}{4}}$ is the Lawson minimal surface $\xi_{1,g}$ of genus $g$.
\end{enumerate}
\end{theorem}
\begin{remark}
We will show in Proposition \ref{prop:willmore} that the Willmore energy of $f^g_{\varphi}$ is below $8\pi$ for large $g$, ensuring that $f^g_{\varphi}$ is embedded.
\end{remark}

\begin{proof}[Proof of Theorem \ref{thm:construction-fixed-varphi}]%
$\;$
\begin{itemize}
\item The Riemann surface $M^g_\varphi$ is given as an algebraic curve  by the equation
$$y^{g+1} = \frac{(z-p_1)(z-p_3)}{(z-p_2)(z-p_4)}.$$
Let $\pi\colon M^g_\varphi\to\C P^1$ denote the projection onto the $z-$plane, namely
$\pi(y,z)=z$.
Then $\pi$ is a $(g+1)$-sheeted covering totally branched over $p_1,\cdots,p_4$.
Let $\widetilde{p}_i=\pi^{-1}(p_i),$ $\widetilde{\eta}_t=\pi^*\eta_{t, \varphi}^{\xx(t,\varphi)}$ and
$\widetilde{\Phi}_t$ be the solution of $d\widetilde{\Phi}_t=\widetilde{\Phi}_t\widetilde{\eta}_t$ with initial condition $\widetilde{\Phi}_t(\widetilde{0})=\Phi_0(t,\varphi)$,
where $\pi(\widetilde{0})=0$ and
$\Phi_0$ is given by Lemma \ref{lem:monodromyproblemequiv}.
We first prove that $\widetilde{\Phi}_t$ solves the Monodromy Problem \eqref{monodromy-problem}. 
Let $\gamma\in\pi_1(M^g_\varphi\setminus\{\widetilde{p}_1,\cdots,\widetilde{p}_4\},\widetilde{0})$. We decompose $\pi\circ\gamma$ as
$$\pi\circ\gamma=\prod_{j=1}^l(\gamma_{i_j})^{k_j}$$
for some integers $l, k_1, ..., k_l \in \N$ satisfying
\begin{equation}
\label{eq:congruence}\sum_{j=1}^l (-1)^{i_j} k_j\equiv 0\mod (g+1).
\end{equation}
Equation \eqref{eq:congruence} comes from the condition that the closed curve $\pi \circ\gamma$ on the quotient is induced from a closed curve $\gamma$ on $M^g_\varphi,$ i.e., it stems from the monodromy of the covering map $\pi$,
see \cite[Section 4.2]{HHT}.
Then
$$\mathcal M(\widetilde{\Phi}_t,\gamma)=\prod_{j=1}^l M_{i_j}(t)^{k_j}.$$
and $\mathcal M(\widetilde{\Phi}_t,\gamma)\in\Lambda SU(2)$.
Moreover, at the Sym points, all $M_j$ are diagonal and commute with each other, with
$M_1=M_2^{-1}=M_3=M_4^{-1}.$
Hence $\mathcal M(\widetilde{\Phi}_t,\gamma)\mid_{\lambda=\lambda_1,\lambda_2}$ is diagonal and has eigenvalues
$$\exp\left(\pm 2\pi\ii t\sum_{j=1}^l (-1)^{i_j} k_j\right)=\pm 1$$
thanks to Equations \eqref{eq:t} and \eqref{eq:congruence}.
Hence the Monodromy Problem \eqref{monodromy-problem} is solved
so the Sym-Bobenko formula defines an immersion $f^g_\varphi$ on
$M^g_{\varphi}\setminus \{\widetilde{p}_1,\cdots,\widetilde{p}_4\}$ with mean curvature 
$$H^g_\varphi= \ii\frac{\lambda_1+\lambda_2}{\lambda_1-\lambda_2}
=\frac{e^{i \theta} + e^{-i\theta}}{e^{i \theta} - e^{-i\theta}}=\operatorname{cotan}(\theta(t,\varphi)).$$
\item Next we prove that $\widetilde{p}_1,\cdots,\widetilde{p}_4$ are apparent singularities.
Let $w$ be a local coordinate in a neighborhood of $\widetilde{p}_1$ such that
$w^{g+1}=z-p_1$. Then using Equation \eqref{eq:t}
$$\widetilde{\eta}_t=r\,t\,(g+1)\frac{dw}{w}A_1+O(w^g)dw
=\frac{r}{2}\begin{pmatrix}\ii a&b+\ii c\\b-\ii c&-\ii a\end{pmatrix}
\frac{dw}{w}+O(w^g)dw.$$

Consider the local gauge
$$G_1=\begin{pmatrix}k&0\\-1&k^{-1}\end{pmatrix}\begin{pmatrix}
w^{-1/2}&0\\0&w^{1/2}\end{pmatrix}
\quad\text{with}\quad
k=\frac{b+\ii c}{\ii a-r^{-1}}.$$
A computation gives
$$\widetilde{\eta}_t . G_1=\begin{pmatrix}0&\frac{(\ii a r-1)^2}{2 r(b+\ii c)}\\
-\frac{r(b+\ii c)(1+r^2(a^2-b^2-c^2))}{2(\ii a r- 1)^2 w^2}&0\end{pmatrix}
+G_1^{-1} O(w^g)G_1 dw$$
which is holomorphic at $w=0$ thanks to $r^2(a^2-b^2-c^2)=-1$.
Moreover,  the $\lambda^{-1}$ term of $\widetilde{\eta}_t . G_1$ is
non-zero at $t=0$, and therefore remains non-zero for $t$ small enough.
This ensures that $f^g_\varphi$ extends to a regular immersion at $\widetilde{p}_1$,
and at $\widetilde{p}_2$, $\widetilde{p}_3$ and $\widetilde{p}_4$ using similar gauges
(it suffices to multiply $a$, $b$, $c$ by the proper $\pm$ signs).

\begin{remark} At $t=0$, we have
$$k=e^{-\ii\varphi}\frac{\lambda+\ii}{\lambda-\ii}$$
so $G_1$ has poles at $\lambda=\pm\ii$. For small $t\neq 0$, $G_1$ may have poles in the unit $\lambda$-disk close to $\pm \ii$. Therefore, we need to apply the $r$-Iwasawa decomposition instead of the ordinary Iwasawa decomposition. This does not alter the corresponding immersion.
\end{remark}

\item By  \cite[Corollary 17]{HHT}, the area of a compact minimal or CMC immersion $f$ in the 3-sphere constructed from a meromorphic DPW potential $\eta$ is given
by
\begin{equation}\label{areaasresidue}
\operatorname{Area}(f)=4\pi\sum_{j=1}^n\operatorname{Res}_{q_j}\operatorname{trace}(\eta_{-1}G_{j,1} G_{j,0}^{-1}),\end{equation}
where $q_1,\dots,q_n$ are the poles of the potential $\eta=\sum_{k=-1}^{\infty}\eta_k\lambda^k$ and $G_j=\sum_{k=0}^\infty G_{j,k}\lambda^k$ is a local gauge defined near $q_j$ such that
$\eta.G_j$ is regular at $q_j$.\\
As our gauges involve square roots,
we need to work on a certain double covering where the gauges are well-defined. Hence,
in order to apply \cite[Corollary 17]{HHT} directly, we work on a double covering
\[\widehat M^g_\varphi\to  M^g_\varphi\]
branched at $\widehat{p}_1,\cdots,\widehat{p}_4$ and let $\widehat{\eta}_t$ be the pullback of $\widetilde{\eta}_t$ to $\widehat M^g_\varphi$.
Using $v=\sqrt{w}$ as local coordinate in a neighborhood of $\widehat{p}_1$, we have
$$\widehat{\eta}_{-1}=r\begin{pmatrix}\ii a_{-1}& b_{-1}+\ii c_{-1}\\
b_{-1}-\ii c_{-1}&-\ii a_{-1}\end{pmatrix}\frac{dv}{v}$$
$$G_{1,1}G_{1,0}^{-1}=\begin{pmatrix}k_1&0\\0&-k_1k_0^{-2}\end{pmatrix}
\begin{pmatrix}k_0^{-1}&0 \\ 1 &k_0\end{pmatrix}
=\frac{k_1}{k_0}\begin{pmatrix}1&0\\ -k_0^{-1}&-1\end{pmatrix}$$
$$k_0=\frac{b_{-1}+\ii c_{-1}}{\ii a_{-1}},\qquad
\frac{k_1}{k_0}=\frac{b_0+\ii c_0}{b_{-1}+\ii c_{-1}}-\frac{\ii a_0-r^{-1}}{\ii a_{-1}}$$
\begin{eqnarray*}
\operatorname{Res}_{\widehat{p}_1}\operatorname{trace}\left(\widehat{\eta}_{-1}G_{1,1} G_{1,0}^{-1}\right)&=&r\ii a_{-1}\frac{k_1}{k_0}=
1+r\left(-\ii a_0+\ii a_{-1}\frac{b_0+\ii c_0}{b_{-1}+\ii c_{-1}}\right)\\
&=&1+r\left(-\ii a_0-e^{\ii\varphi}(b_0+\ii c_0)\right).
\end{eqnarray*}
By substitution of $(a,b,c)$ with $(-a,-b,c)$, $(a,-b,-c)$ and $(-a,b,-c)$ respectively, we obtain
$$\operatorname{Res}_{\widehat{p}_2}\operatorname{trace}\left(\widehat{\eta}_{-1}G_{2,1}G_{2,0}^{-1}\right)=1+r\left(\ii a_0+e^{-\ii\varphi}(-b_0+\ii c_0)\right)$$
$$\operatorname{Res}_{\widehat{p}_3}\operatorname{trace}\left(\widehat{\eta}_{-1}G_{3,1}G_{3,0}^{-1}\right)=1+r\left(-\ii a_0-e^{\ii\varphi}(b_0+\ii c_0)\right)$$
$$\operatorname{Res}_{\widehat{p}_4}\operatorname{trace}\left(\widehat{\eta}_{-1}G_{4,1}G_{4,0}^{-1}\right)=1+r\left(\ii a_0+e^{-\ii\varphi}(-b_0+\ii c_0)\right).$$
Adding all four residues, we obtain, since we are working on a double cover,
$$2\operatorname{Area}(f^g_{\varphi})=16\pi\left(1-r(b_0\cos(\varphi)-c_0\sin(\varphi))\right).$$
\item That the surface is equivariant with respect to the symmetries $\delta$ and $\sigma$ follows by standard arguments from the fact that these symmetries fix the base point $z=0$ and that the unitarizer $\Phi_0$ commutes with $D$.
Understanding the symmetry $\tau$ is a little more involved.
Let $\pi:\widetilde{\Sigma}\to\C P^1\setminus\{p_1,\cdots,p_4\}$ be the universal cover,
$\widetilde{\eta}_t=\pi^*\eta_{t, \varphi}^{\xx(t,\varphi)}$ and
$\widetilde{\Phi}_t$ be the solution of $d\widetilde{\Phi}_t=\widetilde{\Phi}_t\widetilde{\eta}_t$ with initial condition $\widetilde{\Phi}_t(\widetilde{0})=\Phi_0(t,\varphi)$.
Fix $\widetilde{1}$ lying in the fiber $\pi^{-1}(1)$.
 Then
$$d(C\tau^*\widetilde{\Phi}_t C^{-1})=(C\tau^*\widetilde{\Phi}_t C^{-1}) (C\tau^* \widetilde{\eta}_tC^{-1})=(C\tau^*\widetilde{\Phi}_t C^{-1})\widetilde{\eta}_t.$$
Hence
$$C\tau^*\widetilde{\Phi}_t C^{-1}= M \widetilde{\Phi}_t\quad\text{with}\quad
M=C\widetilde{\Phi}_t(\widetilde{1})C^{-1}\widetilde{\Phi}_t(\widetilde{1})^{-1},$$
as both sides have the same value at $z=\widetilde{1},$ they are solutions to the same Cauchy Problem.
For any closed curve $\gamma$ we have
$$\mathcal M(C\tau^*\widetilde{\Phi}_tC^{-1},\gamma)=C\mathcal M(\widetilde{\Phi}_t,\tau\circ\gamma)C^{-1}\in\Lambda SU(2).$$
Since the unitarizer is unique modulo loops in SU(2), see Lemma \ref{Lem:uniuni} below, we conclude that $M\in\Lambda SU(2)$.
Observe that
$$C=PDP^{-1}\quad\text{with}\quad P=\frac{1}{\sqrt{2}}\begin{pmatrix}1&-1\\1&1\end{pmatrix}\in SU(2).$$
Let $(F,B)$ be the Iwasawa decomposition of $P^{-1}\widetilde{\Phi}_t(\widetilde{1})P.$
Then
$$M=C(PFBP^{-1})C^{-1}(PB^{-1}F^{-1}P^{-1})=(CPFD^{-1})(DBD^{-1}B^{-1})(F^{-1}P^{-1}).$$
Hence $DBD^{-1}B^{-1}\in\Lambda SU(2)$. On the other hand, $DBD^{-1}B^{-1}
\in\Lambda_+^{\R} SL(2,\C)$, so $DB=BD$.
Write
$$\widetilde{\Phi}_t(\widetilde{1})=UR\quad\text{with}\quad
U=PFP^{-1}\in\Lambda SU(2)\quad\text{and}\quad
R=PBP^{-1}.$$
From $DB=BD$ we conclude that $RC=CR$.
Up to an isometry of $\S^3$, we may assume that $F=I_2$. Then
$$\tau^*\widetilde{\Phi}_t=C^{-1}\widetilde{\Phi}_t C$$
yielding the isometry $(x_1,x_2,x_3,x_4)\to (x_1,-x_2,-x_3,x_4)$
of $\S^3$. The corresponding isometry of $f_{\varphi}^g$ is conjugated by $F$,
which depends on the choice of $\widetilde{1}$ in the fiber $\pi^{-1}(1)$ (so there are in fact $(g+1)$ such isometries).
\item To prove point (4), set $\widehat{\varphi}=\frac{\pi}{2}-\varphi$ and define $\other{\xx}(t)$ by
\begin{equation*}
\begin{split}
&\other{a}(t)(\lambda)=-a(t)(-\lambda),\quad
\other{b}(t)(\lambda)=-c(t)(-\lambda),\quad
\other{c}(t)(\lambda)=-b(t)(-\lambda)\\
&\other{r}(t)=r(t)\quad\text{and}\quad
\other{\theta}(t)=\pi-\theta(t).
\end{split}
\end{equation*}
By Proposition \ref{propo:symmetries}, we have
$$\iota^*\Phi_{t,\other{\varphi}}^{\other{\xx}(t)}(\lambda)=S^{-1}\Phi_{t,\varphi}^{\xx(t)}(-\lambda)S.$$
Again, by uniqueness of the unitarizer $\Phi_0(t,\varphi)$, we obtain
$$\Phi_0(t,\other{\varphi})(\lambda)=US^{-1}\Phi_0(t,\varphi)(-\lambda)S\quad\text{with}\quad U\in\Lambda SU(2).$$
This gives for the solution $\widetilde{\Phi}_{t,\varphi}$ with initial condition $\Phi_0(t,\varphi)$ at $z=\widetilde{0}$ in the universal cover, with $\wt \iota$  being the lift of the involution $\iota$ to $\widetilde{\Sigma}$ fixing $\wt 0$,
$$\widetilde{\iota}^*\widetilde{\Phi}_{t,\other{\varphi}}(\lambda)=US^{-1}\widetilde{\Phi}_{t,\varphi}(-\lambda)S.$$
We have 
$$\other{\lambda}_1(t)=e^{\ii\other{\theta}_1(t)}=e^{\ii(\pi-\theta_1(t))}=-\lambda_2(t)\quad\text{and}\quad \other{\lambda}_2(t)=-\lambda_1(t).$$
By the Sym-Bobenko formula,
$$\widetilde{\iota}^* f_{\other{\varphi}}^g=US^{-1}(f_{\varphi}^g)^{-1}S U^{-1}$$
so $\widetilde{\iota}^* f_{\other{\varphi}}^g$ and $f_{\varphi}^g$ are congruent
by an orientation-reversing isometry of $\S^3$.
Since $\other{\theta}=\pi-\theta$, we can see directly that they have opposite mean curvature.
\item For $\varphi = \tfrac{\pi}{4}$, the immersion $f^g_{\varphi}$ is minimal. Given its symmetries and the fact that its area is less than $8\pi$ (which we prove in Section \ref{areaestimates}), the same argument as in \cite[Theorem 5]{HHT} proves that it is Lawson surface $\xi_{1,g}$ for $g \gg 1$.
\end{itemize}
\end{proof}

\begin{lemma}\label{Lem:uniuni}
Consider  the DPW potential $\eta_t^{\xx(t)}$ provided by 
Proposition \ref{Fuchsianpotential}.
For $t\sim0$  small
the unitarizer $\Phi_0= \Phi(z=0)$ is unique up to multiplication with $\Lambda \mathrm{SU}(2)$ from the left. 
\end{lemma}

\begin{proof}
Assume that there is a second unitarizer, this gives rise to  two CMC surfaces
 on $M^g_\varphi$, with two associated families of flat connections
$\lambda\mapsto$ 
$\nabla^\lambda_1$ and $\lambda\mapsto$$\nabla^\lambda_2$. By construction, these two families are both gauge equivalent to $\pi^*\eta_{t, \varphi}^{\xx(t,\varphi)}$. 
Therefore, by \cite[Theorem 7]{He2} the two families either differ by a (non-trivial) dressing deformation or they are equal up to a $\lambda$-independent $\mathrm{SU}(2)$ gauge transformation. In the latter case the second unitarizer differs from the first only by a $\Lambda SU(2)$ element.\\

Moreover, a non-trivial dressing transformation exists if and only if  the family $\lambda\mapsto \nabla^\lambda_1$
(or equivalently the second family) contains a reducible connection $\nabla^{\lambda_0}_1$ for some $\lambda_0\in D(0,1).$ 
Using the information about the monodromies near $t=0$ (e.g. \eqref{mondiff} or  \cite[Proposition 8]{nnoids} and the central value \eqref{central1}) 
this can be excluded, as the monodromy
on $M^g_\varphi$ is irreducible for every $\lambda$ except at the Sym points  $\lambda_1,\lambda_2\in\mathbb S^1.$
\end{proof}

\section{Degenerating conformal type} \label{limitvarphi}
 
The aim in this section is to prove a uniform time interval for the existence of $\eta^{\xx(t, \varphi)}_{t,\varphi}$ given by  Proposition \ref{Fuchsianpotential} for all 
$\varphi\in(0,\tfrac{\pi}{2}).$
By Theorem \ref{thm:construction-fixed-varphi}, it suffices to study the behaviour for $\varphi\rightarrow0,$ where the singular points $p_1$ and $p_4$  as well as the singular points $p_2$ and $p_3$ coalesce at $\pm 1$, respectively. Since the differentials in the proof of Proposition \ref{Fuchsianpotential} do not depend on $\varphi$, we only need to show that  the maps $\mathcal F,$ $\mathcal G$, $\mathcal K$ and $\mathcal H_1,$ $\mathcal H_2$ remain smooth enough to apply a version of the implicit function theorem at $\varphi = 0$. For $\varphi\rightarrow 0$ the limit Riemann surface is given by $$\Sigma_0 = \C P^1\setminus\{\pm  1\}.$$
Hence the quantities corresponding to $\mathcal Q$, i.e., $\mathcal G$ and $\mathcal H_2,$ remain well-defined and depend smoothly on $\varphi$ in the $\varphi \rightarrow 0$ limit.
To be explicit, we have
\begin{equation}
\label{eq:eta-phi0}
\eta^{\xx}_{t, 0} = r t(A_1+A_4)\frac{dz}{z-1}+ rt(A_2+A_3)\frac{dz}{z+1}
=2rt\matrix{0&b\\b&0}\left(\frac{dz}{z-1}-\frac{dz}{z+1}\right)
\end{equation}
and
\begin{equation*}
\begin{split}
\Phi_{t,0}^{\xx}&=\exp\left(2rt\begin{pmatrix}0&b\\b&0\end{pmatrix}\log\left(\frac{1-z}{1+z}\right)\right)\\
\mathcal Q\mid_{\varphi=0}=\Phi_{t,0}^{\xx}(z=\ii)&=
\matrix{\cos(\pi rt b)&-\ii\sin(\pi rtb)\\-\ii\sin(\pi rtb)&\cos(\pi rtb)}\\
\mathfrak q\mid_{\varphi=0}&=\sin(2\pi r t b).
\end{split}
\end{equation*}
 
Applying the implicit function theorem to the equations $\mathcal G=0$ and $\mathcal H_2=0$ uniquely determines the parameter $\wt{b}\in \mathcal W_\R^{\geq0}$
as a smooth function of $(t,\varphi)$ in a neighborhood of $(0,0)$ and
the remaining parameters $(\wt{a},\wt{c},\theta)$.
Moreover, at $\varphi=0$, the solution is $b=0$ for all $t$.
Dealing with the limits of the remaining equations is more difficult, as $\mathcal P$ becomes singular.  
\subsection{The asymptotic of $\mathcal P$} $\;$\\
 Let $\wt \xx = (\wt a,\wt c, \theta)$ denote the vector with remaining parameters and let 
$\eta_{t,\varphi}^{\wt \xx}$ denote the corresponding Fuchsian DPW potential satisfying 
$$\mathcal G = 0 \quad \text{ and } \quad \mathcal H_2=0$$
for $t\sim 0$ and $\varphi \in[0, \tfrac{\pi}{2})$.

It turns out that $\mathcal P$ does not extend smoothly at $\varphi=0$, but
rather extends as a smooth function of $\varphi$ and $\varphi\log\varphi$, in the following sense:
\begin{definition}\cite{nodes}
 Let $f(x)$ be a function of the real variable $x\geq 0$. We say that f is a smooth function of
$x$ and $x\log x$ if there exists a smooth function of two variables $g(x,y)$ defined in a neighborhood of $(0,0)$
in $\R^2$ such that 
$$f(x) = g(x,x\log x) \quad  \text{ for } x > 0 \quad \text{ and } \quad f(0) = g(0,0).$$
 \end{definition}
 Note that a smooth function of $x$ and $x\log x$ is in general only continuous but not differentiable at $x=0$.

Let $\Gamma$ denote the straight line from $z=0$ to $z=1.$ For $\varphi =0$ the end point of the curve is singular and thus the principal solution    $\mathcal P$ along $\Gamma$ is no longer well-defined.
Fix some positive numbers $0<\varepsilon<\varepsilon_0<1$.
To study the behaviour of the principal solution for $\varphi \sim 0^+$ around $z=1$ we subdivide the curve $\Gamma$ into  
\[\Gamma=\Gamma_0\circ\Gamma_{1,\varphi}\circ\Gamma_{2,\varphi}\]
with
\begin{equation}
\begin{split}
\Gamma_0&\colon  s\mapsto (1-\varepsilon)s,\\
\Gamma_{1,\varphi}&\colon  s\mapsto (1-\varepsilon)(1-s)+(1-\tfrac{\varphi}{\varepsilon})s\\
\Gamma_{2,\varphi}&\colon s\mapsto 1+\tfrac{\varphi}{\varepsilon}(s-1).\,\\
\end{split}
\end{equation}
For a flat connection $d+\xi$ we denote $\chi_\gamma(\xi)$ the principal solution along $\gamma$, so 
$$\mathcal P(t,\varphi,\wt\xx)=\chi_{\Gamma_0}(\eta_{t,\varphi}^{\wt \xx})\,
\chi_{\Gamma_{1,\varphi}}(\eta_{t,\varphi}^{\wt \xx})\,
\chi_{\Gamma_{2,\varphi}}(\eta_{t,\varphi}^{\wt \xx})
.$$

The principal solution along $\Gamma_0$ is clearly a smooth function of $\varphi$ in a neighborhood of $0$, since the path $\Gamma_0$ is fixed and
$\eta_{t,\varphi}^{\wt\xx}$ depends smoothly on $\varphi$ on $\Gamma_0$.
It is more delicate for the paths $\Gamma_{1,\varphi}$ and $\Gamma_{2,\varphi}$:
we will see that the principal solution along $\Gamma_{1,\varphi}$ extends as a smooth function of $\varphi$ and $\varphi\log\varphi$ at $\varphi=0$, while the principal solution along $\Gamma_{2,\varphi}$
is a smooth function of $\varphi$.

\subsection*{Principal solution along $\Gamma_{2,\varphi}$}$\;$\\
To analyse the $\varphi\to 0$
limit of the principal solution along $\Gamma_{2,\varphi}$ we consider for $\varphi>0$ the diffeomorphism
\[\psi_\varphi\colon D(1;1)\to D(1;\tfrac{2\varphi}{\varepsilon}); \quad z\mapsto 1+\tfrac{2\varphi}{\varepsilon}(z-1),\]
where $D(1, R)$ is the disc of radius $R$ around $z=1.$ Then
\[\wt{\Gamma}_2:=(\psi_\varphi)^{-1}\circ\Gamma_{2,\varphi}\colon [0,1] \longrightarrow D(1,1); \quad \,s\longmapsto  \tfrac{1}{2}(1+s) \]
is independent of $\varphi$.
The pullback potential $\wt{\eta}_{t,\varphi}^{\wt{\xx}}=\psi_{\varphi}^*\eta_{t,\varphi}^{\wt{\xx}}$
has simple poles at $\psi_{\varphi}^{-1}(p_k)$ for $1\leq k\leq 4$.
We have
\begin{equation}
\begin{split}
&\lim_{\varphi\to 0}\psi_{\varphi}^{-1}(p_1(\varphi))=1+\tfrac{\varepsilon}{2}\ii,\quad
\lim_{\varphi\to 0}\psi_{\varphi}^{-1}(p_4(\varphi))=1-\tfrac{\varepsilon}{2}\ii,\\
&\lim_{\varphi\to 0}\psi_{\varphi}^{-1}(p_2(\varphi))=
\lim_{\varphi\to 0}\psi_{\varphi}^{-1}(p_3(\varphi))=\infty.
\end{split}
\end{equation}
So $\wt{\eta}_{t,\varphi}^{\wt{\xx}}$ extends smoothly at $\varphi=0$ with
$$\wt{\eta}_{t,0}^{\wt{\xx}}=r tA_1\frac{dz}{z-1-\tfrac{\varepsilon}{2}\ii}
+r tA_4\frac{dz}{z-1+\tfrac{\varepsilon}{2}\ii}.$$
Hence
$$\chi_{\Gamma_{2,\varphi}}(\eta_{t,\varphi}^{\wt{\xx}})=\chi_{\wt{\Gamma}_2}(\wt{\eta}_{t,\varphi}^{\wt{\xx}})$$
extends smoothly at $\varphi=0$.

\subsection*{Principal solution along $\Gamma_{1,\varphi}$}\label{subseubgamma0} $\;$\\
We apply \cite[Theorem 5]{nodes} which we restate here as Theorem \ref{theorem:philogphi} with adjusted notations.
To use this result, it is necessary to view $\varphi$ as a complex number.

\begin{remark}
Note that for $\varphi\in\C$ in a neighborhood of $0$, $\eta_{t,\varphi}^{\xx}$
is well defined and depends holomorphically on $\varphi$, although the poles
$\pm e^{\pm\ii\varphi}$ are not on the unit circle anymore breaking the symmetries.
Nevertheless,  the equations $\mathcal G=0$ and $\mathcal H_2=0$ can still be solved using the implicit function theorem: the solution $\wt{b}(t,\varphi,\wt{\xx})$ is in $\mathcal W^{\geq 0}$ instead of $\mathcal W^{\geq 0}_{\R}$ and depends holomorphically on $\varphi$.
\end{remark}

This time we consider for $\varphi \in D(0,\varepsilon_0^2)\subset\C$ the diffeomorphism 
$$\psi_\varphi \colon \mathcal A_\varphi \longrightarrow \mathcal A_\varphi, \quad z \longmapsto 1+ \frac{\varphi}{z-1}$$
where
$\mathcal A_{\varphi}$ is the annulus
\[\mathcal A_\varphi=\{z\in\C\mid  \tfrac{|\varphi|}{\varepsilon_0}<|z-1|<\varepsilon_0\}.\] 
Consider the change of parameter $\varphi= e^\omega$ with $\Re(\omega)<2\log\varepsilon_0$. 
Let $\beta_\omega$ denote the spiral curve from $z= 1-\varepsilon$ to $z = 1-\tfrac{\varphi}{\varepsilon}$ defined by
\[\beta_\omega \colon [0,1] \rightarrow \mathcal A_\varphi, \quad \beta_\omega(s) =1 -\varepsilon^{1-2s}e^{s \omega}.\]

Note that
\[\psi_{\varphi}\circ\beta_\omega(s)=\beta_{\omega}(1-s),\]
and for real $\omega$, the path $\beta_{\omega}$ is homotopic to
$\Gamma_{1,\varphi}$. Let 
\[\gamma  \colon [0,1] \rightarrow \mathcal A_\varphi; \; \gamma(s)=1- \varepsilon e^{2 \pi i s}
\]
 be the circle of radius $\varepsilon$ around $z=1.$ Then we can define
\begin{equation}\label{tildeF}
\wt F(\omega)=\chi_\gamma(\eta_\varphi)^{ - \frac{\omega}{2\pi i} }\chi_{\beta_\omega }(\eta_\varphi).
\end{equation}
\begin{theorem}[Theorem 5 in \cite{nodes}]

\label{theorem:philogphi}
With the notations introduced above
consider a family of DPW potentials $\eta_\varphi$ in $\mathcal A_{\varphi}$
depending holomorphically on $\varphi\in D(0,\varepsilon_0^2)$.
Let $\wh \eta_{\varphi} := \psi_\varphi^* \eta_{\varphi}$. Assume that there exists $\sl(2, \C)$-valued 1-forms $\eta_0$ and $\wh \eta_0$ holomorphic in the disk $D(1,\varepsilon_0)$ such that 
$$\lim_{\varphi \rightarrow 0} \eta_{\varphi} = \eta_{0} \quad \text{and} \quad  \lim_{\varphi \rightarrow 0} \wh \eta_{\varphi} = \wh \eta_0$$
on compact subsets of the punctured disk $D^*(1,\varepsilon_0).$
Then we have for $|\varphi|$ small enough:
\begin{enumerate}
\item The function $\wt F (\omega)$
satisfies $\wt F(\omega + 2\pi\ii) = \wt F(\omega)$ and therefore descends to a holomorphic function $F$ on $D(0,\varepsilon_0^2)$ by  $F(e^\omega):=\wt F(\omega)$.
\item The function $F$ extends holomorphically  to  $\varphi= 0$ with 
$$F(0) = \chi(\eta_0 ,1-\varepsilon,1)\chi(\wh \eta_0 ,1,1-\varepsilon),$$
where $\chi(\eta_0, 1-\varepsilon, 1 )$ denotes the principal solution of the extended frame of $\eta_0$ along 
the straight line from $1-\varepsilon$ to $1$.
\item If $\varphi> 0$, the function $\chi_{\beta_\omega }(\eta_\varphi)$ extends to a smooth function of $\varphi$ and $\varphi\log\varphi$ with value $F(0)$ at $\varphi = 0$. 
\end{enumerate}
\end{theorem}

Returning to our problem, we have $|1-p_1|\simeq |\varphi|$ for small $\varphi$ so
$\mathcal A_{\varphi}$ does not contain $p_1$, nor $p_2$, $p_3$ or $p_4$.
Hence $\eta_{t,\varphi}^{\wt{\xx}}$ is holomorphic in $\mathcal A_{\varphi}$.
Moreover, by Equation \eqref{eq:eta-phi0} and recalling that $b\mid_{\varphi=0}=0$,
we have
$$\lim_{\varphi\to 0}\eta_{t,\varphi}^{\wt{\xx}}=0.$$
Since $\psi_{\varphi}$ is involutive, the pullback potential $\wh{\eta}_{t,\varphi}^{\wt{\xx}}=\psi_{\varphi}^*\eta_{t,\varphi}^{\wt{\xx}}$ has poles at
$\psi_{\varphi}(p_k)$ for $1\leq k\leq 4$. We have
$$\lim_{\varphi\to 0} \psi_{\varphi}(p_1(\varphi))=1-\ii,\quad
\lim_{\varphi\to 0} \psi_{\varphi}(p_4(\varphi))=1+\ii\quad\text{ and }\quad
\lim_{\varphi \to 0} \psi_{\varphi}(p_2(\varphi))=\lim_{\varphi\to 0} \psi_{\varphi}(p_3(\varphi))=1.$$
Hence, and because $b\mid_{\varphi=0}=0$ for all $t,$

\begin{eqnarray*}
\lim_{\varphi\to 0}\wh{\eta}_{t,\varphi}^{\wt{\xx}}
&=&rt A_1\frac{dz}{z-1+\ii}+rtA_4\frac{dz}{z-1-\ii}+rt(A_2+A_3)\frac{dz}{z-1}\\
&=&rt \matrix{\ii a& \ii c\\ -\ii c&-\ii a}\left(\frac{dz}{z-1+\ii}+\frac{dz}{z-1+\ii}\right).
\end{eqnarray*}
Consequently, the limit $\wh{\eta}_{t,0}^{\wt{\xx}}$
is holomorphic in $D(1,\varepsilon_0)$.
By Point (3) of Theorem \ref{theorem:philogphi}, $\chi_{\Gamma_{1,\varphi}}(\eta_{t,\varphi}^{\wt{\xx}})$ extends to a smooth function of $\varphi$ and $\varphi\log\varphi$
at $\varphi=0$.

\subsection*{Conclusion}
We have proved that $\mathcal P$, hence $\wh{\mathcal P}$, extends as a smooth function of $\varphi$
and $\varphi\log\varphi$ at $\varphi=0$. In other words,
$$\wh{\mathcal P}(t,\varphi,\wt{\xx})=g(t,\varphi,\varphi\log\varphi,\wt{\xx})$$
where $g$ is a smooth function of the variables $(t,\varphi,\phi,\wt{\xx})$.
Note that at $t=0$, $\wh{\mathcal P}$ does not depend on $\varphi$, so the value
of $g$ at $t=0$ is known.
The remaining parameters $\wt{\xx}=(\wt{a},\wt{c},\theta)$ are determined by solving the equations $\mathcal F=0,$ $\mathcal K= \mathcal K^0$ and $\mathcal H_1=0$.
Specifically, we use the implicit function theorem to determine $\wt{\xx}$ as a smooth function of
$(t,\varphi,\phi)$ in a neighhorhood of $(0,0,0)$ and then specialize to $\phi=\varphi\log\varphi$.
This proves that the solution $\wt{\xx}$ is a smooth function of $t$, $\varphi$
and $\varphi\log\varphi$. This yields a solution $\eta_{t,\varphi}^{\xx(t,\varphi)}$ of the Monodromy Problem \eqref{monodromy-problem3} for every $\varphi $ in the compact interval $[0, \tfrac{\pi}{2}]$ and $t$ in a uniform interval $(-\varepsilon,\varepsilon)$.
 For rational $t$ the DPW potential $\eta_{t,\varphi}^{\xx(t)}$ gives rise to a compact CMC surface in $\mathbb S^3$ of genus $g$ on the $(g+1)$-fold cover of $\Sigma_\varphi$ totally branched over $p_1, ..., p_4.$ Together with Proposition \ref{Fuchsianpotential}, Lemma \ref{lem:monodromyproblemequiv} and $b\mid_{\varphi=0}=0$ this shows the following theorem.

\begin{theorem}\label{surfaces}
There exist an $\varepsilon >0$ such that for every $t\in (-\varepsilon,\varepsilon)$ and  $\varphi \in [0, \tfrac{\pi}{2}]$  there exists a unique DPW potential $\eta_{t, \varphi}^{\xx(t, \varphi)}$ solving the Monodromy Problem \eqref{monodromy-problem3}, where $\xx(t, \varphi)$ is smooth in $t$ and smooth in $\varphi$ and $\varphi \log(\varphi)$.
In particular, there exist a $g_0 \in \N$ such that for all $g>g_0$
there exist a complete family of CMC surfaces $f^g_\varphi: M^g_\varphi \rightarrow \S^3$ parametrized by $\varphi \in [0, \tfrac{\pi}{2}]$ which converges for $\varphi \rightarrow 0, \tfrac{\pi}{2}$ uniformly on every compact set of $M_0^g\setminus \pi^{-1}\{\pm1\}$ to a geodesic $2$-sphere with $2g+2$ branch points at the preimages of $\pm1$.\end{theorem}

\section{Area and Willmore energy estimates}\label{areaestimates}
 
In this section, we compute the first order expansion of the mean curvature, area and
Willmore energy of the immersion $f^g_{\varphi}$, as a function of $t= \tfrac{1}{2(g+1)}$.
To do so we need first to compute the time derivatives of the parameters $\xx(t, \varphi)$ at $t=0$ (denoted by $()'$) for which the family of DPW potentials $\eta_t = \eta_{t, \varphi}^{\xx(t, \varphi)}$ solves the Monodromy Problem \eqref{monodromy-problem3}. 

\subsection{First order derivatives of the parameters}$\;$\\
\begin{proposition}\label{Prop:1stderivatives}
\label{prop:derivatives} We have at $t=0$
\[a'=a'_0+a'_2\lambda^2,\quad
b'=b'_0+b'_2\lambda^2\quad\text{and}\quad
c'=c'_0+c'_2\lambda^2\]
with
\begin{equation}\label{proeq:secondordercoeff}
\begin{split}
a'_0&=a'_2=\sin(2\varphi)\log(\tan(\varphi))\\
b'_2&=-2\cos(\varphi)\log(\cos(\varphi))\\
c'_2&=2\sin(\varphi)\log(\sin(\varphi))\\
b'_0&=b'_2\cos(2\varphi)-c'_2\sin(2\varphi)\\
c'_0&=-b'_2\sin(2\varphi)-c'_2\cos(2\varphi).
\end{split}
\end{equation}
Moreover,
\[r'=0\quad\text{ and }\quad \theta'=2 \sin(2\varphi)\log(\tan(\varphi)).\]

\end{proposition}

\begin{proof}
We differentiate $d\Phi_t=\Phi_t\eta_t$ twice at $t=0$ and obtain
\begin{equation}\label{secondderivative}
d \Phi'' = \Phi'' \eta_0 + 2\Phi' \eta' + \eta'' = 2 \Phi' \eta' + \eta''.
\end{equation}
with
\begin{equation}\label{eta''}\eta'' = 
2\begin{pmatrix} \ii a' \omega_a & b' \omega_b + \ii c' \omega_c \\b' \omega_b - \ii c' \omega_c& -\ii a' \omega_a\end{pmatrix}
+2r'\begin{pmatrix}\ii \cv{a}\omega_a&\cv{b}\omega_b+\ii \cv{c}\omega_c\\
\cv{b}\omega_b-\ii\cv{c}\omega_c&-\ii\cv{a}\omega_a\end{pmatrix}.
\end{equation}

From the computations in the proof of Proposition \ref{Pro:pq} we have
\begin{equation}\label{phi'eta'}
\begin{split}
\Phi'\eta'&=\begin{pmatrix} \ii \cv{a}\Omega_a  &
\cv{b} \Omega_b +
\ii \cv{c} \Omega_c \\
\cv{b} \Omega_b -
\ii \cv{c} \Omega_c &-\ii \cv{a} \Omega_a
\end{pmatrix}
\begin{pmatrix} \ii \cv{a} \omega_a  &
\cv{b} \omega_b +
\ii \cv{c} \omega_c \\
\cv{b} \omega_b -
\ii \cv{c} \omega_c &-\ii \cv{a}  \omega_a
\end{pmatrix}\\
&=\begin{pmatrix}*& 
\ii \cv{a}\Omega_a(\cv{b}\omega_b+\ii \cv{c}\omega_c)-\ii \cv{a}\omega_a(\cv{b}\Omega_b+\ii \cv{c}\Omega_c)\\
\ii \cv{a}\omega_a(\cv{b}\Omega_b-\ii \cv{c}\Omega_c)- \ii \cv{a}\Omega_a(\cv{b}\omega_b-\ii \cv{c}\omega_c)
&*\end{pmatrix}.
\end{split}
\end{equation}
Differentiate $\mathfrak p$ and $\mathfrak q$, as defined in \eqref{pq}, twice at $t=0,$
keeping in mind that $\Phi_0=\Id,$ yields 
\begin{equation}
\begin{split}
\mathfrak p''&=\mathcal P''_{21}-\mathcal P''_{12}+2(\mathcal P'_{11}\mathcal P'_{21}-\mathcal P'_{12}\mathcal P'_{22})\\
\mathfrak q''&=\ii(\mathcal Q''_{21}+\mathcal Q''_{12})+2\ii(\mathcal Q'_{11}\mathcal Q'_{21}+\mathcal Q'_{12}\mathcal Q'_{22}).
\end{split}
\end{equation}
We compute each term separately.
Using Equations \eqref{central1}, \eqref{eq:Omega1} and \eqref{eq:OmegaI} we obtain
\begin{equation}
\begin{split}
\mathcal P'_{11}\mathcal P'_{21}-\mathcal P'_{12}\mathcal P'_{22}&=2\ii \cv{a}\Omega_a(1)\cv{b}\Omega_b(1)=\tfrac{1}{2}(\lambda^{-2}-\lambda^2)(\pi-2\varphi)\sin(\varphi)\log\left(\frac{1-\cos(\varphi)}{1+\cos(\varphi)}\right)\\
\mathcal Q'_{11}\mathcal Q'_{21}+\mathcal Q'_{12}\mathcal Q'_{22}
&=2\cv{a}\Omega_a(\ii)\cv{c}\Omega_c(\ii)
=\ii(\lambda^{-2}-\lambda^2)\varphi\cos(\varphi)\log\left(\frac{1-\sin(\varphi)}{1+\sin(\varphi)}\right).
\end{split}
\end{equation}
Using Equations \eqref{secondderivative}, \eqref{eta''} and \eqref{phi'eta'} we obtain
\begin{eqnarray*}
\mathcal P''_{21}-\mathcal P''_{12}&=&
2\int_0^1(\Phi'\eta')_{21}-(\Phi'\eta')_{12}+\int_0^1 \eta''_{21}-\eta''_{12}\\
&=&4\ii\cv{a}\,\cv{b}\int_0^1(\omega_a\Omega_b-\omega_b\Omega_a)
-4\ii (c'+\cv{c}r')\int_0^1\omega_c\\
&=&4\pi c'-2\pi(\lambda^{-1}+\lambda)\cos(\varphi)r'
-\ii(\lambda^{-2}-\lambda^2)\sin(\varphi)\int_0^1\omega_a\Omega_b-\omega_b\Omega_a
\end{eqnarray*}
and
\begin{eqnarray*}
\mathcal Q''_{21}+\mathcal Q''_{12}&=&
2\int_0^{\ii}(\Phi'\eta')_{21}+(\Phi'\eta')_{12}+\int_0^{\ii} \eta''_{21}+\eta''_{12}\\
&=&4\cv{a}\,\cv{c}\int_0^{\ii}(\omega_a\Omega_c-\omega_c\Omega_a)
+4(b'+\cv{b}r')\int_0^{\ii}\omega_b\\
&=&-4\pi\ii b' +2\pi\ii (\lambda^{-1}+\lambda)\sin(\varphi)r'
-(\lambda^{-2}-\lambda^2)\cos(\varphi)\int_0^{\ii}\omega_a\Omega_c-\omega_c\Omega_a.
\end{eqnarray*}
The required integrals are computed in Proposition \ref{prop:integrals} at the end of this section. After simplification, this gives
\begin{equation}
\begin{split}
\mathfrak p''&=4\pi c'-2\pi(\lambda^{-1}+\lambda)\cos(\varphi)r'
+4\pi(\lambda^{-2}-\lambda^2)\sin(\varphi)\log(\sin(\varphi))\\
\mathfrak q''&=4\pi b'-2\pi(\lambda^{-1}+\lambda)\sin(\varphi)r'
-4\pi(\lambda^{-2}-\lambda^2)\cos(\varphi)\log(\cos(\varphi)).
\end{split}
\end{equation}

Since $\eta_t$ solves the Monodromy Problem, we have $\mathfrak p_t=\mathfrak p_t^*$ and
$\mathfrak q_t=\mathfrak q_t^*$ for all $t$ by construction. Therefore, 
$\mathfrak p''=(\mathfrak p'')^*$ and $\mathfrak q''=(\mathfrak q'')^*$ and 
projecting onto $\mathcal W^{>0}$ gives
\begin{equation}
\begin{split}
(c')^+&=2\lambda^2\sin(\varphi)\log(\sin(\varphi))\\
(b')^+&=-2\lambda^2\cos(\varphi)\log(\cos(\varphi)).
\end{split}
\end{equation}
In other words, $b'=b'_0+b'_2\lambda^2$ and $c'=c'_0+c'_2\lambda^2$ where
$b'_2$, $c'_2$ are as in Proposition \ref{prop:derivatives} and
$b'_0$, $c'_0$ are yet to be determined. We now have
\begin{equation}
\label{eq:p''}
\mathfrak p''=4\pi c'_0+2\pi c'_2(\lambda^{-2}+\lambda^2)-2\pi(\lambda^{-1}+\lambda)\cos(\varphi)r'
\end{equation}
\begin{equation}
\label{eq:q''}
\mathfrak q''=4\pi b'_0+2\pi b'_2(\lambda^{-2}+\lambda^2)-2\pi(\lambda^{-1}+\lambda)\sin(\varphi)r'
\end{equation}
which are real on the unit circle as required.\\

Moreover, the Sym-point conditions yield $\mathfrak p_t(e^{\ii\theta_t})=0$ and $\mathfrak q_t(e^{\ii\theta_t})=0$ for all $t$.
We differentiate these equations  twice with respect to $t$ at $t=0$ and obtain, as
$\mathfrak p_0=\mathfrak q_0=0,$
\[\mathfrak p''(\ii)-2\frac{\partial\mathfrak p'}{\partial\lambda}(\ii)\theta'=0\quad\text{and}\quad
\mathfrak q''(\ii)-2\frac{\partial\mathfrak q'}{\partial\lambda}(\ii)=0.\]
By Equations \eqref{eq:p''} and \eqref{eq:q''}
\[\mathfrak p''(\ii)=4\pi(c'_0-c'_2)\quad\text{and}\quad
\mathfrak q''(\ii)=4\pi(b'_0-b'_2).\]
Moreover,  Proposition \ref{Pro:pq} shows
\[\frac{\partial\mathfrak p'}{\partial\lambda}(\ii)=\frac{\partial}{\partial\lambda}(2\pi\cv{c})\mid_{\lambda=\ii}=-2\pi\cos(\varphi)\]
\[\frac{\partial\mathfrak q'}{\partial\lambda}(\ii)=\frac{\partial}{\partial\lambda}(2\pi\cv{b})\mid_{\lambda=\ii}=-2\pi\sin(\varphi).\]
So we obtain the system of equations
\begin{equation}
\label{eq:b0c0}
\left\{\begin{split}
c'_0-c'_2+\cos(\varphi)\theta'=0\\
b'_0-b'_2+\sin(\varphi)\theta'=0\end{split}\right ..\end{equation}
Finally, we have
\[\mathcal K'=-2r'+2\cv{a}a'-2\cv{b}b'-2\cv{c}c'.\]
Observe that
\begin{equation*}
\begin{split}&b'-b'^*=-2\cos(\varphi)\log(\cos(\varphi))(\lambda^{2}-\lambda^{-2})=-8\,\cv{a}\cv{c}\log(\cos(\varphi))\\
&c'-c'^*=2\sin(\varphi)\log(\sin(\varphi))(\lambda^2-\lambda^{-2})=8\,\cv{a}\cv{b}\log(\sin\varphi).\end{split}\end{equation*}
Hence using $\cv{a}^*=-\cv{a}$, $\cv{b}=\cv{b}^*$ and $\cv{c}=\cv{c}^*$
$$0=\mathcal K'-\mathcal K'^*=2\,\cv{a}(a'+a'^*)-2\cv{b}(b'-b'^*)-2\cv{c}(c'-c'^*)=2\,\cv{a}(a'+a'^*)-16\,\cv{a}\cv{b}\cv{c}\log(\tan(\varphi)).$$
Simplifying by $\cv{a}$, we obtain
$$a'+a'^*=8\,\cv{b}\cv{c}\log(\tan(\varphi))=\sin(2\varphi)\log(\tan(\varphi))(\lambda^{-1}+\lambda)^2$$
which determines $a'$ as in Proposition \ref{prop:derivatives}.
Then we have
$$\mathcal K'_0=-2 r'=0$$
and
\begin{equation}
\label{eq:K'-1}
\mathcal K'_{-1}=a'_0+\sin(\varphi)b'_0+\cos(\varphi)c'_0=0.
\end{equation}
Solving the system \eqref{eq:b0c0} and \eqref{eq:K'-1} determines $b'_0$, $c'_0$
and $\theta'$ as in Proposition \ref{prop:derivatives}.
\end{proof}

\subsection{Mean curvature}
\begin{figure}[h]
\vspace{-2.cm}
\centering
\includegraphics[width=0.75\textwidth]{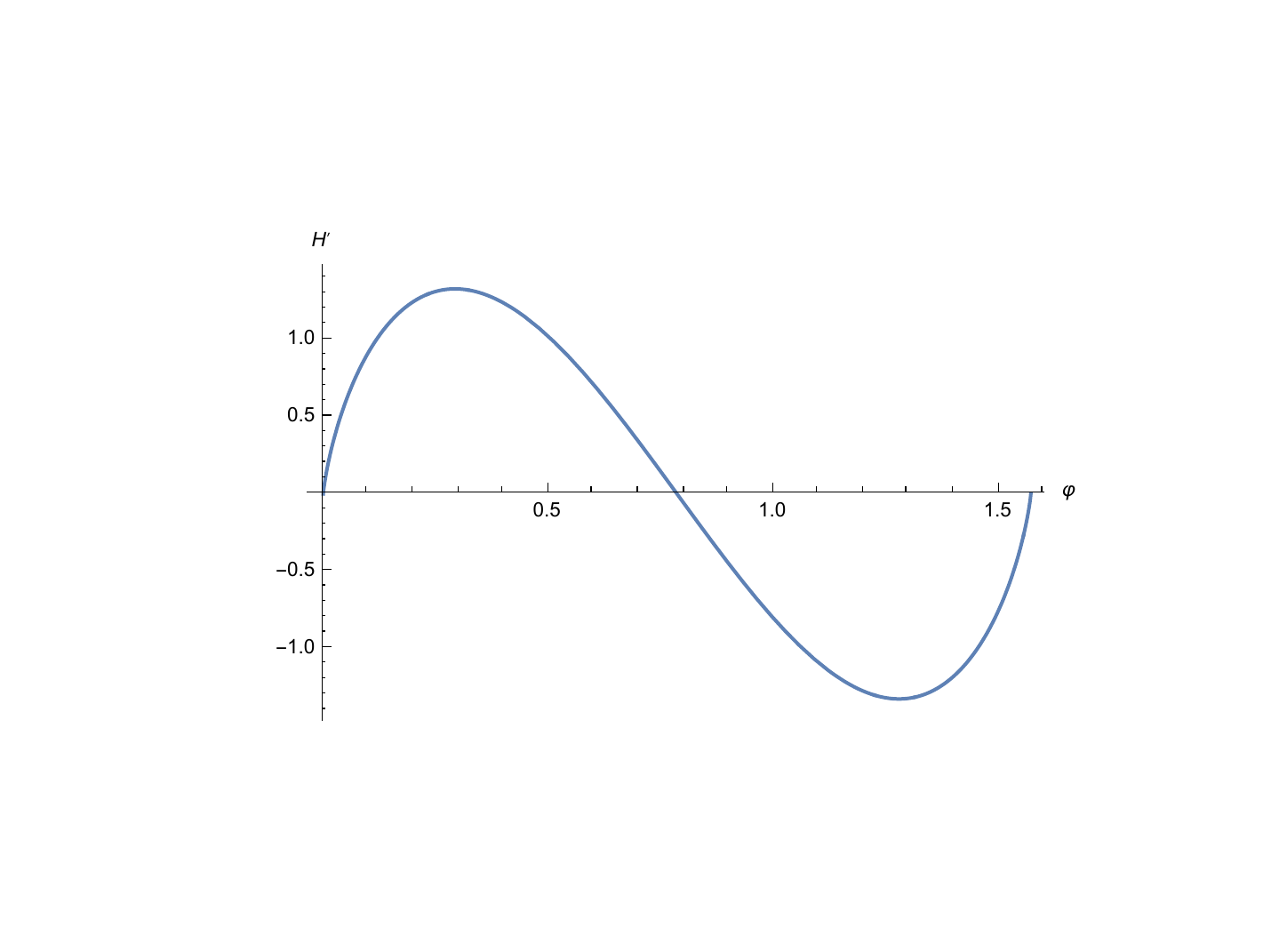}
\vspace{-1.75cm}
\caption{
\footnotesize The derivative $H'$ of the mean curvature $H(t,\varphi)$ at $t=0$ with respect to $t$ as a function of $\varphi $ for $\varphi \in [0, \tfrac{\pi}{2}].$ The zeros of $H'$ are at $\varphi = 0, \tfrac{\pi}{4}, \frac{\pi}{2}$.}
\label{Hprime}
\end{figure}

Let $H^g_{\varphi}$ be the mean curvature of $f^g_{\varphi}$. From Theorem \ref{thm:construction-fixed-varphi} (3) we have $H^g_{\varphi}=-H^g_{\tfrac{\pi}{2}-\varphi}$
and therefore, $f^g_{\varphi}$ is minimal for $\varphi\in\{0,\frac{\pi}{4},\frac{\pi}{2}\}$.
In fact we will show that these are the only minimal surfaces within the family for $g$ large enough.
\begin{proposition}
\label{prop:H}
For $g \gg1 $ fixed the mean curvature $H^g_{\varphi}$ is strictly positive for all 
$\varphi\in (0, \frac{\pi}{4})$.
\end{proposition}
\begin{proof}
We have $H^g_{\varphi}=H(t(g),\varphi)$ with $t(g)=\frac{1}{2g+2}$ and $$H(t,\varphi)=\operatorname{cotan}(\theta(t,\varphi)).$$ 
We want to find a uniform $\varepsilon>0$ such that $H(t,\varphi)>0$ for
all $0<t<\varepsilon$ and $\varphi\in(0,\frac{\pi}{4})$.
By Proposition \ref{prop:derivatives}, we have
$$\frac{\partial H}{\partial t}(0,\varphi)=-\theta'=-2\sin(2\varphi)\log(\tan(\varphi))$$
which is positive for $0<\varphi<\frac{\pi}{4}$.
So the existence of $\varepsilon$ is ensured for $\varphi$ in any proper subinterval of $(0,\frac{\pi}{4})$.
To study the sign of $H$ near $\varphi=0$ and $\varphi=\frac{\pi}{4}$, 
define
$$\widehat{H}(t,\varphi)=\frac{1}{t}H(t,\varphi)$$
which extends smoothly at $t=0$ with
$$\widehat{H}(0,\varphi)=-2\sin(2\varphi)\log(\tan(\varphi)).$$
We have $\frac{\partial \widehat{H}}{\partial\varphi}(0,\frac{\pi}{4})=-4$.
Since $\widehat{H}(t,\frac{\pi}{4})=0$ for all $t$,
a first order Taylor expansion of $\widehat{H}$ at $(t,\frac{\pi}{4})$ gives
$$\widehat{H}(t,\varphi)=(-4+O(t))(\varphi-\tfrac{\pi}{4})+O((\varphi-\tfrac{\pi}{4})^2)$$
so $\widehat{H}(t,\varphi)>0$ for $t$ small enough and $\varphi<\frac{\pi}{4}$ close
enough to $\frac{\pi}{4}$.

In a neighborhood of $\varphi=0$, $\widehat{H}$ extends as a smooth function of $\varphi$ and
$\varphi\log(\varphi)$ so there exists a smooth function $\mathsf H(t,\varphi,\psi)$ such that
$$\widehat{H}(t,\varphi)=\mathsf H(t,\varphi,\varphi\log(\varphi)).$$
At $t=0$, we have explicitly
$$\mathsf H(0,\varphi,\psi)=-2\sin(2\varphi)\log\left(\frac{h(\varphi)}{\cos(\varphi)}\right)-4h(2\varphi)\psi$$
with $h(x)=\frac{\sin(x)}{x}$, which extends analytically at $x=0$.
We have
$\frac{\partial\mathsf H}{\partial\varphi}(0,0,0)=0$ and
$\frac{\partial\mathsf H}{\partial\psi}(0,0,0)=-4$.
Since $\mathsf H(t,0,0)=0$ for all $t$, a first order Taylor expansion at $(t,0,0)$ gives
$$\mathsf H(t,\varphi,\psi)=O(t)\varphi+(-4+O(t))\psi+O(\varphi^2+\psi^2).$$
We substitute $\psi=\varphi\log(\varphi)$ and obtain
$$\widehat H(t,\varphi)=(-4+O(t))\varphi\log(\varphi)+O((\varphi\log(\varphi))^2).$$
Hence $\widehat H(t,\varphi)>0$ for $t$ and $\varphi>0$ small enough.
\end{proof}

\subsection{Willmore energy and Area}$\;$\\

The Willmore energy of a surface $f\colon M\to\mathbb S^3$ of constant mean curvature $H$ is given by 
\[\mathcal W(f)=\int_M (H^2+1)dA=(H^2+1)\operatorname{Area}(f).\]
By Theorem \ref{thm:construction-fixed-varphi}, the Willmore energy of $f^g_{\varphi}$ is given by
$\mathcal W(f^g_{\varphi})=\mathcal W(t,\varphi)$ with $t=\frac{1}{2g+2}$ and
\begin{equation}\mathcal W(t,\varphi)=8\pi\left[1-r(t,\varphi)\left(\cos(\varphi)b_0(t,\varphi)-\sin(\varphi)c_0(t,\varphi)\right)\right]\left(H(t,\varphi)^2+1\right).
\end{equation}
At $t=0$, we have $b_0=c_0=0$ and $H=0,$ so $\mathcal W(0,\varphi)=8\pi$ as expected, since $f^g_{\varphi}$ converges to the union of two great spheres as $g\to\infty$.
By Proposition \ref{prop:derivatives}, we obtain after simplification
\begin{eqnarray*}\label{eq:Wstrich}
\frac{\partial \mathcal W}{\partial t}(0,\varphi)&=&
-8\pi\left(\cos(\varphi)b'_0(t,\varphi)-\sin(\varphi)c'_0(t,\varphi)\right)\\
&=&8\pi\left[2\cos(\varphi)^2\log(\cos(\varphi))+2\sin(\varphi)^2\log(\sin(\varphi))\right]
\end{eqnarray*}
see Figure \ref{Wder}.
This gives the following expansion for the Willmore energy of
$f^g_{\varphi}$:
\begin{equation}\label{eq:willfirstord}\mathcal W(f^g_{\varphi})=8\pi\left[1+2\left(\cos(\varphi)^2\log(\cos(\varphi))+\sin(\varphi)^2\log(\sin(\varphi))\right)t+O(t^2)\right]
\quad\mbox{ with }\quad t=\frac{1}{2g+2}.
\end{equation}
In particular, in the minimal case $\varphi=\frac{\pi}{4}$
\begin{equation}\label{eq:lawarea}\mathcal W(f^g_{\varphi})=8\pi\left[1-\ln(2) t+O(t^2)\right]\end{equation}
which we already obtained in \cite{HHT}. 
\begin{figure}[h]
\centering
\vspace{-1.75cm}
\includegraphics[width=0.75\textwidth]{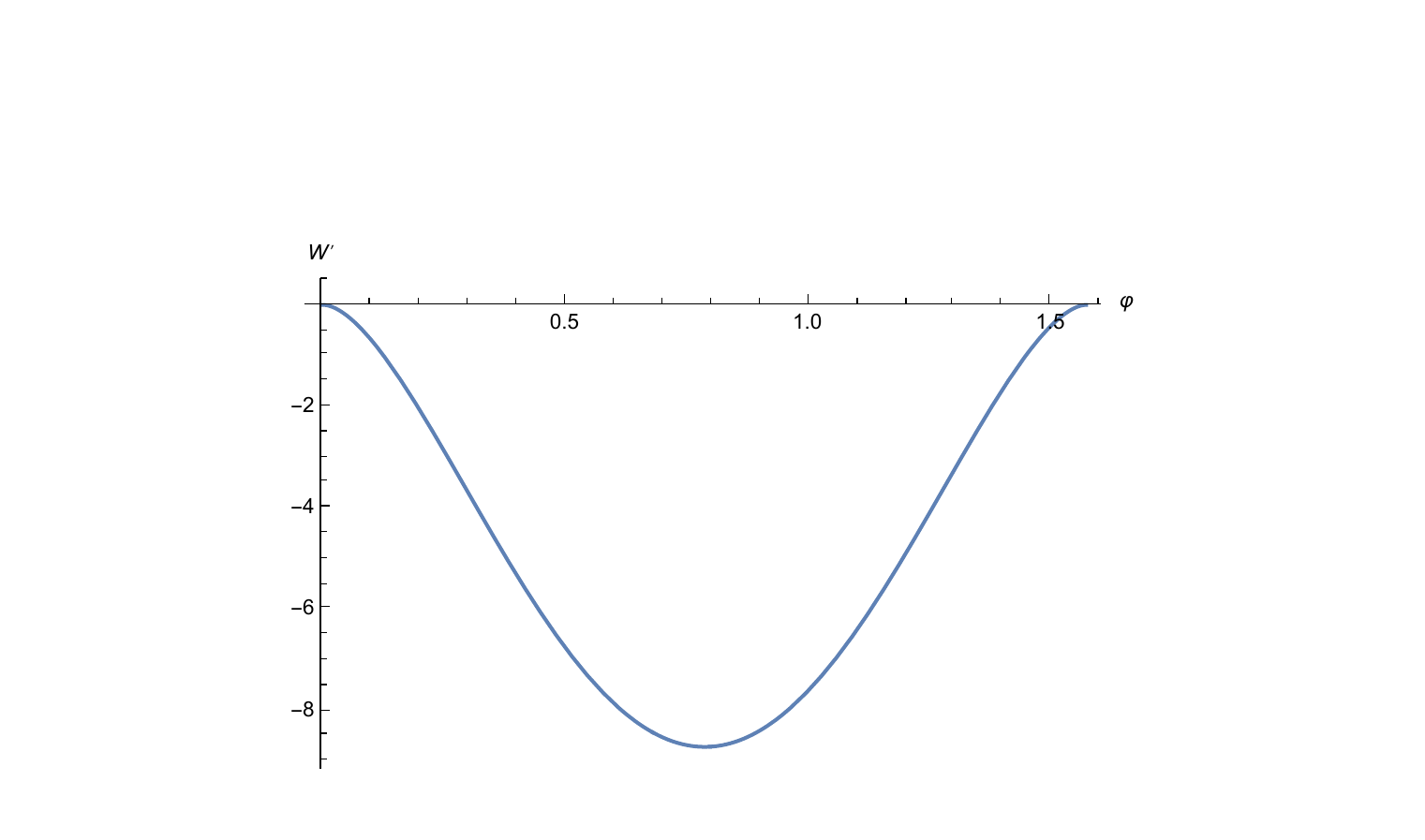}
\vspace{-0.5cm}
\caption{
\footnotesize The graph of $\mathcal W'(0,\varphi)$ for $\varphi \in [0, \tfrac{\pi}{2}].$}
\label{Wder}
\end{figure}
\begin{proposition}\label{prop:willmore}
For $g\gg1 $ fixed the immersions $f^g_\varphi$ satisfies  $\mathcal W(f^g_{\varphi})<8\pi$ for all $\varphi\in(0,\frac{\pi}{2})$.
As a consequence, the surfaces $f^g_{\varphi}$ are embedded by the
Li-Yau estimate \cite{LiYau}.
\end{proposition}
\begin{proof} 
Fix $g\gg1$ such that the complete family of CMC surfaces $f^g_\varphi$ of genus $g$ exists.
We claim that the Willmore energy $\mathcal W(f^g_\varphi)$ is strictly monotonic for $\varphi\in(0,\tfrac{\pi}{4}).$
Note that $\mathcal W(f^g_\varphi)=\mathcal W\big(f^g_{\tfrac{\pi}{2}-\varphi}\big)$ by Theorem \ref{thm:construction-fixed-varphi} (3).
Then the proposition follows, since the Willmore energy for the Lawson surfaces $\xi_{1,g} = f^g_{\tfrac{\pi}{4}}$ is strictly below $8\pi$ and the Willmore energy at $\varphi=0$ is $8\pi$. \\

It is well-known that CMC surfaces are critical points of the 
Willmore functional under conformal variations \cite{BPP}.
The cotangent space to the Teichm\"uller space at the Riemann surface $M$ 
can be identified with the space of holomorphic quadratic differentials $H^0(M,K^2)$.
The Lagrange-multiplier (see \cite[Corollary 16 and Remark 4]{BPP}) for the Willmore functional is given by $\tfrac{1}{2}HQ$,
as the Hopf differential $Q$ is holomorphic for CMC surfaces. Let $\Pi$ denote the projection from the space of immersions to the Teichm\"uller space and consider the image of the map $\Pi(f^g_\varphi)$ which is a real $1-$dimensional submanifold, as $\varphi$ determines the conformal structure of the surface. For the symmetric surfaces we consider the space of  holomorphic and symmetric quadratic differentials is real 2 dimensional. Moreover, it is well known that $\Pi$  fails to be submersive at isothermic, and hence at CMC, surfaces. Therefore, the pairing between
$Q$ and $\frac{\partial}{\partial\varphi}$ is non-degenerate and  the Willmore functional is monotonic along the
family of CMC surfaces as long as $H\neq0.$ Thus, the result follows from Proposition \ref{prop:H}.
\end{proof}
\begin{remark}
Proposition \ref{prop:willmore} can also be proven   analogously to Proposition  \ref{prop:H} via Theorem \ref{thm:construction-fixed-varphi}.
\end{remark}

\begin{corollary}
For all $\varphi \in (0, \tfrac{\pi}{2})$ and $g \gg1$ fixed there exists a
 constrained Willmore minimizer in the conformal class of $f^g_\varphi$.\end{corollary}
\begin{proof}
This follows from Proposition \ref{prop:willmore} and \cite{KuwertSchatzle}, which gives the existence of Willmore energy minimizer with prescribed conformal class 
provided that the infimum energy in the conformal class is below $8 \pi.$
\end{proof}

We end this section by computing some integrals needed to obtain the first order expansion term of the area.
\begin{proposition}
With the notations \eqref{omegaabc} and \eqref{Omegaabc}, we have
\[\int_0^1\omega_a\Omega_b-\omega_b\Omega_a=4\pi\ii\log(\sin(\varphi))-\ii(\pi-2\varphi)\log\left(\frac{1-\cos(\varphi)}{1+\cos(\varphi)}\right)\]
\[\int_0^{\ii}\omega_a\Omega_c-\omega_c\Omega_a=-4\pi\ii\log(\cos(\varphi))+2\ii\varphi\log\left(\frac{1-\sin(\varphi)}{1+\sin(\varphi)}\right).\]
\label{prop:integrals}
\end{proposition}

\begin{proof}
Let $\gamma_1$ be the closed loop given  as the real half-line from $0$ to $+\infty$ composed with the imaginary half-line from $+\ii\infty$ to $0$.
The proposition follows from comparing the results obtained from computing the integrals
\[\int_{\gamma_1}\omega_a\Omega_b-\omega_b\Omega_a
\quad\text{and}\quad 
\int_{\gamma_1}\omega_a\Omega_c-\omega_c\Omega_a,\]
in two different ways -- by applying the symmetries and by applying the Residue Theorem.\\

We have the symmetries
\[\tau^*\omega_a=-\omega_a,\quad 
\tau^*\omega_b=\omega_b\quad\text{and}\quad
\tau^*\omega_c=-\omega_c.\]
Since $\tau(1)=1$ we therefore get
\[\tau^*\Omega_a=-\Omega_a+2\Omega_a(1),\quad
\tau^*\Omega_b=\Omega_b\quad\text{and}\quad
\tau^*\Omega_c=-\Omega_c+2\Omega_c(1).\]
By the change of variable rule
\begin{eqnarray*}
\int_1^{+\infty}\omega_a\Omega_b-\omega_b\Omega_a&=&
-\int_0^1\tau^*(\omega_a\Omega_b-\omega_b\Omega_a)\\
&=&\int_0^1(\omega_a\Omega_b-\omega_b\Omega_a)+2\Omega_a(1)\Omega_b(1).
\end{eqnarray*}
Hence
\begin{equation}
\label{eq:integral1}
\int_0^{+\infty}\omega_a\Omega_b-\omega_b\Omega_a=
2\int_0^1(\omega_a\Omega_b-\omega_b\Omega_a)
+2\Omega_a(1)\Omega_b(1).
\end{equation}
In the same way
\begin{eqnarray*}
\int_1^{+\infty}\omega_a\Omega_c-\omega_c\Omega_a&=&
-\int_0^1 -\omega_a(-\Omega_c+2\Omega_c(1))+\omega_c(-\Omega_a+2\Omega_a(1))\\
&=&-\int_0^1(\omega_a\Omega_c-\omega_c\Omega_a)
+2\Omega_a(1)\Omega_c(1)-2\Omega_c(1)\Omega_a(1).
\end{eqnarray*}
Hence
\begin{equation}
\label{eq:integral2}
\int_0^{+\infty}\omega_a\Omega_c-\omega_c\Omega_a=0.
\end{equation}
Furthermore, we have the symmetries
\[(\delta\circ \tau)^*\omega_a=-\omega_a,\quad
(\delta\circ \tau)^*\omega_b=-\omega_b\quad\text{and}\quad
(\delta\circ \tau)^*\omega_c=\omega_c.\]
Since $\delta\circ \tau(\ii)=\ii$, we have by the same argument, with the
substitution $(a,b,c,1)\to (a,c,b,\ii)$
\begin{equation}
\label{eq:integral3}
\int_0^{+\ii\infty}\omega_a\Omega_c-\omega_c\Omega_a
=2\int_0^{\ii}(\omega_a\Omega_c-\omega_c\Omega_a)
+2\Omega_a(\ii)\Omega_c(\ii)
\end{equation}
\begin{equation}
\label{eq:integral4}
\int_0^{+\ii\infty}\omega_a\Omega_b-\omega_b\Omega_a=0.
\end{equation}
Note that in these computations, $\Omega_{\alpha}$ for $\alpha\in\{a,b,c\}$
denotes the primitive of $\omega_{\alpha}$ on a simply connected domain containing
the real and imaginary axes, respectively.
On the other hand, in the formulas of Proposition \ref{prop:integrals}, $\Omega_{\alpha}$
denotes the analytic continuation along $\gamma_1$. Hence for $\alpha,\beta\in\{a,b,c\}$
\[\int_{\gamma_1}\omega_{\alpha}\Omega_{\beta}=
\int_0^{+\infty}\omega_{\alpha}\Omega_{\beta}
-\int_0^{+\ii\infty}\omega_{\alpha}(\Omega_{\beta}+2\pi\ii),\]
because the analytic continuation of $\Omega_{\beta}$ along $\gamma_1$ coincides on $[0,\ii\infty]$
with $\Omega_{\beta}+\int_{\Gamma_1}\omega_{\beta}$, and $\int_{\Gamma_1}\omega_{\beta}= 2\pi i$ 
by the Residue Theorem.
By symmetry we have
\[\int_0^{+\ii\infty}\omega_a=2\Omega_a(\ii),\quad
\int_0^{+\ii\infty}\omega_b=2\Omega_b(\ii)\quad\text{and}\quad
\int_0^{+\ii\infty}\omega_c=0.\]
Hence using Equations \eqref{eq:integral1}, \eqref{eq:integral4} and then
\eqref{eq:Omega1} and \eqref{eq:OmegaI}
\begin{eqnarray}
\int_{\gamma_1}\omega_a\Omega_b-\omega_b\Omega_a
&=&2\int_0^1(\omega_a\Omega_b-\omega_b\Omega_a)
+2\Omega_a(1)\Omega_b(1)+4\pi\ii(\Omega_b(\ii)-\Omega_a(\ii))
\nonumber\\
&=&2\int_0^1(\omega_a\Omega_b-\omega_b\Omega_a)
+2\ii(\pi-2\varphi)\log\left(\frac{1-\cos(\varphi)}{1+\cos(\varphi)}\right)
+4\pi(\pi-2\varphi).
\label{eq:integral5}\end{eqnarray}
In the same way, using Equations \eqref{eq:integral2} and \eqref{eq:integral3}
\begin{eqnarray}
\int_{\gamma_1}\omega_a\Omega_c-\omega_c\Omega_a
&=&-2\int_0^{\ii}(\omega_a\Omega_c-\omega_c\Omega_a)
-2\Omega_a(\ii)\Omega_c(\ii)-4\pi\ii\,\Omega_a(\ii)
\nonumber\\
&=&-2\int_0^{\ii}(\omega_a\Omega_c-\omega_c\Omega_a)
+4\ii\varphi\log\left(\frac{1-\sin(\varphi)}{1+\sin(\varphi)}\right)
-8\pi\varphi.
\label{eq:integral6}\end{eqnarray}
Next we compute the contour integrals using the Residue Theorem.
For $\alpha\in\{a,b,c\}$, define
$$\widetilde\omega_{\alpha}=\omega_{\alpha}-\frac{dz}{z-p_1}.$$
Then $\widetilde\omega_{\alpha}$ is holomorphic in the disk $\Delta_1$ bounded
by $\gamma_1$ and
$$\widetilde\Omega_{\alpha}(z):=\int_0^z\widetilde\omega_{\alpha}
=\Omega_{\alpha}(z)-\log\left(1-\frac{z}{p_1}\right)$$
is well-defined and holomorphic in $\Delta_1$.
For $\alpha,\beta\in\{a,b,c\}$ we have
\begin{eqnarray*}
\lefteqn{\int_{\gamma_1}\omega_{\alpha}\Omega_{\beta}-\omega_{\beta}\Omega_{\alpha}
=\int_{\gamma_1}\omega_{\alpha}\widetilde\Omega_{\beta}-\omega_{\beta}\widetilde\Omega_{\alpha}+\int_{\gamma_1}(\widetilde\omega_{\alpha}-\widetilde\omega_{\beta})\log\left(1-\frac{z}{p_1}\right)}
\\
&=&2\pi\ii\left(\widetilde\Omega_{\beta}(p_1)-\widetilde\Omega_{\alpha}(p_1)\right)
+\left[\left(\widetilde\Omega_{\alpha}-\widetilde\Omega_{\beta}\right)\log\left(1-\frac{z}{p_1}\right)\right]_{\gamma_1(0)}^{\gamma_1(1)} -\int_{\gamma_1}(\widetilde\omega_{\alpha}-\widetilde\omega_{\beta})\frac{dz}{z-p_1}\\
&=&2\pi\ii\left(\widetilde\Omega_{\beta}(p_1)-\widetilde\Omega_{\alpha}(p_1)\right)
+0-2\pi\ii\left(\widetilde\Omega_{\alpha}(p_1)-\widetilde\Omega_{\beta}(p_1)\right)\\
&=&4\pi\ii\left(\widetilde\Omega_{\beta}(p_1)-\widetilde\Omega_{\alpha}(p_1)\right),
\end{eqnarray*}
where we have used the Residue Theorem, integration by parts, the fact that $\widetilde\Omega_{\alpha}(0)=0$ and the Residue Theorem again.
Note that $\log(1-z)$ is well-defined for $z\in\overline{D}(0,1)\setminus\{1\}.$
Then
\begin{eqnarray}
\int_{\gamma_1}\omega_a\Omega_b-\omega_b\Omega_a
&=&
4\pi\ii\left[2\log(1-p_1/p_4)-2\log(1-p_1/p_3)\right]\nonumber\\
&=&4\pi\ii\left(2\log(\sin(\varphi))+2\ii\varphi-\ii\pi\right)
\label{eq:integral7}
\end{eqnarray}
\begin{eqnarray}
\int_{\gamma_1}\omega_a\Omega_c-\omega_c\Omega_a&=&
4\pi\ii\left[2\log(1-p_1/p_2)-2\log(1-p_1/p_3)\right]\nonumber\\
&=&4\pi\ii\left(2\log(\cos(\varphi))+2\ii\varphi\right).
\label{eq:integral8}\end{eqnarray}
Comparing \eqref{eq:integral5} with \eqref{eq:integral7}
and \eqref{eq:integral6} with \eqref{eq:integral8} gives Proposition \ref{prop:integrals}.
\end{proof}

\section{Higher order derivatives}\label{HOD}
In this section, we present an iterative algorithm to compute the higher order derivatives of the parameters with respect to $t.$ The algorithm works for general $\varphi,$ but simplifies significantly for the most interesting minimal case $\varphi=\pi/4$ as in this case  $\theta=\pi/2$ is constant. We focus here on this minimal case to enhance the presentation. The algorithm has been implemented (for arbitrary $\varphi$) and gives the following estimate for the area of Lawson
minimal surfaces of genus $g\gg 1$ as a function of $t=\frac{1}{2g+2}$:
\def\ar{\alpha}
$$\mbox{Area}(t)=8\pi\big(
1-\ar_1\,t-\ar_3\,t^3
-\ar_5\,t^5-\ar_7\,t^7
+O(t^9)\big).$$
where 
\begin{align*}
&\ar_1=\log(2)\\
&\ar_3= \tfrac{9}{4}\zeta(3)\\
&\ar_5\simeq 3.69962699449761843989338013547104461773632954830910\ldots\\
&\ar_7\simeq -53.1688000602634657601186493744463143722221041377109\ldots.
\end{align*}


\begin{remark}
The value $\ar_1$ was computed in  \eqref{eq:lawarea}, see also \cite{HHT}. The third order term was first computed using numerical values for Multilogarithms and WolframAlpha suggested  $\ar_3= \tfrac{9}{4}\zeta(3).$ Steven Charlton
proved this equality in Appendix \ref{appendixC}. There he also found conjectural expressions for \( \ar_5 \) and \( \ar_7 \), which hold to at least 1000 decimal places.
\end{remark}

\newcommand{\low}[1]{{#1}_{\mbox{\scriptsize lower}}}
\newcommand{\lowind}[2]{{#1}_{\mbox{\scriptsize lower},#2}}
\newcommand{\indlow}[2]{{#1}_{#2,\mbox{\scriptsize lower}}}
\subsection{Higher order derivatives of $\Phi$.}\label{sec:higherderivatives}
Fix $\varphi\in (0, \tfrac{\pi}{2})$ and 
let $\Phi_t=\Phi_{t,\varphi}^{\xx(t,\varphi)}$.
We denote by  $\Phi^{(n)}$ and $\xx^{(n)}$ the $n$-th order derivatives of $\Phi_t$ and $\xx(t,\varphi)$ with respect to $t$ at $t=0$.
The goal in this section is to express $\Phi^{(n)}$ in terms of the lower order derivatives of the parameters and certain iterated integrals.
To do so, it will be convenient to relabel $\omega_1=\omega_a$, $\omega_2=\omega_b$ and
$\omega_3=\omega_c$ and similarly,
$\mathfrak m_1=\mathfrak m_a$, $\mathfrak m_2=\mathfrak m_b$ and
$\mathfrak m_3=\mathfrak m_c$.
Moreover, the coefficients $a$, $b$, $c$ will also be relabelled as $\xx_1$, $\xx_2$ and
$\xx_3$, respectively, whenever the latter notation is more convenient.
Finally we let $\yy_i=r\xx_i$ for $1\leq i\leq 3$. With these notations, we can write the potential as
$$\eta_t^{\xx}=\sum_{i=1}^3 t\yy_i \mathfrak{m}_i \omega_i.$$
We introduce the following notations, for $(i_1,\cdots,i_n)\in\{1,2,3\}^n :$
$$
\yy_{i_1,\cdots,i_n}=\prod_{j=1}^n \yy_{i_j}
\qquad\text{and}\qquad
\mathfrak{m}_{i_1,\cdots,i_n}=\prod_{j=1}^n \mathfrak{m}_{i_j}.$$
Recall that for $i\in\{1,2,3\}$
$$\Omega_i(z)=\int_0^z\omega_i(z).$$
For $n\geq 2$ and $(i_1,\cdots,i_n)\in\{1,2,3\}^n$, we define the iterated integral $\Omega_{i_1,\cdots,i_n}$ recursively by
\begin{equation}\label{eqn:def:omegaintegral}
\Omega_{i_1,\cdots,i_n}(z)=\int_0^z\Omega_{i_1,\cdots,i_{n-1}}\omega_{i_n}.
\end{equation}
We need the value of $\Omega_{i_1,\cdots,i_n}(1)$
which can be expressed in terms of multipolylogs explicitly, see Appendix \ref{section:polylog}.  

\begin{proposition}
For $n\geq 1$,
\begin{equation}
\label{eq:higher1}
\Phi^{(n)}(z)=\sum_{\ell=1}^n\frac{n!}{(n-\ell)!}\sum_{i_1,\cdots,i_{\ell}}
\yy_{i_1,\cdots,i_{\ell}}^{(n-\ell)}
\mathfrak m_{i_1,\cdots,i_{\ell}}
\Omega_{i_1,\cdots,i_{\ell}}(z).
\end{equation}
\end{proposition}
\begin{proof} 
We prove the formula by induction over $n.$
For $n=1$, the formula gives
$\Phi'=\sum_i \cv{\xx}_i\mathfrak{m}_i\Omega_i$
which was already shown in \eqref{eq:dphiprime}.
Assume that Equation \eqref{eq:higher1} holds for all  $k \leq n$.
We have for $k\geq 1$
$$\eta^{(k)}=k\sum_i\yy_i^{(k-1)}\mathfrak{m}_i\omega_i.$$
Differentiating $d\Phi_t=\Phi_t\eta_t$ with respect to $t$ at $t=0$ we obtain by using the Leibnitz rule
\begin{eqnarray*}
d\Phi^{(n+1)}&=&\sum_{k=0}^{n}{n+1\choose k}\Phi^{(k)}\eta^{(n+1-k)}\\
&=&(n+1)\sum_i \yy_i^{(n)}\mathfrak{m}_i\omega_i+\sum_{k=1}^n\sum_{\ell=1}^k{n+1\choose k}\frac{k!}{(k-\ell)!}(n+1-k)\\
&&\times\sum_{i_1,\cdots,i_{\ell+1}}\yy_{i_1,\cdots,i_{\ell}}^{(k-\ell)}
\yy_{i_{\ell+1}}^{(n-k)}\mathfrak{m}_{i_1,\cdots,i_{\ell+1}}\Omega_{i_1,\cdots,i_{\ell}}\omega_{i_{\ell+1}}.
\end{eqnarray*}
Integrating the equation with respect to $z$ and exchange the sums using 
$${n+1\choose k}\frac{k!}{(k-\ell)!}(n+1-k)=\frac{(n+1)!}{(n-k)!(k-\ell)!}$$
gives
\begin{eqnarray*}
\Phi^{(n+1)}&=&(n+1)\sum_i \yy_i^{(n)}\mathfrak{m}_i\Omega_i
+\sum_{\ell=1}^n
\sum_{i_1,\cdots,i_{\ell+1}}
\sum_{k=\ell}^n 
\frac{(n+1)!}{(n-k)!(k-\ell)!}
\yy_{i_1,\cdots,i_{\ell}}^{(k-\ell)}
x_{i_{\ell+1}}^{(n-k)}\mathfrak{m}_{i_1,\cdots,i_{\ell+1}}\Omega_{i_1,\cdots,i_{\ell+1}}\\
&=&(n+1)\sum_i \yy_i^{(n)}\mathfrak{m}_i\Omega_i
+\sum_{\ell=1}^n
\sum_{i_1,\cdots,i_{\ell+1}}\sum_{s=0}^{n-\ell}\frac{(n+1)!}{(n-\ell-s)!s!}
\yy_{i_1,\cdots,i_{\ell}}^{(s)}
x_{i_{\ell+1}}^{(n-\ell-s)}\mathfrak{m}_{i_1,\cdots,i_{\ell+1}}\Omega_{i_1,\cdots,i_{\ell+1}}
\end{eqnarray*}
where we made an index change  $k=\ell+s$.
Using 
$$\frac{(n+1)!}{(n-\ell-s)!s!}=\frac{(n+1)!}{(n-\ell)!}{n-\ell\choose s}$$
and the Leibniz rule this yields
\begin{eqnarray*}
\Phi^{(n+1)}&=&(n+1)\sum_i \yy_i^{(n)}\mathfrak{m}_i\Omega_i
+\sum_{\ell=1}^n\sum_{i_1,\cdots,i_{\ell+1}}\frac{(n+1)!}{(n-\ell)!}
\yy_{i_1,\cdots,i_{\ell+1}}^{(n-\ell)}
\mathfrak{m}_{i_1,\cdots,i_{\ell+1}}\Omega_{i_1,\cdots,i_{\ell+1}}\\
&=&\sum_{m=1}^{n+1}\frac{(n+1)!}{(n+1-m)!}\sum_{i_1,\cdots,i_{m}}
\yy_{i_1,\cdots,i_m}^{(n+1-m)}\mathfrak{m}_{i_1,\cdots,i_m}
\Omega_{i_1,\cdots,i_m}\end{eqnarray*}
with the change of indices $m=\ell+1$.
\end{proof}
\subsection{The iterative algorithm}\label{sec:algorithm}
To enhance presentation we now restrict to the case $\varphi=\tfrac{\pi}{4}$. The general case works analogously, with the same conclusion, but has significantly more terms dealing with the derivatives of $\theta$.

\begin{proposition}
\label{ansatz:higher-order-derivatives}
For $n\geq1$ the quantities $a^{(n)}$, $b^{(n)}$ and $c^{(n)}$ are polynomials in $\lambda$ of degree at most $(n+1)$  and $\mathcal P^{(n+1)}$ is a Laurent polynomial in $\lambda$ of degree at most $(n+1)$.
\end{proposition}
The proof will show how to compute $\xx^{(n)}$ from the lower order derivatives $\xx^{(k)}$ for $k<n$ yielding an iterative algorithm to compute the derivatives of any order.
\begin{proof}
We prove this proposition by induction over $n\geq1$. For $n=1$ we have already explicitly computed $a',$ $b'$ and $c'$ as well as $\mathcal P^{(2)}$ 
in the proof of Proposition \ref{Prop:1stderivatives} and they satisfy the desired properties. Therefore, fix $n\geq 2$ and assume that for $1\leq k<n$, $a^{(k)}$, $b^{(k)}$ and $c^{(k)}$ are polynomials of degree at most $(k+1)$ and $\mathcal P^{(k+1)}$ is a Laurent polynomial of degree at most $(k+1)$.

By Leibnitz rule
\begin{equation}
\label{eq:yk}
\yy_i^{(k)}=\sum_{\ell=0}^k{k\choose \ell}r^{(k-\ell)} \xx_i^{(\ell)}\quad \mbox{ for }\quad 1\leq i\leq 3
\end{equation}
so $\yy_i^{(k)}$ is a polynomial of degree at most $(k+1)$ for $1\leq k<n$. Since $\yy_i^{(0)}=\cv{\xx}_i$,
an easy induction on $\ell$ shows that for $0\leq k<n$, $\yy_{i_1,\cdots,i_{\ell}}^{(k)}$ is a Laurent polynomial of degree at most $k+\ell$.

In the following we use a subscript ``lower'' to denote a quantity which only depends on the lower order derivatives $a^{(k)}$, $b^{(k)}$, $c^{(k)}$, $r^{(k)}$ and $\mathcal P^{(k+1)}$ for $k< n$. Using Equation \eqref{eq:higher1} at $z=1$ and Equation \eqref{eq:yk} with $k=n$, we write
$$\mathcal P^{(n+1)}=(n+1)\sum_{i=1}^3 \left(\xx_i^{(n)}+r^{(n)}\cv{\xx}_i\right)\mathfrak{m}_i\Omega_i(1)
+\low{\mathcal P}^{(n+1)}$$
where
\begin{eqnarray*}
\low{\mathcal P}^{(n+1)}&=&(n+1)\sum_{i=1}^3\sum_{\ell=1}^{n-1}{n\choose \ell}r^{(n-\ell)}\xx_i^{(\ell)}\mathfrak{m}_i\Omega_i(1)\\
&&+\sum_{\ell=2}^{n+1}\frac{(n+1)!}{(n+1-\ell)!}
\sum_{i_1,\cdots,i_{\ell}}\yy_{i_1,\cdots,i_{\ell}}^{(n+1-\ell)}\mathfrak{m}_{i_1,\cdots,i_{\ell}}\Omega_{i_1,\cdots,i_{\ell}}(1).
\end{eqnarray*}

Therefore, $\low{\mathcal P}^{(n+1)}$ is a Laurent polynomial in $\lambda$ of degree at most $(n+1)$.\\

By Leibnitz rule, we have
\begin{eqnarray}
\mathfrak p^{(n+1)}&=&\mathcal P_{21}^{(n+1)}-\mathcal P_{12}^{(n+1)}
+\sum_{k=1}^n{n+1\choose k}\left(\mathcal P_{11}^{(k)}\mathcal P_{21}^{(n+1-k)}
-\mathcal P_{12}^{(k)}\mathcal P_{22}^{(n+1-k)}\right)\nonumber\\
&=&2\pi(n+1)\left(c^{(n)}+r^{(n)}\cv{c}\right)+\low{\mathfrak p}^{(n+1)}
\label{eq:higher-p}
\end{eqnarray}
where 
$$\low{\mathfrak p}^{(n+1)}=\lowind{\mathcal P}{21}^{(n+1)}-\lowind{\mathcal P}{12}^{(n+1)}+\sum_{k=1}^n {n+1\choose k}\left(\mathcal P_{11}^{(k)}\mathcal P_{21}^{(n+1-k)}
-\mathcal P_{12}^{(k)}\mathcal P_{22}^{(n+1-k)}\right).$$
It follows that $\low{\mathfrak p}^{(n+1)}$ is a Laurent polynomial of degree at most $(n+1)$.
From
$\mathfrak p^{(n+1)}=(\mathfrak p^{(n+1)})^*$ and $\cv{c}=\cv{c}^*$
 we obtain
$$(c^{(n)})^+=\frac{1}{2\pi(n+1)}\left(
(\low{\mathfrak p}^{(n+1)})^{-*}-(\low{\mathfrak p}^{(n+1)})^+\right).$$
So $c^{(n)}$ is a polynomial of degree at most $n+1$.
Since $\theta=\tfrac{\pi}{2}$ is constant and $\cv{c}(\ii)=0$, we have
\[\mathcal H_1^{(n)}=\widehat{\mathfrak p}^{(n)}(\ii)=\frac{1}{n+1}\mathfrak p^{(n+1)}(\ii)
=\frac{1}{n+1}\low{\mathfrak p}^{(n+1)}(\ii)+2\pi c^{(n)}(\ii).\]
Then $\mathcal H_1^{(n)}=0$ gives
\[c_0^{(n)}=-(c^{(n)})\mbox{}^+(\ii)-\frac{1}{2\pi(n+1)}\low{\mathfrak p}^{(n+1)}(\ii)\]
so $c^{(n)}$ is completely determined by $\low{\mathfrak p}^{(n+1)}$.
By Proposition 
\ref{propo:symmetries}
for $\varphi=\tfrac{\pi}{4}$ we have
$$b^{(n)}=(-1)^nc^{(n)}.$$

It remains to determine $a^{(n)}$ and $r^{(n)}.$
By Leibniz rule
\begin{eqnarray*}\mathcal K^{(n)}&=&\left(\yy_1^2-\yy_2^2-\yy_3^2\right)^{(n)}\\
&=&2\left(\yy_1^{(0)}\yy_1^{(n)}-\yy_2^{(0)}\yy_2^{(n)}-\yy_3^{(0)}\yy_3^{(n)}\right)+\sum_{k=1}^{n-1} {n\choose k}\left(\yy_1^{(k)}\yy_1^{(n-k)}-\yy_2^{(k)}\yy_2^{(n-k)}-\yy_3^{(k)}\yy_3^{(n-k)}\right)\end{eqnarray*}
$$2\yy_i^{(0)}\yy_i^{(n)}=2\cv{\xx}_i\left(r\xx_i\right)^{(n)}=2\cv{\xx}_i^2 r^{(n)}+2\cv{\xx}_i\xx_i^{(n)}+
2\cv{\xx}_i\sum_{k=1}^{n-1}{n\choose k}r^{(k)}\xx_i^{(n-k)}.$$
Hence
$$\mathcal K^{(n)}=2\left(\cv{a}^2-\cv{b}^2-\cv{c}^2\right)r^{(n)}
+2\cv{a}a^{(n)}+\low{\mathcal K}^{(n)}=
-2r^{(n)}+(\lambda^{-1}-\lambda)a^{(n)}+\low{\mathcal K}^{(n)}$$
with
\begin{eqnarray*}
\low{\mathcal K}^{(n)}&=&
-2\cv{b}b^{(n)}-2\cv{c}c^{(n)}+
2\sum_{k=1}^{n-1}{n\choose k}r^{(k)}\left(\cv{\xx_1}\xx_1^{(n-k)}-\cv{\xx_2}\xx_2^{(n-k)}-\cv{\xx_3}\xx_3^{(n-k)}\right)\\
&&+\sum_{k=1}^{n-1}{n\choose k}\left(\yy_1^{(k)}\yy_1^{(n-k)}-\yy_2^{(k)}\yy_2^{(n-k)}-\yy_3^{(k)}\yy_3^{(n-k)}\right)\end{eqnarray*}
which is a Laurent polynomial  in $\lambda$ of degree at most
$n+2$.
(We have put $b^{(n)}$ and $c^{(n)}$ in $\low{\mathcal K}^{(n)}$
because they have already been determined.)
From $\mathcal K^{(n)}=0$ we obtain
\[\lambda\low{\mathcal K}^{(n)}=(\lambda^2-1)a^{(n)}+2\lambda r^{(n)}.\]
Since $a^{(n)}$ is holomorphic at $\lambda=0$, $\lambda\low{\mathcal K}^{(n)}$
is a polynomial and has degree at most $n+3$. Then $(\lambda^2-1)$ must divide
the polynomial $\lambda\low{\mathcal K}^{(n)}-2\lambda r^{(n)}$, otherwise
$a^{(n)}$ would have poles at $\lambda=\pm 1$. Hence $a^{(n)}$ is a polynomial of
degree at most $n+1$ and can be computed as the quotient of $\lambda\low{\mathcal K}^{(n)}$
divided by  $\lambda^2-1$ with $2\lambda r^{(n)}$ as the remainder.
\end{proof}

\appendix
	{
		\renewcommand\appauthor{\\ \small Steven Charlton}
		\section[Third order coefficient of the area expansion via \( \zeta(3) \)]{Third order coefficient of the area expansion via \( \zeta(3) \)\protect\appauthor}
		\label{appendixC}
		}
	
	The goal of this appendix is to establish the following claim on the value of \( \ar_3 \), the third order coefficient of the area expansion, given in Section \ref{HOD}.
	\begin{theorem}\label{thm:ar3eval}
		The following evaluation holds
		\[
			\ar_3 = \frac{9}{4} \zeta(3) \,,
		\]
	where \( \zeta(3) \) denotes the corresponding value of the Riemann \( \zeta \) function.
	\end{theorem}
	
	To prove this claim, we shall convert the value into an expression involving multiple polylogarithms, and calculate therefrom.
	
	\subsection{Overview of multiple polylogarithms}
	
	For  positive integers \( n_1,\ldots,n_d \in \mathbb{Z}_{>0} \), the multiple polylogarithm (MPL) function \( \Li_{n_1,\ldots,n_d} \) is defined by
	\begin{equation*}
		\Li_{n_1,\ldots,n_d}(z_1,\ldots,z_d)
	=\sum_{0<k_1<k_2<\cdots<k_d}\frac{z_1^{k_1}\cdots z_d^{k_d}}{k_1^{n_1}\cdots k_d^{n_d}} \,.
	\end{equation*}
	This converges for \( z_i \in \mathbb{C}^d \) in the region given by \( \abs{z_i \ldots z_d} < 1 \), where \( 1 \leq i \leq d \).  Two common terms related to MPL's are: the \emph{depth} \( d \) which counts the number of indices \( n_1,\ldots,n_d \), and the \emph{weight} given by the sum \( n_1 + \cdots + n_d \) of the indices.
	
	\begin{remark}
		It should be noted already that two different conventions exist with regard to the summation index of the multiple polylogarithms; one might instead take \( k_1 > k_2 > \cdots > k_d > 0 \) which has the effect of reversing the argument string compared with the definition above.  Different authors (and sometimes the same author in different papers) use each of these different conventions with roughly equal frequency, and so the reader is advised to be aware of which convention is employed in the references herein.
	\end{remark}
	
	An important special case occurs when \( d = 1 \), wherein we reduce to the depth 1, or so-called \emph{classical}, polylogarithm functions
	\[
		\Li_n(z) = \sum_{k=1}^\infty \frac{z^k}{k^n} \,.
	\]
	Moreover, the case \( n = 1 \) evaluates explicitly as \( \Li_1(z) = -\log(1-z) \) in terms of elementary functions.  This hence clarifies the name (multiple) polylogarithm as a (multi-variable) higher order generalisation of the usual logarithm function.  For \( n  > 1\), one has \( \Li_n(1) = \zeta(n) \) in terms of the Riemann zeta function \( \zeta(s) = \sum_{n=1}^\infty n^{-s} \).
	
	Study of the dilogarithm \( \Li_2 \) goes back to at least the works of Leibnitz and of Euler.  Multivariable versions, under the name hyperlogarithms, were also considered by Kummer, and Lappo-Danilevsky.  For extensive overviews, including much of the historical context, see the classical book by Lewin \cite{lewin} and the articles by Kirillov \cite{kirillov} and by Zagier \cite{zagier}.  For much of their history, the polylogarithms remained somewhat obscure functions only known to a few experts.  However in recent years and decades the MPL's have attained a prominent place at the intersection of many fields of mathematics and mathematical physics, in particular appearing in volume computations in hyperbolic geometry, in connection with presentations for (higher) algebraic \( K \)-groups of number fields, as special values of Dedekind \( \zeta \) functions, and in the computation of Feynman integrals and scattering amplitudes in high energy physics.
	
	One of the most curious and important features of MPL's is the plethora of identities and functional relations they satisfy.  Most famously of which is the so-called five-term relation satisfied by the dilogarithm \( \Li_2 \) (discovered, and re-discovered by the likes of Abel, Spence, Kummer, and others in one of many equivalent forms, see Section 1.5 in \cite{lewin}), which states
	\[
		\Li_2(x) + \Li_2(y)
		 - \Li_2\big(\tfrac{x}{1-y} \big)
		 - \Li_2\big(\tfrac{x}{1-y} \big)
		 + \Li_2\big(\tfrac{x y}{(1-x)(1-y)}\big) = -\log(1-x)\log(1-y) \,,
	\]
	for \( \abs{x} + \abs{y} < 1 \).  One also has relations between (some) depth 2 functions and lower depth ones, such as the reduction, given by Zagier \cite{zagier}, of \( \Li_{1,1}(x,y) \) to \( \Li_2 \), namely
	\[
		\Li_{1,1}(x,y) = \Li_1(x) \Li_1(y) + \Li_2\big(\tfrac{-x}{1-x}\big) - \Li_2\big(\tfrac{x(y-1)}{1-x}\big) \,,
	\]
	for \( \abs{xy} < 1, \abs{y} < 1 \).  Both of the above identities can be proven on the level of the power-series expansion, or by differentiation and fixing the constant of integration.  A family of identities, which can also be proven on the level of power-series, and which will be useful later, are the so-called distribution relations, for \( n \in \mathbb{Z}_{>0} \):
	\begin{equation}\label{eqn:dist}
		n^{1-s} \Li_s(z^n) = \sum\nolimits_{\lambda^n = 1} \Li_s(\lambda z)
	\end{equation}	
	Understanding the structure of multiple polylogarithms (and the closely related iterated integrals, below), including the existence and forms of relations like those above, has been the subject of extensive and on-going research since the seminal works of Goncharov \cite{Go-ams,Go-mpl} on the motivic theory of polylogarithms and his conjectures on the structure of the so-called motivic Lie coalgebra.
	
	\subsection{Definition and properties of iterated integrals}
	\label{sec:app:iteratedintegral}
	The connection between the \( \Omega \)-values and MPL's proceeds best through the iterated integral representation of MPL's.  To this end, we recall the general setup of iterated integrals of a family of differential forms \( f_1, \ldots, f_k \) on a manifold \( M \).  See \cite{Chen} for the original foundation, \cite[\Sec2.7]{Go-mpl} for the case of multiple polylogarithms, and \cite[\Sec2.1]{Br-mzv} for the key points in a general context.
	
	Let \( M \) be a smooth \( C^\infty \) manifold over \( \mathbb{R} \), and \( \gamma \colon [0:1] \to M \) a (piecewise) smooth path.  Let \( f_1,\ldots,f_k \) be smooth \( \C \) or \( \R \)-valued 1-forms on \( M \).  Suppose the pullback of \( f_i \) under \( \gamma \) to \( [0,1] \) is given by
	\[
		\gamma^\star(f_i) = F_i(t) d t \,.
	\]
	Then the iterated integral of \( f_1,\ldots,f_k \) along \( \gamma \) is defined as 
	\begin{equation}\label{eqn:def:iteratedintegral}
		\int_\gamma f_{1} \ldots f_{k} \coloneqq \int_{0 < t_1 < \cdots < t_{k}<1} F_1(t_1) d t_1 \cdots F_k(t_k) dt_k \,.
	\end{equation}
	One may extend this by linearity to (formal) linear combinations of differential forms. \medskip
	
	In the setup, a key properties of iterated integrals which we will imply is the so-called shuffle product property.  This says the product of two iterated integrals along the same path \( \gamma \) can be reduced to a linear combination of iterated integrals over certain permutations of all the differential forms, namely:
	\begin{equation}\label{eqn:shuffle}
		\int_{\gamma} f_1 \cdots f_r \int_{\gamma} f_{r+1} \cdots f_{r+s} = \sum_{\sigma \in \Sigma(r,s)} \int_{\gamma} f_{\sigma(1)} \cdots f_{\sigma(r+s)}
	\end{equation}
	where \( \Sigma(r,s) \) is the set of \( (r,s) \)-shuffles, defined as the following subset of \( \Sigma(r+s) \), the permutations on \( (r+s) \) points,
	\begin{align*}
	& \Sigma(r,s) =  \bigg\{ \sigma \in \Sigma(r+s) \:\bigg|\: \begin{aligned}[c] & \sigma(1) < \sigma(2) < \cdots < \sigma(r) \text{ and } \\
	& \sigma(r+1) < \sigma(r+2) < \cdots < \sigma(r+s) \end{aligned}\bigg\} \,.
	\end{align*}
	The name shuffle-product here is justified as the set \( \Sigma(n,m) \) is analogous to the (possible) interleavings undergone by the two sets of cards during a riffle-shuffle.  
	
	Alternatively, one can more algebraically describe the above construction.  Let \( {A} = \{ f_1,\ldots,f_k \} \) be a set of all relevant differential forms, which will constitute our alphabet, and let \( \mathcal{A} = \mathbb{Q}\langle A \rangle \) be (formal) linear combinations of words over \( A \).  Let \( w_1, w_2 \in \mathcal{A} \) be words, and \( u, v \in A \) be letters of the alphabet.  Then recursively define the \emph{shuffle} product \( \shuffle \) by
	\begin{align*}
	(a w_1) \shuffle ( b w_2) &= a (w_1 \shuffle (b w_2)) + b ((a w_1) \shuffle w_2) \\
	w_1 \shuffle \mathbbm{1} &= \mathbbm{1} \shuffle w_1 = w_1
	\end{align*}
	where \( \mathbbm{1} \) denotes the empty word (with 0 letters).  Then \( \shuffle \) is extended by linearity to  all (formal) linear combinations of words \( \mathbb{Q}\langle A \rangle\).  For example (with \( A = \{ a, b, c, d, e \} \))
	\[
	ab \shuffle cde = \begin{aligned}[t]
	&abcde + acbde + acdbe + acdeb + cabde \\
	& + cadbe + cadeb + cdabe + cdaeb + cdeab
	\end{aligned}
	\]
	One can check that 
	\(
		f_1 \cdots f_r \shuffle f_{r+1} \cdots f_{r+s} = \sum_{\sigma \in \Sigma(r,s)} f_{\sigma(1)} \cdots  f_{\sigma(r+s)} \,,
	\)
	so that one may formally write
	\begin{equation}\label{eqn:shuffle2}
		\int_\gamma f_1 \cdots f_r \int_\gamma f_{r+1} \cdots f_{r+s} = \int_\gamma f_1 \cdots f_r \shuffle f_{r+1} \cdots f_{r+s} \,,
	\end{equation}
	if one extends the iterated integral on the right hand side is by linearity to the resulting formal linear combination of differential forms.  One can say that the iterated integral \( \int_\gamma \) is a homomorphism from the algebra of words, with a shuffle product, to the \( \mathbb{C} \) with the usual product of complex numbers.
	
	The claim that \eqref{eqn:shuffle} gives the product of two iterated integrals can be easily verified by virtue of how the integration limits (after pulling back)
	\begin{align*}
	0 < t_1 < t_2 < \cdots < t_k < 1  \\
	0 < t'_1 < t'_2 < \cdots < t'_\ell < 1 
	\end{align*}
	can be interleaved in a variety of compatible ways, via such riffle-shuffles.  (Cases where \( t_i = t'_j \) define an integration domain of measure 0, and so contribute nothing to the product.)  However, as each \( t_i \) and \( t'_j \) appears only in one particular pullback \( \gamma^\star(f_m) \), one must reorder the forms (or pullbacks thereof) in the resulting sum of integrals.  From this one obtains the claimed shuffle product formula above.
	
	\subsection{\( \Omega \)-values and MPL's}
	\label{section:polylog}
	
	Theorem 2.2 in \cite{Go-mpl} provides the following representation of the multiple polylogarithm \( \Li_{n_1,\ldots,n_d} \) in terms of iterated integrals.  Write \( \{x\}^n = x, \ldots, x \) to denote the string \( x \) repeated \( n \) times.  For \( \abs{z_i} < 1 \), we have
	\[
		\Li_{n_1,\ldots,n_d}(z_1,\ldots,z_d) = (-1)^d \int_L \frac{dt}{t - a_1} \Big\{ \frac{dt}{t}  \Big\}^{n_1 - 1} \frac{dt}{t - a_2} \Big\{ \frac{dt}{t} \Big\}^{n_2 - 1}  \cdots \frac{dt}{t - a_d} \Big\{ \frac{dt}{t} \Big\}^{n_d - 1}
	\]
	where \( L : [0,1] \to \mathbb{C} \), \( L(t) = t \) denotes the straight line path from 0 to 1, and \( a_i = \prod_{k=i}^d z_k^{-1} \).	This follows by integration of the geometric series expansion of \( (t-a_i)^{-1} \).
	
	As noted therein, this iterated integral representation provides a means of analytically continuing the multiple polylogarithm \( \Li_{n_1,\ldots,n_d} \) to \( \mathbb{C}^d \) as a multi-valued analytic function. \medskip
	
	We recall now the differential forms \( \omega_1,\omega_2,\omega_3 \) from \eqref{omegaabc} used to define the \( \Omega \)-values.  (Recall, the relabelling \( \omega_1 \leftrightarrow \omega_a, \omega_2 \leftrightarrow \omega_b, \omega_3 \leftrightarrow \omega_c, \) introduced in Section \ref{sec:higherderivatives}.)  The forms are given by
	\begin{equation*}
	\begin{aligned}
	\omega_1 & = \left (\frac{1}{z-p_1} - \frac{1}{z-p_2}+\frac{1}{z-p_3}- \frac{1}{z-p_4}\right)dz 	\\
	\omega_2 & = \left (\frac{1}{z-p_1} - \frac{1}{z-p_2}-\frac{1}{z-p_3} + \frac{1}{z-p_4}\right)dz 
	\\
	\omega_3 & = \left (\frac{1}{z-p_1} +\frac{1}{z-p_2}-\frac{1}{z-p_3}- \frac{1}{z-p_4}\right)dz  \,,
	\end{aligned} 
	\end{equation*}
	where the points \( p_i \) are 
	\[
	p_1 = \exp(\ii \varphi), p_2 = -\exp(-\ii \varphi), p_3 = -p_1, p_4 = -p_2,
	\]
	for \( \varphi = \frac{\pi}{4} \).
	
	From \eqref{eqn:def:omegaintegral}, we see that
	\[
		\Omega_{i_1,\ldots,i_k} = \int_{L} \omega_{i_1} \cdots \omega_{i_k} \,,
	\]
	along the straight line path \( L \) from 0 to 1, as an iterated integral in the sense of Section \ref{sec:app:iteratedintegral}.  From the linearity of iterated integrals, we see that
	\begin{equation*}
	\Omega_{i_1,\ldots,i_n} =\sum_{1\leq j_1,\ldots,j_n\leq 4}\!\! \varepsilon_{i_1,j_1}\cdots\varepsilon_{i_n,j_n} 	\int_L \frac{dz}{z-p_{j_1}} \cdots \frac{dz}{z-p_{j_n}} \,,
	\text{ where } (\varepsilon_{\ell, k}) = \left(\begin{smallmatrix}1&-1&1&-1\\1&-1&-1&1\\1&1&-1&-1\end{smallmatrix}\right) \,, \end{equation*} 
	here \( \varepsilon_{k,\ell} \) is encoding the coefficient of \( \frac{dz}{z-p_k} \) in \( \omega_{\ell} \).
	
	From the iterated integral representation of \( \Li_{1,\ldots,1} \) given above, we see
	\[
		\int_L \frac{dz}{z-p_{j_1}} \cdots \frac{dz}{z-p_{j_n}} = (-1)^n \Li_{1,\ldots,1}\Big(\frac{p_{j_2}}{p_{j_1}}, \frac{p_{j_3}}{p_{j_2}}, \ldots, \frac{p_{j_n}}{p_{j_{n-1}}} , \frac{1}{p_{j_n}}\Big) \,,	
	\]
	which leads to the following expression for \( \Omega_{i_1,\ldots,i_n} \)  in terms of \( 4^n \) MPL's
	\[
	\Omega_{i_1,\ldots,i_n} =	\sum_{1\leq j_1,\ldots,j_n\leq 4}\!\! \varepsilon_{i_1,j_1}\cdots\varepsilon_{i_n,j_n} \cdot (-1)^n \Li_{1,\ldots,1}\Big(\frac{p_{j_2}}{p_{j_1}}, \frac{p_{j_3}}{p_{j_2}}, \ldots, \frac{p_{j_n}}{p_{j_{n-1}}} , \frac{1}{p_{j_n}}\Big) 
	\]
	
	\begin{remark} It can be shown that the convergence region of \( \Li_{n_1,\ldots,n_d}(z_1,\ldots,z_d) \) extends to the boundary of the above region \( \abs{z_i \ldots z_d} \leq 1 \), where \( 1 \leq i \leq d \), if and only if \( (n_d,z_d) \neq (1,1) \).  See Corollary 2.3.10 in \cite{zhao} for the details.  In particular, \( \Li_{n_1,\ldots,n_d}(z_1,\ldots,z_d) \) converges for any choice of indices \( n_1,\ldots,n_d \) and any choice of roots of unity \( z_1,\ldots,z_d \in \mathbb{C}^\times \), except for the case \( (n_d,z_d) = (1,1) \).  Since each \( p_{j_1},\ldots,p_{j_n} \) above is a root of unity, it follows that the ratios are also.  Moreover \( \frac{1}{p_{j_n}} \neq 1 \), so the series definition of the MPL's in the above expression are in fact all convergent.
	\end{remark}
	
	For the purposes of numerical experimentation, one can now delegate to any computer algebra system which implements multiple polylogarithms.  It is implemented in Maple as \texttt{MultiPolylog} and gp/pari as \texttt{polylogmult}, in each case with the opposite summation convention.  The \texttt{GiNaC} library also implements numerical evaluation of multiple polylogarithms via \texttt{Li}.  Although standard numerical integration methods could be used to evaluate \( \Omega_{i_1,\ldots,i_k} \) directly, these specially optimised multiple polylogarithm routines seem to have significantly better efficiency at high precision, making their usage worthwhile.
	
	\subsection{Shuffle and stuffle products} 
	
	By virtue of their iterated integral expression, the MPL functions inherit a shuffle product structure, as discussed above.  For example
	\[
		\Li_1(z_1) \Li_{1,2}(z_2,z_3) = (-1)^{1+2} \int_L \frac{dt}{t-z_1^{-1}} \int_L \frac{dt}{t-(z_2z_3)^{-1}} \frac{dt}{t-z_3^{-1}} \frac{dt}{t}
	\]
	Taking the shuffle product of the forms gives
	\begin{align*}
		& = - \int_L \tfrac{dt}{t-z_1^{-1}} \shuffle \tfrac{dt}{t-(z_2z_3)^{-1}} \tfrac{dt}{t-z_3^{-1}} \tfrac{dt}{t}  = - \int_L \begin{aligned}[t] 
		& \tfrac{dt}{t-z_1^{-1}} \tfrac{dt}{t-(z_2z_3)^{-1}} \tfrac{dt}{t-z_3^{-1}} \tfrac{dt}{t} +
		\tfrac{dt}{t-(z_2z_3)^{-1}} \tfrac{dt}{t-z_1^{-1}} \tfrac{dt}{t-z_3^{-1}} \tfrac{dt}{t} \\
		& + \tfrac{dt}{t-(z_2z_3)^{-1}} \tfrac{dt}{t-z_3^{-1}} \tfrac{dt}{t-z_1^{-1}} \tfrac{dt}{t} +
		\tfrac{dt}{t-(z_2z_3)^{-1}} \tfrac{dt}{t-z_3^{-1}} \tfrac{dt}{t} \tfrac{dt}{t-z_1^{-1}} \end{aligned}
	\end{align*}
	Upon converting back to MPL's and keeping the terms in the same order, we see this is 
	\[
		= \Li_{1,1,2}\Big(\frac{z_1}{z_2 z_3}, z_2,z_3\Big) +
		\Li_{1,1,2}\Big(\frac{z_2 z_3}{z1}, \frac{z_1}{z_3}, z_3\Big) +
		 \Li_{1,1,2}\Big(z_2, \frac{z_3}{z_1}, z_1\Big) + \Li_{1,2,1}\Big(z_2, \frac{z_3}{z_1}, z_1\Big) \,.
	\]
	
	Likewise, the \( \Omega \)-values also inherit a shuffle-product structure, but this one is easier to calculate: the indices of \( \Omega_{i_1,\ldots,i_n} \) correspond directly to the forms \( \omega_{i_1}, \ldots, \omega_{i_n} \) inside the iterated integral.  Therefore the product of two \( \Omega \)-values is given directly by the shuffle product of their indices.  For example
	\[
		\Omega_{2,1} \Omega_{3,1} = \Omega_{2,1,3,1} + 2 \Omega_{2,3,1,1} + \Omega_{3,1,2,1} + 2 \Omega_{3,2,1,1} \,,
	\]	
	since
	\[
		\omega_{2}\omega_{1} \shuffle \omega_{3}\omega_{1} = \omega_2 \omega_1 \omega_3 \omega_1 + 2 \cdot \omega_2 \omega_3 \omega_1 \omega_1 + \omega_3 \omega_1 \omega_2 \omega_1 + 2 \cdot \omega_3 \omega_2 \omega_1 \omega_1 \,.
	\]
	
	By virtue of their definition as the following nested sum
	\begin{equation}\label{eqn:mpl:series}
		\Li_{n_1,\ldots,n_d}(z_1,\ldots,z_d) = \sum_{0 < k_1 < \cdots < k_d} \frac{z_1^{k_1} \cdots z_d^{n_d}}{k_1^{n_1} \cdots k_d^{n_d}} \,,
	\end{equation}
	the multiple polylogarithms more naturally fulfil a so-called stuffle-product (also called harmonic, or quasi-shuffle) product structure.  (See \cite[\Sec2.5]{Go-mpl}, wherein this product is called the first-shuffle product, in contrast to the second-shuffle product given by the iterated integral representation.  See also \cite[\Sec5]{bbbl}, and for an account of the general algebraic framework see \cite{hoff}.)  This is obtained by interleaving two sets summation indices
	\[
		0 < k_1 < k_2 < \cdots < k_d \, \quad \text{ and } \quad 
		0 < \ell_1 < \ell_2 < \cdots < \ell_{d'} \, 
	\]
	in all possible compatible ways.  This time one must include the cases where \( k_i = \ell_j \), in order to obtain a correct expression; while the corresponding \( z_i = z'_j \) case of integrals is of measure zero, and so does not contribute, here the contribution is non-zero and therefore essential to keep.  As a simple example
	\begin{align*}
		& \Li_{n_1,n_2}(z_1,z_2) \Li_{n_3}(z_3)\\
		 &= \sum_{\substack{0 < k_1 < k_2 \,, \\ 0 < \ell_1}} \frac{z_1^{k_1} z_2^{k_2} z_3^{\ell_1}}{k_1^{n_1} k_2^{n_2} k_3^{n_3}} \\
		& = \bigg( \sum_{0 < \ell_1 < k_1 < k_2} + \sum_{0 < k_1 <  \ell_1 < k_2}  + \sum_{0 < k_1 < k_2 < \ell_1 }  + \sum_{0 < \ell_1 = k_1 < k_2}  + \sum_{0 < k_1 < \ell_1 = k_2} \bigg) \frac{z_1^{k_1} z_2^{k_2} z_3^{\ell_1}}{k_1^{n_1} k_2^{n_2} k_3^{n_3}} \\
		&= \begin{aligned}[t] \Li_{n_1,n_2,n_3}\left(z_1,z_2,z_3\right)+\Li_{n_1,n_3,n_2}\left(z_1,z_3,z_2\right)+\Li_{n_3,n_1,n_2}\left(z_3,z_1,z_2\right) & \\
		{} + \Li_{n_1,n_2+n_3}\left(z_1,z_2 z_3\right)+\Li_{n_1+n_3,n_2}\left(z_1 z_3,z_2\right) \hspace{2em} &\end{aligned}
	\end{align*}
	
	More formally this can be encoded as an operation defined on words over the alphabet \( B = \{ Z_{n,z} \mid z \in \mathbb{C} , n \in \mathbb{Z}_{>0} \} \).  We define the linear map \( \Li \) on words in \( \mathcal{B} = \mathbb{Q}\langle B \rangle\) by
	\[
		\Li(Z_{n_1,z_1} \cdots Z_{n_d,z_d}) = \Li_{n_1,\cdots,n_d}(z_1,\ldots,z_d) \,,
	\]
	and extend by linearity to the whole of \( \mathcal{B} \).  Then the stuffle product \( \ast \) is defined recursively via
	\begin{align*}
		& Z_{n_1,z_1} W \ast  Z_{m_1,y_1} W' \\ 
		& = Z_{n_1,z_1}  (W  \ast Z_{m_1,y_1} W') + Z_{m_1,y_1}  (Z_{n_1,z_1}  W \ast W') + Z_{n_1+m+1, z_1 y_1} (W \ast W') \,,
	\end{align*}
	for letters \( Z_{n_1,z_1}, Z_{m_1,y_1} \in B \), and \( W, W' \) words in \( \mathcal{B} \) over the alphabet \( B \).  One takes the initial conditions to be
	\(
		W \ast \mathbbm{1} = \mathbbm{1} \ast W = W ,
	\)
	where \( \mathbbm{1} \) denotes the empty-word containing 0 letters.  Note that the stuffle product recursion contains an extra term \( Z_{n_1+m+1, z_1 y_1} \), compared to the shuffle product recursion, which \emph{stuffs} two letters into one position.  Then \( \Li \) is a homomorphism from the algebra \( \mathcal{B} \) with the stuffle product,  to \( \mathbb{C} \) with the usual product:
	\[
		\Li(W \ast W') = \Li(W) \Li(W') \,.
	\]
	
	\subsection{Parity theorem for multiple polylogarithms}
	
	A useful ingredient for the simplification of \( \ar_3 \) is the so-called parity theorem for multiple polylogarithms, established by Panzer \cite{panzer}.  (A related result, valid on the unit \( m \)-torus is already given by Goncharov in \cite[\Sec2.6]{Go-mpl}, under the title the inversion formula.)  Roughly, the parity theorem states that
	\[
		\Li_{n_1,\ldots,n_d}(z_1,\ldots,z_d) - (-1)^{n_1 + \cdots + n_d - d} \Li_{n_1,\ldots,n_d}(z_1^{-1}, \ldots, z_d^{-1})
	\]
	reduces to a combination of lower depth MPL's \( \Li_{n_1,\ldots,n_{d'}} \), \( d' < d \) and products of lower weight functions.  Here $\Li_{n_!,\ldots,n_d}(z_1,\ldots,z_d) $, denotes some suitably analytically continued version of the multiple polylogarithms, extended by the iterated integral definition, along a straight-line path starting near \( (0, \ldots, 0) \in \mathbb{C}^d \).  This can be identified with the value given by the series definition in that case that \( z_1,\ldots,z_d \in \mathbb{C}^\ast \) lie on the unit circle, and \( (n_d,z_d) \neq (1,1) \) for convergence reasons.
	
	This parity theorem generalises the Jonqui\`ere inversion relation \cite{jonq} of the classical polylogarithms \( \Li_n(z) \), namely 
	\[
		\Li_n(z) + (-1)^n \Li_n(z^{-1}) = -\frac{(2\pi \ii)^n}{n!} B_n \Big( \frac{1}{2} + \frac{\log(-z)}{2 \ii \pi } \Big) \,, z \in \mathbb{C}\setminus[0,\infty) \,,
	\]
	to the higher depth multiple polylogarithms.  Here \( B_n(x) \) denotes the Bernoulli polynomials generated by \( \tfrac{t e^{x t}}{e^t - 1} = \sum_{n=0}^\infty B_n(x) \frac{t^n}{n!} \).  In the case of roots of unity, the parity theorem re-establishes and extends depth-weight parity reduction result for multiple zeta values to higher order (so-called coloured) multiple zeta values. \medskip
	
	The parity theorem in \cite{panzer} is given in an explicit and algorithmic way, allowing one to generate the parity theorem identity for \( \Li_{n_1,\ldots,n_d}(z_1,\ldots,z_d) \) in a systematic way.  Panzer provides a Maple program for this purpose, and explicitly gives all identities up to weight 4 in (computer-readable) Maple format, as supplementary material to the paper \cite{panzer}.  The simplest examples of which (beyond depth 1) include:
	\begin{align}
		& \begin{aligned}[c] \label{eqn:mpl:li11inv} & \Li_{1,1}\left(z_1,z_2\right) - \Li_{1,1}\left({z_1}^{-1},{z_2}^{-1}\right) \\ 
		& = -\Li_2\left(z_1\right)+\Li_2\left(z_2 \right)-\Li_2\left(z_1 z_2\right)-\Li_1\left(z_1\right) \log \left(-z_2\right)+\Li_1\left(z_1\right) \log \left(-z_1	z_2\right) \\ 
		& \quad -\Li_1\left(z_2\right) \log \left(-z_1 z_2\right)+\tfrac{1}{2} \log ^2\left(-z_2\right)-\log \left(-z_1 z_2\right) \log\left(-z_2\right)+\tfrac{\pi ^2}{6} \end{aligned} \\[2ex]
		& \begin{aligned}[c] & \Li_{1,2}\left(z_1,z_2\right) + \Li_{1,2}\left({z_1}^{-1},{z_2}^{-1}\right) \\
		& = 
		\Li_3\left(z_1\right)
		+2	\Li_3\left(z_2\right)
		-\Li_3\left(z_1 z_2\right)
		-\tfrac{1}{2} \Li_1\left(z_1\right) \log ^2\left(-z_2\right)
		  +\tfrac{1}{2}\Li_1\left(z_1\right) \log ^2\left(-z_1 z_2\right)
		  \\ &\quad {}
		-\Li_2\left(z_1\right) \log \left(-z_1 z_2\right)
		-\Li_2\left(z_2\right) \log \left(-z_1 z_2\right)
		+\tfrac{1}{3} \log ^3\left(-z_2\right)
		 \\ &\quad {}
	  	-\tfrac{1}{2} \log \left(-z_1	z_2\right) \log ^2\left(-z_2\right)
		+\tfrac{1}{3} \pi ^2 \log \left(-z_2\right)
		-\tfrac{1}{6} \pi ^2 \log \left(-z_1 z_2\right) \\
		\end{aligned}  \\[2ex]
		& \begin{aligned}[c] & \Li_{2,1}\left(z_1,z_2\right) + \Li_{2,1}\left({z_1}^{-1},{z_2}^{-1}\right) \\ 
		& = 
		-\tfrac{1}{6} \pi ^2 \Li_1\left(z_2\right)
		-2	\Li_3\left(z_1\right)
		-\Li_3\left(z_2\right)
		-\Li_3\left(z_1 z_2\right)
		-\tfrac{1}{2} \Li_1\left(z_2\right) \log^2\left(-z_1 z_2\right)
		\\ &\quad {}
		-\Li_2\left(z_1\right) \log \left(-z_2\right) 
		 +\Li_2\left(z_1\right) \log \left(-z_1	z_2\right)
		 +\Li_2\left(z_2\right) \log \left(-z_1 z_2\right)
		 \\ &\quad {}
		 -\tfrac{1}{6} \log ^3\left(-z_2\right)
		   +\tfrac{1}{2} \log \left(-z_1 z_2\right) \log ^2\left(-z_2\right)
		 -\tfrac{1}{2} \log ^2\left(-z_1 z_2\right) \log \left(-z_2\right)
		 \\ &\quad {}
		 -\tfrac{1}{3} \pi ^2 \log\left(-z_2\right)
		 +\tfrac{1}{6} \pi ^2 \log \left(-z_1 z_2\right)
		\end{aligned}
	\end{align}
	
	The explicit statement of the parity theorem in \cite{panzer} shows that only arguments which occur in \( \Li \) are those given by consecutive products \( z_\mu z_{\mu+1} \cdots z_{\nu} \).  When we apply this parity theorem, each \( z_i \) will be a root of unity, and so (after any necessary regularisation), the resulting MPL's will be convergent, avoiding any need for analytic continuation or multi-valuedness considerations in our evaluations.
				
	\subsection{Simplification and evaluation of \( \ar_3 \)}

	In terms of \( \Omega_{i_1,\ldots,i_k} \), the implemented algorithm from Section \ref{sec:algorithm} produces the following fully expanded output for \( \ar_3 \), namely:
	\begin{equation}\label{eqn:ar3:formula}
	\ar_3 = -\frac{\ii}{8 \pi ^3}  A -  \frac{1}{32 \pi ^2} B + \frac{\ii}{96 \pi } C \,,
	\end{equation}
	where
	\begin{align*}
	& A = 
	\begin{aligned}[t]
	& -3 \Omega _1^2 \Omega _2^2 \Omega _{1,2}
	+3 \Omega _1^2 \Omega _2^2 \Omega _{2,1}
	+3 \Omega _1 \Omega _2 \Omega _{1,2}^2
	+3 \Omega _1 \Omega _2 \Omega _{2,1}^2
	-6 \Omega _1 \Omega _2 \Omega _{1,2} \Omega _{2,1} \\
	& -\Omega _{1,2}^3
	+\Omega _{2,1}^3
	-3 \Omega _{1,2} \Omega _{2,1}^2
	+3 \Omega _{1,2}^2 \Omega _{2,1}
	+\Omega _1^3 \Omega _2^3
	\end{aligned} \\[1ex]
	& B = \begin{aligned}[t]
	& 
	-6 \Omega _3 \Omega _1^2 \Omega _{1,2}
	-12 \Omega _2 \Omega _1^2 \Omega _{1,3}
	+6 \Omega _3 \Omega _1^2 \Omega _{2,1}
	+12 \Omega _2 \Omega _1^2 \Omega _{3,1}
	+12 \Omega _1 \Omega _{1,2} \Omega _{1,3}
	\\ &
	-12 \Omega _1 \Omega _{1,3} \Omega _{2,1}
	-2 \Omega _2^2 \Omega _1 \Omega _{2,3}
	-12 \Omega _1 \Omega _{1,2}	\Omega _{3,1}
	+12 \Omega _1 \Omega _{2,1} \Omega _{3,1}
	+2 \Omega _2^2 \Omega _1 \Omega _{3,2}
	\\ &
	+12 \Omega _2 \Omega _1 \Omega _{1,1,3}
	-12 \Omega _2 \Omega _1 \Omega _{1,3,1}
	-2 \Omega _2 \Omega _1 \Omega _{2,2,3}
	+2 \Omega _2 \Omega _1 \Omega _{2,3,2}
	+12 \Omega _2 \Omega _1 \Omega _{3,1,1}
	\\ &
	-2 \Omega _2 \Omega	_1 \Omega _{3,2,2}
	-4 \Omega _3^3 \Omega _{1,2}
	+\Omega _2^2 \Omega _3 \Omega _{1,2}
	+4 \Omega _3^3 \Omega _{2,1}
	-\Omega _2^2 \Omega _3 \Omega _{2,1}
	+2 \Omega _2 \Omega _{1,2} \Omega _{2,3}
	\\ &
	-2 \Omega _2 \Omega _{2,1} \Omega _{2,3}
	-2 \Omega _2 \Omega _{1,2} \Omega _{3,2}
	+2 \Omega _2 \Omega _{2,1} \Omega_{3,2}
	-12 \Omega _{1,2} \Omega _{1,1,3}
	+12 \Omega _{2,1} \Omega _{1,1,3}
	\\ &
	+12 \Omega _{1,2} \Omega _{1,3,1}
	-12 \Omega _{2,1} \Omega _{1,3,1}
	+2 \Omega _{1,2} \Omega _{2,2,3}
	-2 \Omega _{2,1} \Omega _{2,2,3}
	-2 \Omega _{1,2} \Omega _{2,3,2}
	\\ &
	+2 \Omega _{2,1} \Omega _{2,3,2}
	-12 \Omega _{1,2} \Omega _{3,1,1}
	+12	\Omega _{2,1} \Omega _{3,1,1}
	+2 \Omega _{1,2} \Omega _{3,2,2}
	-2 \Omega _{2,1} \Omega _{3,2,2}
	\\ &
	+6 \Omega _2 \Omega _3 \Omega _1^3
	+4 \Omega _2 \Omega _3^3 \Omega _1
	-\Omega _2^3 \Omega _3 \Omega _1 
	\end{aligned} \\[1ex]
	& {} C = \begin{aligned}[t]
	& 
	-18 \Omega _1^2 \Omega _{1,2}
	+18 \Omega _1^2 \Omega _{2,1}
	+36 \Omega _1 \Omega _{1,1,2}
	-36 \Omega _1 \Omega _{1,2,1}
	+36 \Omega _1 \Omega _{2,1,1}
	+6 \Omega _1 \Omega _{2,3,3}
	\\ &
	-6 \Omega _1 \Omega _{3,2,3}
	+6 \Omega _1 \Omega _{3,3,2}
	-3 \Omega _2^2 \Omega _{1,2}
	-3 \Omega _3^2 \Omega _{1,2}
	+3\Omega _2^2 \Omega _{2,1}
	+3 \Omega _3^2 \Omega _{2,1}
	\\ &
	+6 \Omega _{1,3} \Omega _{2,3}
	-6 \Omega _{2,3} \Omega _{3,1}
	-6 \Omega _{1,3} \Omega _{3,2}
	+6 \Omega _{3,1} \Omega _{3,2}
	+6 \Omega _2 \Omega _{1,2,2}
	-6 \Omega _3 \Omega _{1,2,3}
	\\ &
	+6 \Omega _3 \Omega _{1,3,2}
	+6 \Omega _2 \Omega _{1,3,3}
	-6 \Omega _2 \Omega	_{2,1,2}
	+6 \Omega _3 \Omega _{2,1,3}
	+6 \Omega _2 \Omega _{2,2,1}
	-6 \Omega _3 \Omega _{2,3,1}
	\\ &
	-6 \Omega _3 \Omega _{3,1,2}
	-6 \Omega _2 \Omega _{3,1,3}
	+6 \Omega _3 \Omega _{3,2,1}
	+6 \Omega _2 \Omega _{3,3,1}
	-36 \Omega _{1,1,1,2}
	+36 \Omega _{1,1,2,1}
	\\ &
	-36 \Omega _{1,2,1,1}
	-6 \Omega _{1,2,2,2}
	-6 \Omega_{1,2,3,3}
	+6 \Omega _{1,3,2,3}
	-6 \Omega _{1,3,3,2}
	+36 \Omega _{2,1,1,1}
	\\ &
	+6 \Omega _{2,1,2,2}
	+6 \Omega _{2,1,3,3}
	-6 \Omega _{2,2,1,2}
	+6 \Omega _{2,2,2,1}
	-6 \Omega _{2,3,1,3}
	+6 \Omega _{2,3,3,1}
	\\ &
	-6 \Omega _{3,1,2,3}
	+6 \Omega _{3,1,3,2}
	+6 \Omega _{3,2,1,3}
	-6 \Omega _{3,2,3,1}
	-6 \Omega _{3,3,1,2}
	+6 \Omega_{3,3,2,1}
	\\ &
	+6 \Omega _2 \Omega _1^3
	+\Omega _2^3 \Omega _1
	\end{aligned}
	\end{align*}
	
	We begin by noticing the following simplifications for the various terms \( A, B, C \) in the formula above.  Firstly
	\[
		A = (-\Omega _{1,2}+\Omega _{2,1}+\Omega _1 \Omega _2)^3 = (2\Omega _{2,1})^3 \,.
	\]
	The first equality comes simply by factoring the expression for \( A \) formally, the second equality comes from applying the shuffle product to \( \Omega_1 \Omega_2 = \Omega_{2,1} + \Omega_{1,2} \).  Likewise, \( B \) can be formally factored as 
	\begin{align*}
		B ={}  & \big(-\Omega _{1,2}+\Omega _{2,1}+\Omega _1 \Omega _2 \big) \cdot \\
			&  \big(
			\begin{aligned}[t] 
			& -12 \Omega _1 \Omega _{1,3}
			-2 \Omega _2 \Omega _{2,3}
			+12 \Omega _1 \Omega _{3,1}
			+2 \Omega _2 \Omega _{3,2}
			+12 \Omega _{1,1,3}
			-12 \Omega _{1,3,1} \\
			& \quad {} -2 \Omega _{2,2,3}
			+2 \Omega _{2,3,2}
			+12 \Omega _{3,1,1}
			-2 \Omega _{3,2,2}
			+4 \Omega _3^3
			+6	\Omega _1^2 \Omega _3
			-\Omega _2^2 \Omega _3 \big) \end{aligned}
	\end{align*}
	Now apply the shuffle product to each bracket individually, and we find
	\[
		B = \big( 2 \Omega _{2,1} \big) \cdot \big( -8 \Omega _{2,2,3}+48 \Omega _{3,1,1}+ 4 \Omega _{3}^3 \big) \,.
	\]
	The term \( C \) does not factor formally; nevertheless on application of the shuffle product, we find
	\[
		C = 288 \Omega _{2,1,1,1}+48 \Omega _{2,1,3,3}+48 \Omega _{2,2,2,1}-48 \Omega _{3,1,2,3}+48 \Omega _{3,3,2,1} \,.
	\]
	
	The upside of this is that we have reduced the expression for \( \ar_3 \) to the following:
	\begin{align*}
		\ar_3 = {} & - \frac{\ii}{\pi ^3} \Omega _{2,1}^3 - \frac{1}{16 \pi ^2}\Omega _{2,1} \big(-8 \Omega _{2,2,3}+48 \Omega _{3,1,1}+ 4 \Omega _{3}^3 \big)  \\
		& {} + \frac{\ii}{2 \pi } \big(6 \Omega _{2,1,1,1}+\Omega _{2,1,3,3}+\Omega _{2,2,2,1}-\Omega _{3,1,2,3}+\Omega _{3,3,2,1}\big)
	\end{align*}
	
	We now make the following claims, on how each of the above \( \Omega \)-values evaluates in terms of \( \zeta(3), \pi \) and \( \log(2) \).
	\begin{alignat*}{3}
		\Omega_3 &= \ii \pi 
		&\Omega_{2,1} &= -\ii \pi \log(2) \\
		\Omega_{2,2,3} &= \frac{\ii \pi^3}{12} 
		&\Omega_{3,1,1} &= \frac{ \ii \pi}{2} \log(2)^2 - \frac{\ii \pi^3}{12} \\	
		\Omega_{2,1,1,1} &= \frac{\ii \pi^3}{12} \log(2) - \frac{\ii\pi}{6} \log(2)^3 - \frac{\ii \pi}{4} \zeta(3) \quad\quad
		&\Omega_{2,1,3,3} &= \frac{\ii \pi^3}{4} \log(2) - \ii \pi \zeta(3) \\
		\Omega_{2,2,2,1} &= -\frac{\ii \pi^3}{12} \log(2) + \frac{\ii \pi}{4} \zeta(3) 
		&\Omega_{3,1,2,3} &= - \frac{\ii \pi^3}{4} \log(2) + \frac{13}{8} \ii \pi \zeta(3) \\
		\Omega_{3,3,2,1} &= \frac{\ii \pi^3}{6} \log(2) - \frac{5}{8} \ii \pi \zeta(3)			
	\end{alignat*}
	It is then straightforward to substitute these values into \( \ar_3 \) in order to obtain the claimed evaluation
	\[
		\ar_3 = \frac{9}{4} \zeta(3) \,.
	\]
	This will complete the proof of Theorem \ref{thm:ar3eval}, once we show the above claim on the \( \Omega \)-values.
	
	\subsection{Evaluation of the required \( \Omega \)-values}
	
	We treat each of the above \( \Omega \)-values in turn, in a rather naive way for the moment .  
	
	\paragraph{Weight 1:} The values
	\[
		\Omega_1 = \frac{\ii \pi}{2} \,, \quad \Omega_2 = -2 \log(1 + \sqrt{2}) \,, \quad \Omega_3 = \ii \pi
	\]
		follow directly from the calculations given in \eqref{eq:Omega1}, when \( \varphi = \frac{\pi}{4} \).
	
	\paragraph{Weight 2:} Directly expanding out gives, writing \( \eta = \frac{1 + \ii}{\sqrt{2}} \) for notational ease, the following expression for \( \Omega_{2,1}\)
	\begin{align}\label{eqn:omega21:asmpl}
		\begin{aligned}[c]
		\Omega_{2,1} = 	& \Li_{1,1}\left(-1,\eta ^3\right)-\Li_{1,1}\left(-1,{\eta ^{-3}}\right)
			\quad {} + \Li_{1,1}(-1,\eta )-\Li_{1,1}\left(-1,\eta ^{-1}\right) 
			\\ & 
			{} + \Li_{1,1}\left(-i,\eta ^{-3}\right)-\Li_{1,1}\left(i,\eta ^3\right) 
			\quad {} + \Li_{1,1}\left(-i,\eta^3\right)-\Li_{1,1}\left(i,\eta ^{-3}\right)
			\\ & 
			{} + \Li_{1,1}(i,\eta )-\Li_{1,1}\left(-i,\eta ^{-1}\right)
			\quad {} + \Li_{1,1}\left(i,\eta ^{-1}\right)-\Li_{1,1}(-i,\eta )
			\\ &
			{} +  \Li_{1,1}\left(1,\eta ^{-3}\right)-\Li_{1,1}\left(1,\eta^3\right)
			\quad {} + \Li_{1,1}\left(1,\eta ^{-1}\right)-\Li_{1,1}(1,\eta )
		\end{aligned}
	\end{align}
	We have deliberately grouped the terms here in such a way that they already simplify by application of the parity theorem \cite{panzer}.  That is to say, by substituting
	\( z_1 = -1, z_2 = \eta^3 \) in \eqref{eqn:mpl:li11inv}, we find 
	\begin{align*}
		& \Li_{1,1}\left(-1,\eta ^3\right) - \Li_{1,1}\left({-1}^{-1},{\eta ^{-3}}\right) \\ 
		& = -\Li_2\left(-1\right)+\Li_2\left(\eta ^3 \right)-\Li_2\left(-\eta ^3\right)-\Li_1\left(-1\right) \log \left(-\eta ^3\right)+\Li_1\left(-1\right) \log \left(\eta ^3\right) \\ 
		& \quad -\Li_1\left(\eta ^3\right) \log \left(\eta ^3\right)+\frac{1}{2} \log ^2\left(-\eta ^3\right)-\log \left(\eta ^3\right) \log\left(-\eta ^3\right)+\frac{\pi ^2}{6}
	\end{align*}
	Using various standard evaluations, such as \( \Li_1(-1) = \log(2) \), or \( \Li_2(-1) = -\frac{1}{2} \Li_2(1) = - \frac{\pi^2}{12} \), which follows from the distribution relation \( \Li_2(x) + \Li_2(-x) = \frac{1}{2} \Li_2(x^2) \) in the case \( x = 1 \) together with \( \Li_2(1) = \sum_{n=0}^\infty \frac{1}{n^2} = \zeta(2) \), we find:
	\begin{align*}
	& \Li_{1,1}\left(-1,\eta ^3\right) - \Li_{1,1}\left(-1,{\eta ^{-3}}\right) \\ 
	& = \frac{\pi^2}{32} - \ii \pi \log(2) - \frac{3}{4} \ii \pi \Li_1(\eta^3) + \Li_2\left(\eta ^3 \right)-\Li_2\left(-\eta ^3\right) 
	\end{align*}
	Since \( -\eta^3 = \eta^{-1} \), one can also apply parity to \( \Li_2(-\eta^3) = \Li_2(\eta^{-1}) \) to write
	\[	
		\Li_2(-\eta^3)  + \Li_{2}(\eta) = \frac{11\pi^2}{96} \,,
	\]
	and hence
	\begin{align*}
	& \Li_{1,1}\left(-1,\eta ^3\right) - \Li_{1,1}\left(-1,{\eta ^{-3}}\right) \\ 
	& = -\frac{\pi^2}{12} - \ii \pi \log(2) - \frac{3}{4} \ii \pi \Li_1(\eta^3) + \Li_2\left(\eta \right) + \Li_2\left(\eta ^3\right) \,.
	\end{align*}
 	Similar calculations for the remainder of \eqref{eqn:omega21:asmpl} produce the evaluation
	\begin{align*}
		\Omega_{2,1} = {} & 
		- \frac{7\pi^2}{4} + 2 \ii \pi \Li_1(-1) - 2 \ii \pi \Li_1(\ii) + \frac{3}{4} \ii \pi \Li_1(\eta^{-3}) + \frac{3}{4} \ii \pi \Li_1(\eta^{-1}) \\ &
		- \frac{3}{4} \ii \pi \Li_1(\eta) - \frac{3}{4} \ii \pi \Li_1(\eta^3) - 2 \Li_2(-1) + 2 \Li_2(1) \,,
	\end{align*}
	whereupon evaluating via \( \Li_1(x) = -\log(1-x) \) and substituting the evaluations \( \Li_2(-1) = -\frac{\pi^2}{12} \), \( \Li_2(1) = \zeta(2) \) already discussed, we find
	\[
		\Omega_{2,1} = -\ii \pi \log(2) \,,
	\]
	as claimed.
	
	\begin{remark}
		From Proposition \ref{prop:integrals}, at \( \varphi = \frac{\pi}{4} \), it follows that
		\[
			\Omega_{2,1}- \Omega_{1,2} = -2 \ii \pi \log(2) +  \ii \pi \log(1 + \sqrt{2})
		\]
		Since
		\[
			\Omega_{2,1} + \Omega_{1,2} = \Omega_{1} \Omega_{2} = -\ii \pi \log(1 + \sqrt{2})
		\]
		using the shuffle product, one can solve these simultaneously for \( \Omega_{2,1} \) and \( \Omega_{1,2} \) to obtain directly that
		\[
			\Omega_{2,1} = -\ii \pi \log(2) \,, \quad \Omega_{1,2}= \ii \pi ( \log(2) - \log(1  +\sqrt{2}) ) \,.
		\]
	\end{remark}
	
	\paragraph{Weight 3:} 
	
	Application of the parity theorem (in weight 3, then weight 2) to \( \Omega_{2,2,3} \) eliminates all pure weight 3 terms from \( \Omega_{2,2,3}\) and directly produces the following, where as before \( \eta = \frac{1+\ii}{\sqrt{2}} \) for notational simplicity:
	\begin{align*}
		\Omega_{2,2,3} = {} & 
		\tfrac{179 \ii \pi ^3}{96}
		-2 \ii \pi  \Li_{1,1}(-1,-\ii)
		-2 \ii \pi  \Li_{1,1}(-\ii,-1)
		+4 \ii \pi  \Li_{1,1}(1,-1)
		+2 \ii \pi  \Li_2(-\ii)		
		\\ &
		+\ii \pi  \Li_1(-1) \Li_1\left(\eta ^3\right)
		-\ii \pi  \Li_1(-1) \Li_1\left(\eta ^5\right)
		+\ii \pi  \Li_1(-1) \Li_1\left(\eta ^7\right)
		\\ &
		-2 \ii \pi  \Li_1(\ii) \Li_1\left(\eta^3\right)
		+2 \ii \pi  \Li_1(\ii) \Li_1\left(\eta ^5\right)
		+2 \ii \pi  \Li_1(\ii) \Li_1\left(\eta ^7\right)
		\\ &
		-\ii \pi  \Li_1(-1) \Li_1(\eta )
		-2 \ii \pi  \Li_1(\ii) \Li_1(\eta )
		-2 \ii \pi  \Li_1(-1)^2
		+2 \ii \pi  \Li_1(i) \Li_1(-1)
		\\ &
		-2 \ii \pi  \Li_1(\ii)^2
		+2 \ii \pi  \Li_1(-\ii) \Li_1(\ii)
		-4 \pi ^2 \Li_1(i)
		-\tfrac{55}{96} \pi ^2 \Li_1\left(\eta ^3\right)
		+\tfrac{55}{96} \pi ^2 \Li_1\left(\eta ^5\right)
		\\ &
		+\tfrac{45}{32} \pi ^2 \Li_1\left(\eta ^7\right)
		-\tfrac{45}{32} \pi ^2 \Li_1(\eta )
		+\tfrac{1}{2} \pi ^2 \Li_1(-1)
		+\pi ^2 \Li_1(-\ii) \,.
	\end{align*}
	Simplifying using the straightforward evaluations of \( \Li_1(\eta^j) = -\log(1-\eta^j) \) reduces this to
	\begin{align}\label{eqn:omega223:step2}
		\begin{aligned}
		\Omega_{2,2,3} = {}
		&
		\frac{\ii \pi ^3}{6}
		-2 \ii \pi  \Li_{1,1}(-1,-\ii)
		-2 \ii \pi  \Li_{1,1}(-\ii,-1)
		+4 \ii \pi  \Li_{1,1}(1,-1) 
		\\ &
		{} + \ii \pi  \Li_2(-1)
		-2 \ii \pi \Li_2(\ii)
		- \ii \pi  \log ^2(2)
		-\frac{1}{2} \pi ^2 \log (2)
		\end{aligned}
	\end{align}
	We now utilize two simple identities which follow from the shuffle and stuffle product of multiple polylogarithms (abusing notation to write the specific product explicitly for clarity):
	\begin{align*}
		\Li_1(-1) \shuffle \Li_1(-1) &= 2 \Li_{1,1}(1, -1) \\
		\Li_1(-1) \ast \Li_1(-\ii) &= \Li_{1,1}(-1,-\ii) +  \Li_{1,1}(-\ii,-1) + \Li_2(\ii)
	\end{align*}
	Hence, when we substitute these into \eqref{eqn:omega223:step2}, and evaluate with \( \Li_1(\eta^j) = -\log(1 - \eta^j) \) and \( \Li_2(-1) = -\frac{\pi^2}{12} \), we find
	\[
		\Omega_{2,2,3} = \frac{\ii \pi^3}{12} \,,
	\]
	as claimed.
	
	The evaluation for \( \Omega_{3,1,1} \) follows in largely the same manner; we need to use some further simple identities to relate and/or evaluate some weight 2 terms in the process.  { This process is given in complete detail in the supplementary \texttt{Mathematica} file}
	
	\paragraph{Weight 4} For the weight 4 integrals, application of the parity theorem eliminates all \( \Li_{1,1,1,1} \) terms; indeed it eliminates all pure weight 4 terms except for the combination \( 2 \Li_{3,1}(1,-1) + 2 \Li_{1,3}(-1,1) \) which survives in only \( \Omega_{2,2,2,1} \).  However this remaining combination can be evaluated explicitly since 
	\begin{align*}
		2 \Li_{3,1}(1,-1) + 2 \Li_{1,3}(-1,1) = {} & 2 \Li_1(-1) \ast \Li_3(1) - 2 \Li_4(-1) \\
		= {} & -2\log(2) \zeta(3) + \frac{7}{4} \zeta(4) \,.
	\end{align*}
	This follows via the stuffle product and the known evaluation of \( \Li_4(-1) \) (via the distribution \( \Li_4(x) + \Li_4(-x) = \frac{1}{8} \Li_4(x^2) \), at \( x = 1 \), as above).
	
	To continue the evaluation in weight 4, we apply brute force in a naive way: in each case the remaining \( \ii \pi \cdot \text{weight 3} \) combination can be written as an explicit, albeit complicated, sum of shuffle products, of stuffle products, and of parity identities.  In particular, the \( \ii \pi \cdot \text{weight 3} \) combination is reduced purely to products of lower weight (i.e. weight \( \leq 2 \)) MPL's.  In each case, the resulting weight 2 MPL's can be simplified and/or evaluated via straightforward shuffle, stuffle or parity identities like those utilized for the evaluation of weight 2 and weight 3 values of \( \Omega \) given above.  { This process is given in complete detail in the supplementary \texttt{Mathematica} file.}

	\begin{remark}
			We conjecture the following formulae for the higher order coefficients of the area expansionin \ref{HOD}, based on numerical evaluation and lattice reduction of a suitably surmised set of candidates.  The identities hold numerically to 1,000 decimal places.
			\begin{align*}
			\ar_5 \overset{?}{=} & -8 \zeta(1,1,\overline{3})+\tfrac{121}{16}\zeta (5)+\tfrac{2 \pi ^2}{3} \zeta (3) -21 \zeta (3) \log ^2(2) \\
			\ar_7 \overset{?}{=} & 
			-256 \zeta (1,1,1,1,\overline{3})
			+\tfrac{1392}{17} \zeta (1,1,\overline{5})
			+\tfrac{720}{17} \zeta (1,3,\overline{3})
			+128 \log ^2(2) \zeta (1,1,\overline{3})
			\\ &
			+28 \zeta (3) \zeta (1,\overline{3}) 
			{} +\tfrac{296921}{1088}\zeta (7)
			-\tfrac{418 \pi ^2}{51} \zeta (5)
			-\tfrac{473 \pi ^4}{765} \zeta (3)
			-\tfrac{109}{2} \zeta (5) \log ^2(2)
			\\ &
			+\tfrac{280}{3} \zeta (3) \log ^4(2)
			-\tfrac{32 \pi ^2}{3} \zeta (3) \log ^2(2)
			-112 \zeta (3)^2 \log (2)
			\end{align*}
			Here \( \zeta(n_1,\ldots,n_d) \) denotes the (alternating)  multiple zeta value (MZV), defined for \( n_i \in \mathbb{Z}_{>0} \cup \overline{\mathbb{Z}_{>0}} \) (with \( \overline{\mathllap{\phantom{t}}\bullet} \) a formal decoration) via
			\[
			\zeta(n_1,\ldots,n_d) = \sum_{0 < k_1 < \cdots < k_d} \frac{(\sgn n_1)^{k_1} \cdots (\sgn n_d)^{k_d}}{k_1^{n_1} \cdots k_d^{n_d}} \,,
			\]
			where \( \sgn(n) = 1 \) and \( \sgn(\overline{n}) = -1 \), for \( n \in \mathbb{Z}_{>0} \) a positive integer.  This can also be represented as a special value of the multiple polylogarithm function at \( \pm 1 \) 
			\[
			\zeta(n_1, \ldots, n_d) = \Li_{\abs{n_1},\ldots,\abs{n_d}}(\sgn{n_1}, \ldots, \sgn{n_d}) \,,
			\]
			where \( \abs{n} = \abs{\overline{n}} = n \), for \( n \in \mathbb{Z}_{>0} \). \medskip
			
			The particular presentation of \( \ar_5 \) and \( \ar_7 \) above is in terms of the basis of the so-called MZV Data Mine \cite{mzvDM}.  Because of the plethora of relations that (alternating) MZV's satisfy, there could conceivably exist presentations which highlight much better the geometric structure and origins of \( \ar_5 \), and \( \ar_7 \).
	\end{remark}



\end{document}